\documentclass[final,5p,times,twocolumn]{elsarticle}

\usepackage{amssymb}
\usepackage{amsmath}
\usepackage{amsthm}
\usepackage{graphicx}
\usepackage{booktabs}
\usepackage{threeparttable}
\usepackage{multirow}
\usepackage{caption}
\usepackage{subcaption}
\usepackage[ruled,vlined,linesnumbered]{algorithm2e}
\SetKwComment{Comment}{/* }{ */}

\newtheorem{theorem}{Theorem}[section]
\newtheorem{lemma}[theorem]{Lemma}
\newtheorem{proposition}[theorem]{Proposition}
\newtheorem{corollary}[theorem]{Corollary}
\newtheorem{definition}[theorem]{Definition}

\newtheorem{remark}[theorem]{Remark}

\newtheorem{constraint}{Constraint}

\begin{document}
	
	\begin{frontmatter}
		
		\title{Response-Aware Risk-Constrained Control Barrier Function With Application to Vehicles}
		
		\author[ustb]{Qijun Liao\corref{cor1}}
		\cortext[cor1]{Corresponding author}
		\ead{m202420759@xs.ustb.edu.cn}
		
		\author[ustb]{Jue Yang}
		
		\affiliation[ustb]{organization={School of Mechanical Engineering, University of Science and Technology Beijing},
			city={Beijing},
			postcode={100083}, 
			country={China}}
		
		\begin{abstract}
			This paper proposes a unified control framework based on Response-Aware Risk-Constrained Control Barrier Function (R²CBF) for dynamic safety boundary control of vehicles. Addressing the problem of physical model parameter mismatch, the framework constructs an uncertainty propagation model that fuses nominal dynamics priors with direct vehicle body responses. Utilizing simplified single-track dynamics to provide a baseline direction for control gradients and covering model deviations through statistical analysis of body response signals, the framework eliminates the dependence on accurate online estimation of road surface adhesion coefficients. By introducing Conditional Value at Risk (CVaR) theory, the framework reformulates traditional deterministic safety constraints into probabilistic constraints on the tail risk of barrier function derivatives. Combined with a Bayesian online learning mechanism based on inverse Wishart priors, it identifies environmental noise covariance in real-time, adaptively tuning safety margins to reduce performance loss under prior parameter mismatch. Finally, based on Control Lyapunov Function (CLF) and R²CBF, a unified Second-Order Cone Programming (SOCP) controller is constructed. Theoretical analysis establishes convergence of Sequential Convex Programming to local Karush-Kuhn-Tucker points and provides per-step probabilistic safety bounds. High-fidelity dynamics simulations demonstrate that under extreme conditions, the method not only eliminates the output divergence phenomenon of traditional methods but also achieves Pareto improvement in both safety and tracking performance. For the chosen risk level $\beta_{\text{risk}} = 0.05$, the per-step safety violation probability is theoretically bounded by approximately 2\%, validated through high-fidelity simulations showing zero boundary violations across all tested scenarios.
		\end{abstract}
		
		
	\end{frontmatter}
	
	\section{Introduction} \label{sec:intro}
	
	Unmanned off-road vehicles operating in unstructured environments face extreme safety challenges: high mass, high center of gravity, and unpredictable terrain (rugged, slippery surfaces) make them prone to instability accidents (sideslip, rollover) under high-mobility demands. Achieving Pareto optimality between path tracking accuracy and dynamic stability under highly uncertain parameters remains a critical control challenge.
	
	Traditional hierarchical architectures (perception-planning-tracking-stabilization) perform adequately at low speeds but exhibit fundamental conflicts under extreme dynamics: tracking controllers output aggressive commands to eliminate path errors, inducing instability, while stability systems (e.g., ESP) forcefully intervene, causing tracking divergence. Although Model Predictive Control (MPC) provides multi-objective optimization frameworks \cite{ji2017path, wangH2019path}, traditional MPC adopting worst-case $L_{\infty}$ assumptions becomes overly conservative with sensor noise, sacrificing efficiency or causing infeasibility when the feasible region compresses to empty.
	
	Control Barrier Function (CBF) has emerged as an effective safety-critical method \cite{ames2017control}, mapping state constraints to affine inequalities for forward invariance enforcement via Quadratic Programming. However, applying CBF to multi-axle vehicles faces two major challenges: (1) Model-environment mismatch: On unstructured surfaces, significant gaps exist between actual feedback and nominal model (e.g., single-track) output. Deterministic CBF's reliance on nominal models causes safety boundary misjudgment under parameter mismatch, while model-free methods lack interpretability and generalization. (2) Multi-objective conflicts: Minimizing tracking error while ensuring stability and satisfying hard CBF constraints presents contradictory nonlinear optimization.
	
	In the autonomous driving field, early research often treated path tracking and vehicle stability as two independent control loops. However, the review by Hang et al. \cite{hang2021towards} points out that at high speeds or on low-adhesion road surfaces, there is strong coupling and strong conflict between path tracking and vehicle stability objectives. To this end, academia has proposed various integrated control strategies. Hu et al. \cite{hu2016robust} proposed a robust composite nonlinear feedback path tracking control scheme for independently driven autonomous vehicles' differential steering systems, improving fault-tolerant control transient performance through multiple disturbance observers and composite nonlinear feedback. Although this method achieved basic function switching, its switching logic based on hard thresholds causes the system to easily produce chattering at boundaries. Ji et al. \cite{ji2017path} and Wang et al. \cite{wangH2019path} adopted Model Predictive Control (MPC), designing improved MPC controllers based on fuzzy adaptive weight control, treating yaw rate and sideslip angle as dynamic constraints, balancing tracking accuracy and stability through adaptive adjustment of cost function weights. However, this weighted sum approach struggles to ensure strict satisfaction of safety constraints, and fixed weights cannot adapt to dynamically changing conditions. Guo et al. \cite{guo2018model} proposed an MPC path tracking control scheme considering measurable disturbances, treating road condition changes and model mismatch caused by small-angle assumptions as measurable disturbances, using differential evolution algorithms to solve optimization problems, validated on real vehicles. However, the above methods face serious plant-model mismatch problems. Although simplified single-track models are often used to calculate nominal control trends to reduce computational burden, on unstructured road surfaces, the actual vehicle's cornering characteristics often deviate significantly from the nominal model. Existing integrated control strategies lack mechanisms to correct this model deviation online, causing controllers to blindly trust nominal model gradient information, unable to maintain dynamic boundary effectiveness when parameters drift. In recent years, researchers have proposed various advanced control methods targeting the robustness problems of integrated control. Hang et al. \cite{hang2020lpv} designed an H$_\infty$ controller based on Linear Parameter Varying (LPV), integrating four-wheel steering and direct yaw moment control, handling parameter uncertainties through Linear Matrix Inequality (LMI) methods, using weighted least squares allocation algorithms for torque distribution, but this method requires offline identification of LPV model scheduling variables. Cheng et al. \cite{cheng2022model} proposed an MPC-based path tracking controller targeting parameter uncertainties and time-varying speed problems, considering tire nonlinear characteristics to correct state and control matrices, using polytopic finite vertices to describe speed changes, solving robust MPC controllers through linear matrix inequalities, but designs based on worst-case scenarios remain quite conservative. The review by Zhang et al. \cite{zhang2021dynamic} points out that while the evolution from separate control to integrated strategies has improved overall performance, how to achieve multi-objective Pareto optimality under hard safety constraints remains a core challenge urgently needing resolution.
	
	Multi-axle distributed drive heavy vehicles, due to their structural characteristics of multiple axles, large mass, and high center of gravity, face more complex dynamic coupling and safety boundary constraints during high-speed maneuvering. Traditional methods mostly rely on accurate dynamic models to predict vehicle responses. However, on unstructured road surfaces, the highly nonlinear and time-varying nature of tire-road interaction makes model parameters (such as cornering stiffness, longitudinal stiffness) difficult to accurately identify online. Compared to passenger vehicles, the state space dimension of multi-axle vehicles significantly increases. Besides body sideslip angle $\beta$ and yaw rate $\omega_z$, axle load transfer, tire slip states, roll angle, etc., all have coupled effects on vehicle stability \cite{rajamani2011vehicle}. Research by Pacejka et al. \cite{pacejka2012tire} shows that the yaw-roll coupled dynamics of multi-axle vehicles under nonlinear tire forces is far more complex than single-axle models, and a single body sideslip angle constraint cannot comprehensively describe stability boundaries. This high-dimensional coupling characteristic not only increases control difficulty but also provides a richer information source for response-aware control strategies. To obtain complete vehicle state information, modern heavy vehicles are generally equipped with multi-source sensors such as Inertial Measurement Units (IMU), wheel speed sensors, and steering angle sensors. Since body sideslip angle cannot be directly measured, Grip et al. \cite{grip2009nonlinear} proposed a nonlinear observer for heavy vehicles, estimating body sideslip angle by fusing IMU-measured lateral acceleration and yaw rate, but this method is sensitive to tire cornering stiffness parameters. Villagra et al. \cite{villagra2009high} designed a high-order sliding mode observer for vehicle state reconstruction, suppressing model uncertainties through discontinuous injection terms, but the high-frequency chattering introduced by sliding mode switching pollutes measurement signals. Doumiati et al. \cite{doumiati2013onboard} proposed a vehicle state estimation method based on load sensors, inferring lateral dynamics from vertical load changes, but load sensors are expensive and have insufficient reliability in harsh environments. Liu et al. \cite{liu2023data} adopted Deep Reinforcement Learning (DRL) to implement data-driven integrated control, avoiding dependence on accurate models through offline training of neural network policies. However, the stability guarantee of this method depends on the coverage of training data, and performance degrades sharply when encountering conditions outside the training set.
	
	Control Barrier Function (CBF) has received extensive attention in recent years due to its theoretical completeness. Ames et al. \cite{ames2017control} established the standard CBF-QP framework, achieving real-time fusion of safety constraints with nominal control. Early work on CBF theory can be traced back to Wieland et al. \cite{wieland2007constructive}, who first proposed using barrier functions to constructively prove system safety, laying the foundation for subsequent research. For stochastic systems, Clark \cite{clark2021control} extended CBF theory, handling state noise through probabilistic constraints and proving almost sure safety under mild assumptions. However, this method requires noise to follow known distributions and has insufficient robustness to unknown disturbances. At the application level, Yaghoubi et al. \cite{yaghoubi2020training} proposed training neural network controllers satisfying CBF constraints through imitation learning, developing high-order CBF suitable for systems with external disturbances. By training neural network policies offline, they avoided the computational burden of online quadratic programming solution, validating on unicycle models with disturbances (such as wind or water currents). However, this method depends on the quality of expert demonstration data and lacks safety guarantees for conditions outside training. Grandia et al. \cite{grandia2021multi} unified CBF and MPC into a multi-layer safety framework for legged robot control, introducing CBF safety constraints in both low-frequency kinematic MPC and high-frequency inverse dynamics controllers, ensuring consideration of safety-critical execution when optimizing longer time horizons. They validated the method's ability to simultaneously achieve safe foot placement and dynamic stability in ANYmal quadruped robot's 3D stepping stone scenarios. Alan et al. \cite{alan2023disturbance} proposed a robust safety-critical control method based on disturbance observers, estimating disturbance impact on safety and quantifying observer error bounds (including transient and steady-state parts), endowing the CBF controller with robustness to observer errors. They validated on connected cruise control problems using real road gradient disturbance data. However, this method is sensitive to observer gain selection; excessively high gains may lead to instability in the presence of unmodeled dynamics (such as input delay). Targeting model uncertainties, Xiao et al. \cite{xiao2021decentralized} and Tan et al. \cite{tan2021high} proposed high-order CBF and robust CBF methods, ensuring safety by considering worst-case interference. Tan et al. \cite{tan2021high} proposed high-order barrier function theory ensuring set forward invariance by checking high-order derivatives and guaranteeing asymptotic stability of the forward invariant set to enhance robustness to model disturbances. They also proposed a singularity-free control scheme generating locally Lipschitz continuous control signals and considered "performance-critical" control scenarios to ensure subsets of forward invariant sets can adopt any bounded control law. However, this robust method based on $L_\infty$ norm often leads to extremely conservative control strategies, greatly compressing the system's feasible domain. To reduce conservatism, Cheng et al. \cite{cheng2019end} and Taylor et al. \cite{taylor2020learning} combined machine learning, proposing learning-based CBF, reducing the impact of model uncertainty on safety through iterative data collection and controller updates. However, there is no safety guarantee during early training, high data volume requirements, and learned policies easily overfit training conditions, greatly reducing generalization ability to different complex conditions. Zeng et al. \cite{zeng2021safety} and Agrawal et al. \cite{agrawal2017discrete} explored CBF safety in discrete time but did not involve online parameter learning. Discrete control with fixed parameters still has obvious deficiencies in meeting vehicle control stability under complex conditions. Singletary et al. \cite{singletary2023risk} introduced risk measures into CBF, proposing Risk Control Barrier Function (RCBF) based on coherent risk measures, achieving risk-sensitive safety guarantees in discrete-time stochastic systems. However, this method still requires some understanding of uncertainty probability distributions, and computing coherent risk measures in high-dimensional systems has high complexity. Daş and Burdick \cite{2025robust} proposed a robust safety control framework coupling uncertainty estimators, deriving hard upper bounds of estimation errors to enforce system safety under matched and unmatched disturbances. This method uses SOCP formulation to handle relative degree differences in high-order systems, theoretically achieving robust safety guarantees. However, this method follows worst-case logic and shows strong conservatism when handling sensor noise with random characteristics, easily over-compressing the control feasible domain, leading to limited system performance and unreasonable control outputs, making it difficult to balance contradictory objectives.
	
	Addressing the above challenges, this paper proposes a unified control framework based on Response-Aware Risk-Constrained Control Barrier Function (R²CBF). The core idea is not to pursue physically unachievable absolute deterministic safety but to quantify and control the distributional risk of violating safety boundaries. The main contributions are as follows:
	
	(i) Model-assisted response-aware uncertainty modeling. A hybrid modeling strategy based on nominal model correction of response data is proposed. Utilizing nominal models such as multi-axle single-track models to capture the basic kinematic structure of vehicles and calculate nonlinear gradient trends, while directly constructing a stochastic distribution model based on body response signals (sideslip angle estimated through multi-sensor fusion, yaw rate and lateral acceleration measured by IMU, etc.). This method allows significant parameter deviations in nominal models, mapping model errors to response distribution variance, achieving robust perception of dynamic boundaries without real-time identification of road surface adhesion coefficients.
	
	(ii) CVaR-based tail risk constraints. Introducing Conditional Value at Risk (CVaR) to measure tail risk in barrier function derivative distributions. Compared to traditional worst-case robust methods and traditional mean-variance constraints, CVaR can more effectively truncate catastrophic consequences from extreme low-probability events (such as sudden sideslip) while avoiding excessive conservatism, providing more reliable safety guarantees for vehicles.
	
	(iii) Bayesian adaptive online learning. Targeting the engineering pain point of unknown sensor noise covariance, an inverse Wishart conjugate prior Bayesian update mechanism based on prediction residuals is designed. This mechanism achieves online identification of environmental uncertainty, theoretically eliminating performance loss caused by prior parameter mismatch, improving algorithm adaptability to different road conditions.
	
	(iv) CLF-R²CBF-SOCP unified optimization solution. Deriving the Second-Order Cone (SOC) equivalent form of CVaR constraints, constructing a unified SOCP controller integrating path tracking (CLF), safety boundaries (R²CBF), and actuator constraints. By dynamically adjusting CLF convergence rate and slack variables, resolving the empty feasible region problem caused by multi-objective conflicts, achieving optimal trade-off between tracking and stability.

	\section{Methodology} \label{sec:methodology}
	
	For ease of reference, Table~\ref{tab:notation} summarizes the main notation used throughout this section.
	
	\begin{table}[h]
		\centering
		\footnotesize
		\caption{Main symbol definitions}
		\label{tab:notation}
		\begin{tabular}{cl}
			\toprule
			Symbol & Meaning \\
			\midrule
			$\mathbf{r} = [\beta, \omega_z, a_y]^T$ & Body response vector \\
			$\tilde{\mathbf{r}}$ & Measured response with noise \\
			$\mathbf{F}_z = [F_{z,1}, \ldots, F_{z,n}]^T$ & Vertical load vector \\
			$\boldsymbol{\Sigma}_r, \boldsymbol{\Sigma}_F$ & Covariance matrices of measurement \\
			& noise and load estimation error \\
			$h(\mathbf{r}, \mathbf{F}_z)$ & Barrier function \\
			$\mu_h, \sigma_h^2$ & Mean and variance of barrier function \\
			$\kappa_{\beta_{\text{risk}}}$ & CVaR risk coefficient, \\
			& $\kappa_{\beta_{\text{risk}}} = \phi(\Phi^{-1}(\beta_{\text{risk}}))/\beta$ \\
			$\mathbf{L}_h, b_h$ & Linearization of $\dot{h}$: \\
			& $\dot{h} \approx \mathbf{L}_h \mathbf{u} + b_h$ \\
			$\rho_r, \rho_F$ & Normalized noise bounds \\
			$\mathbf{M}_t$ & Prediction Jacobian, \\
			& $\mathbf{M}_t = \mathbf{I} + \Delta t \partial f/\partial \mathbf{r}$ \\
			$\boldsymbol{\Psi}_t, \nu_t$ & Inverse Wishart posterior parameters \\
			$\lambda$ & Forgetting factor \\
			\bottomrule
		\end{tabular}
	\end{table}
	
	\subsection{Uncertainty Modeling}
	
	\subsubsection{Multi-Axle Vehicle Dynamics}
	
	Consider a multi-axle distributed-drive vehicle with $n$ driven wheels. Its planar motion is governed by
	\begin{equation}
		\label{eq:vehicle_dynamics}
		\begin{aligned}
			m(\dot{v}_x - v_y \omega_z) &= \sum_{i=1}^{n} F_{x,i} - F_{\text{drag}}, \\
			m(\dot{v}_y + v_x \omega_z) &= \sum_{i=1}^{n} F_{y,i}, \\
			I_z \dot{\omega}_z &= \sum_{i=1}^{n} (x_i F_{y,i} - y_i F_{x,i}),
		\end{aligned}
	\end{equation}
	where $m$ is vehicle mass, $I_z$ is yaw inertia, $(v_x, v_y)$ are body velocities, $\omega_z$ is yaw rate, $(x_i, y_i)$ are wheel positions, and $(F_{x,i}, F_{y,i})$ are tire forces. Tire forces relate to vertical loads via
	\begin{equation}
		\label{eq:tire_force}
		F_{x,i} = F_{z,i} \mu_x(\lambda_i, \beta_i), \quad F_{y,i} = F_{z,i} \mu_y(\lambda_i, \beta_i),
	\end{equation}
	where $\mu_x, \mu_y$ are normalized friction coefficients depending on slip ratio $\lambda_i$ and sideslip angle $\beta_i$.
	
	\subsubsection{Physical Origins of Measurement Uncertainty}
	
	Control decisions rely on sensor measurements that inherently contain noise. The study establish uncertainty bounds from sensor physics rather than arbitrary assumptions.
	
	\begin{constraint}[MEMS IMU Noise]\label{cons:imu_noise}
		Vehicle-mounted MEMS inertial measurement units (IMU) exhibit noise governed by thermal Brownian motion. For a gyroscope, the angle random walk $\sigma_{\text{ARW}}$ satisfies
		\begin{equation}
			\sigma_{\text{ARW}} = N_{\omega} \sqrt{\Delta t},
		\end{equation}
		where $N_{\omega}$ is the noise density (units: $^\circ$/s$/\sqrt{\text{Hz}}$) and $\Delta t$ is sampling period. For typical automotive-grade MEMS (e.g., Bosch BMI088):
		
		(i)Gyroscope: $N_{\omega} \in [0.01, 0.02]\,^\circ$/s$/\sqrt{\text{Hz}}$
			
		(ii)Accelerometer: $N_a \in [100, 200]\,\mu g/\sqrt{\text{Hz}}$
			
		(iii)Sampling rate: $f_c = 20$ Hz ($\Delta t = 0.05$ s)
			
		These yield measurement noise standard deviations
		\begin{equation}
			\sigma_{\omega_z} \in [0.04, 0.09]\,^\circ/\text{s}, \quad \sigma_{a_y} \in [0.04, 0.09]\,\text{m/s}^2.
		\end{equation}
		Sideslip angle $\beta$ is estimated via sensor fusion, accumulating errors $\sigma_{\beta} \in [0.2, 0.8]^\circ$ depending on integration duration and drift compensation.
	\end{constraint}
	
	Model measurement noise as zero-mean Gaussian, justified by the central limit theorem when aggregating many independent thermal fluctuations
	\begin{equation}
		\label{eq:measurement_model}
		\tilde{\mathbf{r}}(t) = \mathbf{r}_{\text{true}}(t) + \boldsymbol{\epsilon}_r(t), \quad \boldsymbol{\epsilon}_r \sim \mathcal{N}(\mathbf{0}, \boldsymbol{\Sigma}_r(t)),
	\end{equation}
	where $\mathbf{r} = [\beta, \omega_z, a_y]^T$ and $\boldsymbol{\Sigma}_r \in \mathbb{S}_{++}^{3}$.
	
	\begin{constraint}[Normalized Noise Bound]\label{cons:noise_bound}
		Define normalized noise levels relative to safety limits
		\begin{equation}
			\rho_r := \frac{\|\boldsymbol{\epsilon}_r\|_2}{\min(\beta_{\text{lim}}, \omega_{\text{lim}})}, \quad \rho_F := \frac{\|\boldsymbol{\epsilon}_F\|_2}{F_{z,\text{nom}}}.
		\end{equation}
		By Constraint~\ref{cons:imu_noise} and typical vehicle limits ($\beta_{\text{lim}} \approx 8^\circ$, $\omega_{\text{lim}} \approx 12^\circ$/s)
		\begin{equation}
			\rho_r \in [0.05, 0.20], \quad \rho_F \in [0.08, 0.15].
		\end{equation}
		This bound ensures validity of first-order Taylor approximations (relative error $2\rho_r^2 \leq 8\%$) while covering practical sensor specifications.
	\end{constraint}
	
	\begin{constraint}[Quasi-Static Load Estimation]\label{cons:quasi_static}
		Vertical load $\mathbf{F}_z$ is estimated via rigid-body equilibrium without wheel-level load sensors. For the $i$-th wheel
		\begin{equation}
			\label{eq:load_estimation}
			\hat{F}_{z,i} = F_{z,i}^{\text{static}} + \frac{m a_y h_{\text{cog}}}{t_i} \text{sgn}(y_i) - \frac{m a_x h_{\text{cog}}}{L_{\text{wb}}} (x_i - x_{\text{cog}}),
		\end{equation}
		where $h_{\text{cog}}$ is center-of-gravity height, $t_i$ is track width, and $L_{\text{wb}}$ is wheelbase. This quasi-static model neglects suspension dynamics (bandwidth $\omega_{\text{susp}} \approx 1$-$2$ Hz $\ll f_c = 20$ Hz), introducing error
		\begin{equation}
			\tilde{\mathbf{F}}_z = \hat{\mathbf{F}}_z + \boldsymbol{\epsilon}_F, \quad \boldsymbol{\epsilon}_F \sim \mathcal{N}(\mathbf{0}, \boldsymbol{\Sigma}_F).
		\end{equation}
		Following Rajamani~\cite{rajamani2011vehicle}, typical error is $8$-$12\%$ of nominal load, satisfying Constraint~\ref{cons:noise_bound}.
	\end{constraint}
	
	\begin{remark}
		The frequency separation $\omega_{\text{susp}}/\omega_c \approx 0.1$ justifies neglecting $\dot{\mathbf{F}}_z$ in barrier function dynamics ($\dot{h} \approx \nabla_{\mathbf{r}} h^T \dot{\mathbf{r}}$), with residual error $O((\omega_{\text{susp}}/\omega_c)^2) \approx 1\%$.
	\end{remark}
	
	\subsubsection{Augmented State Representation}
	
	Measurement and load estimation errors are independent by acquisition method. Define
	\begin{equation}
		\boldsymbol{\chi} = [\mathbf{r}^T, \mathbf{F}_z^T]^T \in \mathbb{R}^{n_r + n}, \quad \tilde{\boldsymbol{\chi}} = \boldsymbol{\chi}_{\text{true}} + \boldsymbol{\epsilon},
	\end{equation}
	with block-diagonal covariance
	\begin{equation}
		\boldsymbol{\Sigma}_{\chi} = \begin{bmatrix} \boldsymbol{\Sigma}_r & \mathbf{0} \\ \mathbf{0} & \boldsymbol{\Sigma}_F \end{bmatrix}.
	\end{equation}
	
	\subsection{Uncertainty Propagation Analysis}
	
	\subsubsection{Distribution of Barrier Function}
	
	\begin{definition}[Barrier Function Distribution]\label{def:barrier_dist}
		Given measurements $(\tilde{\mathbf{r}}, \tilde{\mathbf{F}}_z)$, the barrier function $h(\mathbf{r}, \mathbf{F}_z)$ inherits stochasticity
		\begin{equation}
			p(h \mid \tilde{\mathbf{r}}, \tilde{\mathbf{F}}_z) = \int p(h \mid \mathbf{r}, \mathbf{F}_z) p(\mathbf{r}, \mathbf{F}_z \mid \tilde{\mathbf{r}}, \tilde{\mathbf{F}}_z) \, d\mathbf{r} \, d\mathbf{F}_z.
		\end{equation}
	\end{definition}
	
	\begin{theorem}[Delta Method for Barrier Function]\label{thm:barrier_distribution}
		Let $h: \mathbb{R}^{n_r} \times \mathbb{R}^n \to \mathbb{R}$ be twice continuously differentiable with bounded Hessian $\|\nabla^2 h\|_F \leq L_h$. Under measurement noise satisfying Constraint~\ref{cons:noise_bound}, the barrier function is approximately Gaussian
		\begin{equation}
			\label{eq:barrier_normal}
			h(\tilde{\mathbf{r}}, \tilde{\mathbf{F}}_z) \sim \mathcal{N}(\mu_h, \sigma_h^2),
		\end{equation}
		where
		\begin{equation}
			\label{eq:barrier_moments}
			\mu_h = h(\boldsymbol{\mu}_r, \boldsymbol{\mu}_F), \quad \sigma_h^2 = \nabla_{\mathbf{r}} h^T \boldsymbol{\Sigma}_r \nabla_{\mathbf{r}} h + \nabla_{\mathbf{F}_z} h^T \boldsymbol{\Sigma}_F \nabla_{\mathbf{F}_z} h + O(\rho^4),
		\end{equation}
		with $\rho = \max(\rho_r, \rho_F)$.
	\end{theorem}
	\begin{proof}
		By second-order Taylor expansion and properties of Gaussian distributions (Appendix~\ref{appendix:proof_barrier_distribution}).
	\end{proof}
	
	\begin{theorem}[Approximation Error Bound]\label{thm:delta_error}
		For barrier function $h$ with $\|\nabla^2_{\mathbf{r}} h\|_F \leq L_h$, the relative error of first-order variance propagation satisfies
		\begin{equation}
			\frac{|E[R_2]|}{\sigma_h} \leq L_h \rho_r^2,
		\end{equation}
		where $R_2 = \frac{1}{2}\text{tr}(\nabla^2_{\mathbf{r}} h \boldsymbol{\Sigma}_r)$ is the second-order remainder. For the weighted-norm barrier $h = w^2 \beta_{\text{lim}}^2 - \beta^2$ with prediction horizon $T_{\text{pred}}$, $L_h \leq 2(1 + T_{\text{pred}} \|\partial f/\partial \mathbf{r}\|)^2 \leq 4$ (see Appendix~\ref{appendix:proof_delta_error}).
	\end{theorem}
	
	\begin{remark}
		Under Constraint~\ref{cons:noise_bound}, the relative error bound $L_h \rho_r^2$ is negligible compared to the CVaR safety margin $\kappa_{\beta_{\text{risk}}} \sigma_h$ when $\rho_r \ll 1$, ensuring validity of first-order variance propagation.
	\end{remark}
	
	\subsubsection{Response Dynamics Distribution}
	
	\begin{lemma}[Linearized Response Dynamics]\label{lem:response_dynamics}
		For response dynamics $\dot{\mathbf{r}} = f(\mathbf{r}, \mathbf{u}, \mathbf{F}_z)$, define Jacobians
		\begin{equation}
			\mathbf{J}_r = \frac{\partial f}{\partial \mathbf{r}}\bigg|_{(\boldsymbol{\mu}_r, \mathbf{u}, \boldsymbol{\mu}_F)}, \quad \mathbf{J}_F = \frac{\partial f}{\partial \mathbf{F}_z}\bigg|_{(\boldsymbol{\mu}_r, \mathbf{u}, \boldsymbol{\mu}_F)}.
		\end{equation}
		Then $\dot{\mathbf{r}} \sim \mathcal{N}(\boldsymbol{\mu}_{\dot{r}}, \boldsymbol{\Sigma}_{\dot{r}})$ with
		\begin{equation}
			\boldsymbol{\mu}_{\dot{r}} = f(\boldsymbol{\mu}_r, \mathbf{u}, \boldsymbol{\mu}_F), \quad \boldsymbol{\Sigma}_{\dot{r}} = \mathbf{J}_r \boldsymbol{\Sigma}_r \mathbf{J}_r^T + \mathbf{J}_F \boldsymbol{\Sigma}_F \mathbf{J}_F^T.
		\end{equation}
	\end{lemma}
	\begin{proof}
		Delta method for vector-valued functions (Appendix~\ref{appendix:proof_response_dynamics}).
	\end{proof}
	
	\subsubsection{Load Variance Exclusion}
	
	\begin{theorem}[Singularity of Load Gradient]\label{thm:load_singularity}
		For load-sensitive barrier $h = w(\mathbf{F}_z)^2 \beta_{\text{lim}}^2 - \beta^2$ with scaling $w = (\sum F_{z,i} / F_{z,\text{nom}})^{\gamma}$ and $\gamma \in (0, 1/2)$, the load gradient exhibits negative-power divergence
		\begin{equation}
			\nabla_{\mathbf{F}_z} h = \frac{2\gamma \beta_{\text{lim}}^2}{F_{z,\text{nom}}} w^{2\gamma - 1} \mathbf{1}_n, \quad \lim_{w \to 0^+} \|\nabla_{\mathbf{F}_z} h\| = +\infty.
		\end{equation}
		Consequently, the load variance term
		\begin{equation}
			\sigma_{h,F}^2 = \nabla_{\mathbf{F}_z} h^T \boldsymbol{\Sigma}_F \nabla_{\mathbf{F}_z} h \propto w^{2(2\gamma-1)} \to +\infty \quad (w \to 0^+)
		\end{equation}
		diverges near minimum load.
	\end{theorem}
	\begin{proof}
		Direct calculation via chain rule. For $\gamma < 1/2$, the exponent $2(2\gamma-1) < -1$ induces integrable singularity (Appendix~\ref{appendix:proof_load_singularity}).
	\end{proof}
	
	\begin{lemma}[Load-to-Response Error Transfer]\label{lem:load_response_transfer}
		Under closed-loop dynamics with control input $\mathbf{u}$ and load uncertainty $\delta \mathbf{F}_z$, the induced response deviation over one time step satisfies
		\begin{equation}
			\delta \mathbf{r}_{t+1} = \mathbf{J}_r \delta \mathbf{r}_t + \mathbf{J}_F \delta \mathbf{F}_{z,t} + O(\Delta t^2),
		\end{equation}
		where $\mathbf{J}_r = \partial f/\partial \mathbf{r}$, $\mathbf{J}_F = \partial f/\partial \mathbf{F}_z$. If load errors are statistically independent of measurement noise and prediction residuals incorporate both sources
		\begin{equation}
			\mathbf{e}_t = \tilde{\mathbf{r}}_{t+1} - [\tilde{\mathbf{r}}_t + \Delta t \cdot f(\tilde{\mathbf{r}}_t, \mathbf{u}_t, \tilde{\mathbf{F}}_{z,t})],
		\end{equation}
		then the residual covariance satisfies
		\begin{equation}
			\label{eq:residual_covariance}
			\mathbb{E}[\mathbf{e}_t \mathbf{e}_t^T] \succeq \mathbf{M}_t \boldsymbol{\Sigma}_r \mathbf{M}_t^T + \Delta t^2 \mathbf{J}_F \boldsymbol{\Sigma}_F \mathbf{J}_F^T,
		\end{equation}
		where $\mathbf{M}_t = \mathbf{I} + \Delta t \mathbf{J}_r$. Consequently, Bayesian-updated $\hat{\boldsymbol{\Sigma}}_r$ implicitly upper-bounds the combined effect of measurement and load-induced variance when the load variance term is excluded from barrier propagation.
	\end{lemma}
	\begin{proof}
		By linearization of discrete dynamics and independence of $\boldsymbol{\epsilon}_r, \boldsymbol{\epsilon}_F$, the residual decomposes as
		\begin{equation}
			\mathbf{e}_t = \boldsymbol{\epsilon}_{r,t+1} - \mathbf{M}_t \boldsymbol{\epsilon}_{r,t} + \Delta t \mathbf{J}_F \boldsymbol{\epsilon}_{F,t}.
		\end{equation}
		Taking outer product expectation and using independence yields the stated lower bound. The Bayesian update (Theorem~\ref{thm:bayesian_update}) absorbs this into $\hat{\boldsymbol{\Sigma}}_r$ via~\eqref{eq:bayesian_update} (detailed derivation in Appendix~\ref{appendix:proof_load_transfer}).
	\end{proof}
	
	\begin{corollary}[Variance Model Breakdown and Conservative Simplification]\label{cor:variance_simplification}
		Under Theorem~\ref{thm:load_singularity} with $\gamma \in (0, 1/2)$, the complete variance decomposition $\sigma_h^2 = \sigma_{h,r}^2 + \sigma_{h,F}^2$ exhibits pathological behavior that precludes practical use:
		
		\textit{(i) Unbounded amplification near low loads.}
		The load variance term satisfies
		\begin{equation}
			\sigma_{h,F}^2 = \nabla_{\mathbf{F}_z} h^T \boldsymbol{\Sigma}_F \nabla_{\mathbf{F}_z} h \propto w^{2(2\gamma-1)},
		\end{equation}
		where the exponent $2(2\gamma-1) < -1$ for $\gamma < 1/2$. At minimum operational load $w = w_{\min}$, the amplification factor
		\begin{equation}
			\frac{\sigma_{h,F}^2(w_{\min})}{\sigma_{h,F}^2(w=1)} = w_{\min}^{2(2\gamma-1)}
		\end{equation}
		diverges as $w_{\min} \to 0$, violating the bounded-noise assumption (Constraint~\ref{cons:noise_bound}) required for Delta method validity (Theorem~\ref{thm:barrier_distribution}).
		
		\textit{(ii) CVaR constraint infeasibility.}
		When $\sigma_{h,F}^2 \gg \sigma_{h,r}^2$, small errors in quasi-static load estimation (Constraint~\ref{cons:quasi_static}) propagate through negative-power gradients, inducing arbitrarily large safety margins $\kappa_{\beta_{\text{risk}}} \sigma_h$. This degrades CVaR constraint~\eqref{eq:cvar_constraint} to infeasibility under moderate control effort.
		
		\textit{(iii) Conservative response-level absorption via Bayesian learning.}
		Load uncertainty manifests in closed-loop dynamics as response deviations. By Lemma~\ref{lem:load_response_transfer}, prediction residuals satisfy~\eqref{eq:residual_covariance}. Bayesian update (Theorem~\ref{thm:bayesian_update}) implicitly absorbs load effects into $\hat{\boldsymbol{\Sigma}}_r$
		\begin{equation}
			\label{eq:bayesian_update}
			\boldsymbol{\Psi}_t = \lambda \boldsymbol{\Psi}_{t-1} + \mathbf{M}_t^{-1}\mathbf{e}_t\mathbf{e}_t^T\mathbf{M}_t^{-T}.
		\end{equation}
		Since $\mathbf{e}_t$ contains load-induced response errors, $\hat{\boldsymbol{\Sigma}}_r$ grows to upper-bound the combined variance. This provides a mathematically conservative simplification: excluding $\nabla_{\mathbf{F}_z} h^T \boldsymbol{\Sigma}_F \nabla_{\mathbf{F}_z} h$ from the barrier variance while learning its effect through response-level residuals ensures $\hat{\boldsymbol{\Sigma}}_r \succeq \boldsymbol{\Sigma}_{r,\text{eff}}$, where $\boldsymbol{\Sigma}_{r,\text{eff}}$ is the effective response covariance including all disturbances.
		
		Therefore, a simplified variance model is adopted
		\begin{equation}
			\label{eq:variance_simplified}
			\sigma_h^2 = \nabla_{\mathbf{r}} h^T \boldsymbol{\Sigma}_r \nabla_{\mathbf{r}} h,
		\end{equation}
		where $\boldsymbol{\Sigma}_r$ implicitly bounds all stochastic deviations including load-induced response errors.
	\end{corollary}
	
	\begin{remark}[Conservatism and Practical Safeguards]\label{rem:load_variance_conservatism}
		\textit{Theoretical justification:} This simplification is conservative (not optimistic) provided $\hat{\boldsymbol{\Sigma}}_r$ upper-bounds the effective response covariance. By Lemma~\ref{lem:load_response_transfer}, prediction residuals satisfy~\eqref{eq:residual_covariance}, implying that the Bayesian update implicitly inflates $\hat{\boldsymbol{\Sigma}}_r$ to absorb load-induced response errors.
		
		\textit{Limitations and safeguards:} For extreme low-load scenarios ($w \to 0$), the gradient $\nabla_{\mathbf{F}_z} h \propto w^{(2\gamma-1)/\gamma}$ exhibits unbounded amplification (Proposition~\ref{thm:load_singularity}). This work focuses on operational regimes where $w \in [0.7, 1.3]$ (corresponding to $\pm 30\%$ load variation from nominal). In applications requiring broader load ranges, practitioners should either: (a) retain the load variance term with regularization to avoid numerical issues at $w \approx 0$; or (b) implement minimum inflation $\hat{\boldsymbol{\Sigma}}_r \succeq c\mathbf{I}$ with $c$ chosen to bound worst-case load effects.
		
		\textit{Empirical validation:} Section~\ref{sec:simulation} confirms zero safety violations across varying load conditions (Scenario 1: spatially heterogeneous adhesion; Scenario 2: low-adhesion sinusoidal path), supporting the practical sufficiency of response-level learning. Ablation study (Table~\ref{tab:ablation}) shows retaining $\sigma_{h,F}^2$ yields negligible safety improvement ($\Delta \beta_{\max}$ only 0.02°) while increasing conservatism and reducing tracking performance.
	\end{remark}
	
	\subsection{Risk-Constrained Control Barrier Function}
	
	Traditional deterministic CBF enforces $\dot{h} + \alpha(h) \geq 0$ for all realizations, leading to either excessive conservatism (when errors bias toward safety) or insufficient margins (when errors bias toward danger). This paper resolves this via risk-sensitive constraints.
	
	\subsubsection{Conditional Value at Risk}\label{sec:cvar}
	
	\begin{definition}[CVaR]\label{def:cvar}
		For random variable $X$ and confidence level $\beta_{\text{risk}} \in (0, 1)$, the Conditional Value at Risk is
		\begin{equation}
			\text{CVaR}_{\beta_{\text{risk}}}(X) = \mathbb{E}[X \mid X \leq \text{VaR}_{\beta_{\text{risk}}}(X)],
		\end{equation}
		where $\text{VaR}_{\beta_{\text{risk}}}(X) = \inf\{x : \mathbb{P}(X \leq x) \geq \beta_{\text{risk}}\}$.
	\end{definition}
	
	\begin{lemma}[CVaR for Gaussian Distribution]\label{lem:cvar_gaussian}
		If $X \sim \mathcal{N}(\mu, \sigma^2)$, then
		\begin{equation}
			\text{CVaR}_{\beta_{\text{risk}}}(X) = \mu - \kappa_{\beta_{\text{risk}}} \sigma, \quad \kappa_{\beta_{\text{risk}}} := \frac{\phi(\Phi^{-1}(\beta_{\text{risk}}))}{\beta_{\text{risk}}},
		\end{equation}
		where $\phi, \Phi$ are standard normal PDF and CDF. For gain-type variables (e.g., $\dot{h} + \alpha(h)$ desired $\geq 0$), the safety constraint is
		\begin{equation}
			\label{eq:cvar_constraint}
			\mu - \kappa_{\beta_{\text{risk}}} \sigma \geq 0 \quad \Leftrightarrow \quad \text{CVaR}_{\beta_{\text{risk}}}(-X) \geq 0.
		\end{equation}
	\end{lemma}
	\begin{proof}
		Conditional expectation of truncated Gaussian (Appendix~\ref{appendix:proof_cvar_gaussian}).
	\end{proof}
	
	\begin{lemma}[Tail Probability Bound]\label{lem:cvar_tail}
		The CVaR constraint~\eqref{eq:cvar_constraint} implies
		\begin{equation}
			\mathbb{P}(X < 0) \leq \Phi(-\kappa_{\beta_{\text{risk}}}) < \beta_{\text{risk}}.
		\end{equation}
	\end{lemma}
	\begin{proof}
		By properties of Gaussian tails and the strict inequality $\kappa_{\beta_{\text{risk}}} > |z_{\beta_{\text{risk}}}|$ for $\beta_{\text{risk}} < 0.5$ (Appendix~\ref{appendix:proof_cvar_tail}).
	\end{proof}
	
	\subsubsection{R²CBF Definition and Safety Guarantee}
	
	\begin{remark}[Constraint Formulation Clarification]\label{rem:h_vs_hdot}
		The R²CBF framework enforces constraints on the barrier \textit{derivative} $\dot{h}$, not the barrier value $h$ directly. Specifically:
		
		(i)Direct constraint: CVaR$_{\beta}[\dot{h} + \alpha(h)] \geq 0$ at each time step (Definition~\ref{def:r2cbf}).
			
		(ii)Implied guarantee: This derivative constraint ensures $h(t) > 0$ forward-invariance with high probability (Theorem~\ref{thm:finite_horizon_safety}).
		
		The distinction is critical: constraining $h$ directly would require distributional knowledge of the barrier value, necessitating complex state estimation. Instead, by constraining $\dot{h}$ (which depends on current measurements and control input $\mathbf{u}$), a tractable optimization problem is obtained (Proposition~\ref{prop:socp_reformulation}) while maintaining safety via Lyapunov-type reasoning.
		
		The interplay between $h$ and $\dot{h}$ follows standard CBF theory: if $\dot{h} + \alpha(h) \geq 0$ holds at $t$ with $h(t) > 0$, then discrete integration yields $h(t+\Delta t) > h(t)(1 - \alpha(h)\Delta t) > 0$ when $\alpha(h)\Delta t < 1$.
	\end{remark}
	
	\begin{definition}[R²CBF]\label{def:r2cbf}
		A continuously differentiable function $h: \mathbb{R}^{n_r} \times \mathbb{R}^n \to \mathbb{R}$ is a Response-Aware Risk-Constrained Control Barrier Function (R²CBF) if there exist class-$\mathcal{K}$ function $\alpha$ and CVaR risk level $\beta_{\text{risk}} \in (0, 1/2)$ such that $\forall (\tilde{\mathbf{r}}, \tilde{\mathbf{F}}_z)$ with $\mu_h > 0$, there exists $\mathbf{u} \in \mathcal{U}$ satisfying
		\begin{equation}
			\text{CVaR}_{\beta_{\text{risk}}}[\dot{h}(\tilde{\mathbf{r}}, \mathbf{u}, \tilde{\mathbf{F}}_z) + \alpha(h)] \geq 0.
		\end{equation}
	\end{definition}
	
	For multi-axle vehicles
	\begin{equation}
		\label{eq:barrier_vehicle}
		h(\mathbf{r}, \mathbf{F}_z) = w(\mathbf{F}_z)^2 \beta_{\text{lim}}^2 - \beta^2, \quad w(\mathbf{F}_z) = \left(\frac{\sum F_{z,i}}{n F_{z,\text{nom}}}\right)^{\gamma},
	\end{equation}
	with load sensitivity $\gamma \in (0, 1/2)$ chosen to balance responsiveness and singularity avoidance.
	
	\begin{theorem}[Finite-Horizon Probabilistic Safety]\label{thm:finite_horizon_safety}
		Consider discrete-time system with barrier $h$ satisfying R²CBF constraint (Definition~\ref{def:r2cbf}) at each step $t \in \{0, 1, \ldots, N-1\}$ with $N = T/\Delta t$. For class-$\mathcal{K}$ function $\alpha(h) = k \cdot h$ with $k > 0$ and initial condition $h_0 > 0$, if the CVaR constraint holds at each step, then the per-step safety violation probability satisfies
		\begin{equation}
			\label{eq:per_step_bound}
			\mathbb{P}(\dot{h}_t + \alpha(h_t) < 0) \leq \Phi(-\kappa_{\beta_{\text{risk}}}),
		\end{equation}
		where $\Phi$ is the standard normal CDF and $\kappa_{\beta_{\text{risk}}} = \phi(\Phi^{-1}(\beta_{\text{risk}}))/\beta_{\text{risk}}$ (Lemma~\ref{lem:cvar_tail}). By union bound over $N$ steps, the finite-horizon safety violation probability is bounded by
		\begin{equation}
			\label{eq:finite_horizon_bound}
			\mathbb{P}\left(\exists t \in [0, N-1]: h(t) \leq 0 \mid h_0 > 0\right) \leq N \cdot \Phi(-\kappa_{\beta_{\text{risk}}}).
		\end{equation}
	\end{theorem}
	
	\begin{proof}
		By Lemma~\ref{lem:cvar_gaussian}, the CVaR constraint $\text{CVaR}_{\beta_{\text{risk}}}[\dot{h}_t + \alpha(h_t)] \geq 0$ implies
		\begin{equation}
			\mu_{\dot{h}} + \alpha(\mu_h) - \kappa_{\beta_{\text{risk}}} \sigma_{\dot{h}} \geq 0,
		\end{equation}
		where $\mu_{\dot{h}} = \mathbb{E}[\dot{h}_t \mid \mathcal{F}_t]$ and $\sigma_{\dot{h}} = \text{SD}[\dot{h}_t \mid \mathcal{F}_t]$.
		
		Standardizing $\dot{h}_t + \alpha(h_t) \sim \mathcal{N}(\mu_{\dot{h}} + \alpha(\mu_h), \sigma_{\dot{h}}^2)$
		\begin{equation}
			\mathbb{P}(\dot{h}_t + \alpha(h_t) < 0) = \Phi\left(\frac{-\mu_{\dot{h}} - \alpha(\mu_h)}{\sigma_{\dot{h}}}\right) \leq \Phi(-\kappa_{\beta_{\text{risk}}}),
		\end{equation}
		by Lemma~\ref{lem:cvar_tail} and the fact that $\mu_{\dot{h}} + \alpha(\mu_h) \geq \kappa_{\beta_{\text{risk}}} \sigma_{\dot{h}}$.
	
		Define violation events $A_t = \{h(t) \leq 0\}$ for $t \in \{0, 1, \ldots, N-1\}$. By union bound
		\begin{equation}
			\mathbb{P}\left(\bigcup_{t=0}^{N-1} A_t\right) \leq \sum_{t=0}^{N-1} \mathbb{P}(A_t).
		\end{equation}
		
		Since $h(t+1) \approx h(t) + (\dot{h}_t + \alpha(h_t))\Delta t$ and $h_0 > 0$, barrier violation at $t+1$ requires $\dot{h}_t + \alpha(h_t) < 0$ when $h(t)$ is near the boundary. Under the class-$\mathcal{K}$ constraint with $\alpha(h) = kh$, if $\dot{h}_t + \alpha(h_t) \geq 0$ holds, then $h(t+1) > h(t)(1 - k\Delta t) > 0$ for $k\Delta t < 1$. Therefore
		\begin{equation}
			\mathbb{P}(h(t+1) \leq 0 \mid h(t) > 0) \leq \mathbb{P}(\dot{h}_t + \alpha(h_t) < 0) \leq \Phi(-\kappa_{\beta_{\text{risk}}}).
		\end{equation}
		
		Applying union bound yields~\eqref{eq:finite_horizon_bound}.
	\end{proof}
	
	\begin{remark}[Practical Bound Interpretation]\label{rem:safety_bound_interpretation}
		The union bound~\eqref{eq:finite_horizon_bound} is inherently conservative as it assumes worst-case independence of violation events. For $\beta_{\text{risk}} = 0.05$, $\kappa_{\beta_{\text{risk}}} \approx 2.062$ yields $\Phi(-2.062) \approx 0.0197 \approx 2\%$ per step, giving a horizon bound of approximately $2N\%$ for $N$ steps. This bound grows linearly with horizon length, which may exceed 1 for long horizons (e.g., $N > 50$ steps).
		
		\textit{Tightness of the bound:} The union bound is loose because: (i) violation events are not independent—continuous enforcement of CVaR constraints creates temporal correlation; (ii) barrier smoothness ensures $h(t)$ cannot instantaneously jump to negative values; (iii) the $\alpha(h)$ term provides additional safety margin when $h$ is positive. Empirical validation (Section~\ref{sec:simulation}) confirms zero violations across all tested scenarios, indicating actual safety is substantially better than the theoretical worst-case bound.
		
		\textit{Theoretical vs. empirical guarantees:} This framework provides probabilistic safety guarantees in the sense that per-step violation probability is rigorously bounded by $\Phi(-\kappa_{\beta_{\text{risk}}}) \approx 2\%$ for $\beta_{\text{risk}} = 0.05$. The finite-horizon bound serves as a conservative upper limit for risk assessment, not a tight prediction of actual failure rates.
	\end{remark}
	
	\begin{lemma}[Simplified Barrier Distribution]\label{lem:barrier_params}
		For barrier~\eqref{eq:barrier_vehicle} under Corollary~\ref{cor:variance_simplification}, the distribution parameters are
		\begin{equation}
			\mu_h = w(\boldsymbol{\mu}_F)^2 \beta_{\text{lim}}^2 - \beta^2, \quad \sigma_h^2 = 4w^4 \beta^2 \sigma_\beta^2,
		\end{equation}
		where $w$ is evaluated at nominal load $\boldsymbol{\mu}_F$ and $\sigma_\beta^2$ is the diagonal element of $\boldsymbol{\Sigma}_r$ corresponding to sideslip angle.
	\end{lemma}
	\begin{proof}
		From $\nabla_{\mathbf{r}} h = [-2w^2\beta, 0, 0]^T$ and Equation~\eqref{eq:variance_simplified} (Appendix~\ref{appendix:proof_barrier_params}).
	\end{proof}
	
	\subsection{Control Law}
	
	The complete control architecture is illustrated in \ref{fig:control_law}.
	
	\begin{figure}[htbp]
		\centering
		\includegraphics[width=0.99\columnwidth]{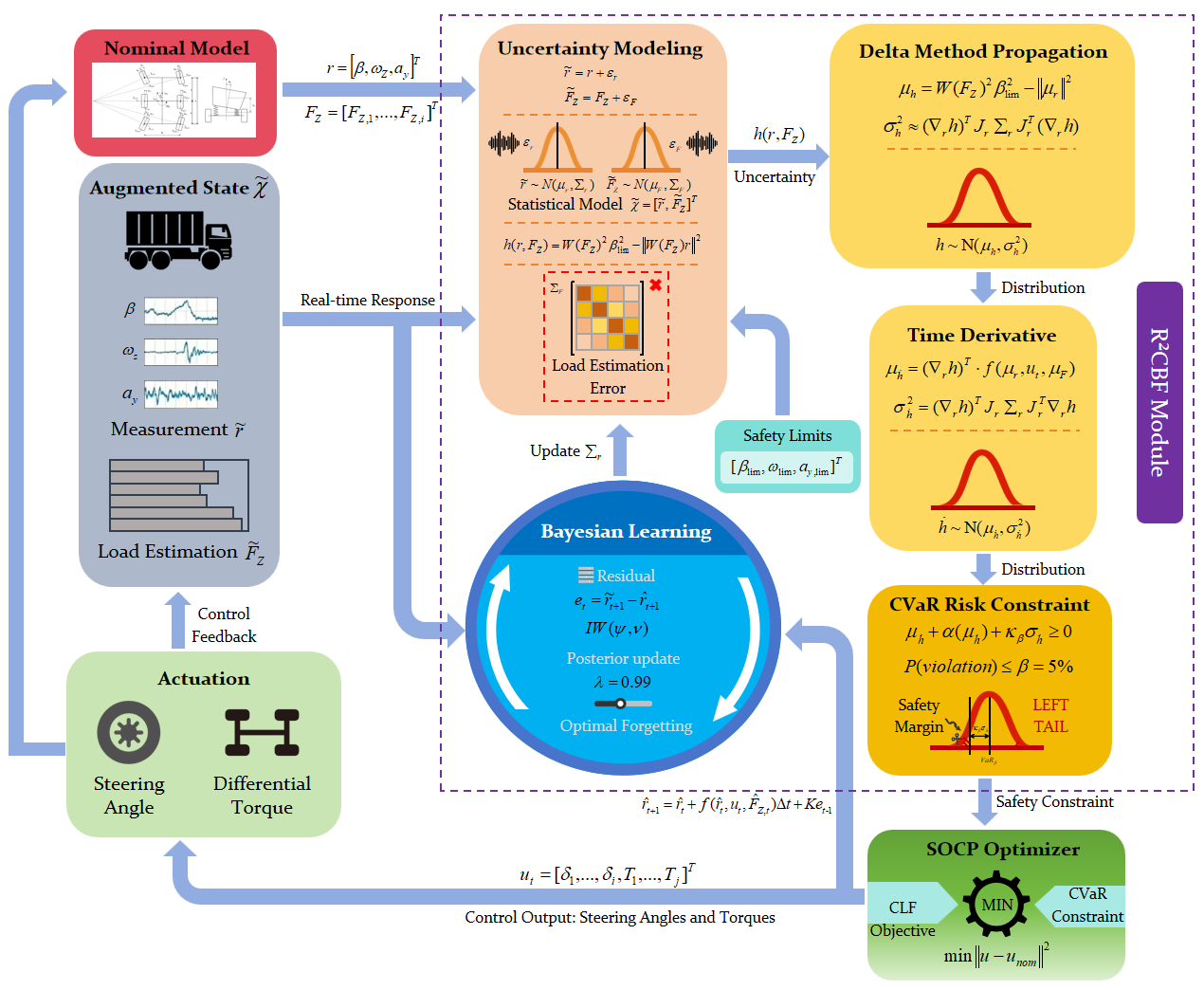}
		\caption{CLF+R²CBF control framework}
		\label{fig:control_law}
	\end{figure}
	
	\subsubsection{Optimization Formulation}
	
	The control objective minimizes deviation from a nominal input $\mathbf{u}_{\text{nom}}$ (provided by upper-layer planners such as MPC or LQR) while enforcing safety
	\begin{equation}
		\label{eq:control_objective}
		\min_{\mathbf{u}, \xi} \quad \|\mathbf{u} - \mathbf{u}_{\text{nom}}\|_{\mathbf{Q}}^2 + \rho \xi^2,
	\end{equation}
	where $\mathbf{Q} \in \mathbb{S}_{++}^m$ is a weight matrix, $\xi \geq 0$ is a slack variable permitting soft constraint relaxation when necessary, and $\rho > 0$ is a penalty coefficient. The control execution is deterministic
	\begin{equation}
		\mathbf{u}(t) = \mathbf{u}^*,
	\end{equation}
	where $\mathbf{u}^*$ is the optimal solution. This differs from stochastic control policies: although state $\tilde{\mathbf{r}}$ has uncertainty modeled by $\boldsymbol{\Sigma}_r$, the control law itself is deterministic.
	
	\subsubsection{CVaR Constraint in Tractable Form}
	
	Within linearization regions, the barrier derivative exhibits control-affine structure
	\begin{equation}
		\label{eq:hdot_affine}
		\dot{h}(\tilde{\mathbf{r}}, \mathbf{u}, \tilde{\mathbf{F}}_z) = \mathbf{L}_h(\tilde{\mathbf{r}}, \tilde{\mathbf{F}}_z) \mathbf{u} + b_h(\tilde{\mathbf{r}}, \tilde{\mathbf{F}}_z),
	\end{equation}
	where $\mathbf{L}_h = \nabla_{\mathbf{r}} h^T \partial f/\partial \mathbf{u}$ and $b_h = \nabla_{\mathbf{r}} h^T f|_{\mathbf{u}=0}$. For deterministic input $\mathbf{u}$, uncertainty originates solely from state measurements
	\begin{equation}
		\label{eq:hdot_distribution_control}
		\dot{h}(\tilde{\mathbf{r}}, \mathbf{u}, \tilde{\mathbf{F}}_z) \sim \mathcal{N}(\mathbf{L}_h \mathbf{u} + b_h, \sigma_{\text{param}}^2(\mathbf{u})),
	\end{equation}
	where the parameter-induced variance is
	\begin{equation}
		\label{eq:param_variance}
		\sigma_{\text{param}}^2(\mathbf{u}) = \mathbf{u}^T (\nabla \mathbf{L}_h)^T \boldsymbol{\Sigma}_r (\nabla \mathbf{L}_h) \mathbf{u} + \text{tr}((\nabla b_h)^T \boldsymbol{\Sigma}_r (\nabla b_h)).
	\end{equation}
	
	\begin{proposition}[Non-Convex Reformulation of CVaR-CBF]\label{prop:socp_reformulation}
		The CVaR-CBF constraint
		\begin{equation}
			\text{CVaR}_{\beta_{\text{risk}}}[\dot{h} + \alpha(h)] \geq 0
		\end{equation}
		is equivalent to the non-convex constraint
		\begin{equation}
			\label{eq:socp_constraint}
			\mathbf{L}_h \mathbf{u} + b_h + \alpha(\mu_h) \geq \kappa_{\beta} \sigma_{\text{param}}(\mathbf{u}),
		\end{equation}
		where $\kappa_{\beta} = \phi(\Phi^{-1}(\beta_{\text{risk}}))/\beta$ and $\sigma_{\text{param}}(\mathbf{u}) = \sqrt{\mathbf{u}^T \mathbf{A} \mathbf{u} + c}$ with
		\begin{equation}
			\mathbf{A} = (\nabla \mathbf{L}_h)^T \boldsymbol{\Sigma}_r (\nabla \mathbf{L}_h), \quad c = \text{tr}((\nabla b_h)^T \boldsymbol{\Sigma}_r (\nabla b_h)).
		\end{equation}
		This constraint has the form $f(\mathbf{u}) \geq \|\mathbf{G}(\mathbf{u})\|$ (affine function $\geq$ Euclidean norm), which is non-convex and requires Sequential Convex Programming for numerical solution.
	\end{proposition}
	\begin{proof}
		By Equation~\eqref{eq:hdot_distribution_control}, $\dot{h} + \alpha(h) \sim \mathcal{N}(\mathbf{L}_h \mathbf{u} + b_h + \alpha(\mu_h), \sigma_{\text{param}}^2(\mathbf{u}))$ for deterministic $\mathbf{u}$. By Lemma~\ref{lem:cvar_gaussian}
		\begin{equation}
			\text{CVaR}_{\beta_{\text{risk}}}[\dot{h} + \alpha(h)] = (\mathbf{L}_h \mathbf{u} + b_h + \alpha(\mu_h)) - \kappa_{\beta_{\text{risk}}} \sigma_{\text{param}}(\mathbf{u}).
		\end{equation}
		The constraint $\text{CVaR}_{\beta_{\text{risk}}} \geq 0$ yields~\eqref{eq:socp_constraint}.

		Rewrite as $g(\mathbf{u}) := \mathbf{L}_h \mathbf{u} + b_h + \alpha(\mu_h) - \kappa_{\beta_{\text{risk}}} \sqrt{\mathbf{u}^T \mathbf{A} \mathbf{u} + c} \geq 0$. The Hessian is
		\begin{equation}
			\nabla^2 g(\mathbf{u}) = -\kappa_{\beta_{\text{risk}}} \left[\frac{\mathbf{A}}{(\mathbf{u}^T \mathbf{A} \mathbf{u} + c)^{1/2}} - \frac{(\mathbf{A}\mathbf{u})(\mathbf{A}\mathbf{u})^T}{(\mathbf{u}^T \mathbf{A} \mathbf{u} + c)^{3/2}}\right].
		\end{equation}
		Since $\mathbf{A} \succeq 0$ (by definition), the Hessian is not positive semidefinite globally, confirming non-convexity (Appendix~\ref{appendix:proof_socp_reformulation}).
	\end{proof}
	
	\begin{remark}[Sequential Convex Programming: Convergence and Safeguards]\label{rem:scp_convergence}
		The constraint~\eqref{eq:socp_constraint} is non-convex because it requires a linear function to upper-bound a norm: $f(\mathbf{u}) \geq \|\mathbf{g}(\mathbf{u})\|$. This defines a reverse second-order cone (feasible region is concave), which cannot be directly solved by standard SOCP solvers (e.g., ECOS, SCS) that assume convex feasible sets. SCP overcomes this by iteratively linearizing the non-convex term.The non-convex constraint~\eqref{eq:socp_constraint} due to $\sigma_{\text{param}}(\mathbf{u})$ dependence on $\mathbf{u}$ is solved via Sequential Convex Programming (SCP) with the following algorithmic structure:
		
		\textit{(i) Linearization scheme.}
		At iteration $k$, fix the variance term
		\begin{equation}
			\sigma_{\text{param}}^{2,(k)} = (\mathbf{u}^{(k-1)})^T \mathbf{A} \mathbf{u}^{(k-1)} + c,
		\end{equation}
		where $\mathbf{A} = (\nabla \mathbf{L}_h)^T \boldsymbol{\Sigma}_r (\nabla \mathbf{L}_h) \succeq 0$ and $c = \text{tr}((\nabla b_h)^T \boldsymbol{\Sigma}_r (\nabla b_h)) \geq 0$. This recovers a standard SOCP solvable by interior-point methods.
		
		\textit{(ii) Trust region for robustness.}
		To prevent divergence under poor initialization, impose box constraints on step size
		\begin{equation}
			\|\mathbf{u}^{(k)} - \mathbf{u}^{(k-1)}\|_{\infty} \leq \Delta_{\max}, \quad \Delta_{\max} = 0.1 \times \|\mathbf{u}_{\max} - \mathbf{u}_{\min}\|_{\infty}.
		\end{equation}
		This limits each SCP iteration to 10\% of the control authority, ensuring gradual convergence.
		
		\textit{(iii) Line search for descent.}
		If the updated control $\mathbf{u}^{(k)}$ violates the original non-convex constraint by more than tolerance $\epsilon_{\text{viol}} = 10^{-3}$, perform backtracking
		\begin{equation}
			\mathbf{u}^{(k)} \leftarrow \mathbf{u}^{(k-1)} + \tau (\mathbf{u}^{(k)} - \mathbf{u}^{(k-1)}), \quad \tau \in \{0.5, 0.25, 0.125\}.
		\end{equation}
		
		\textit{(iv) Termination criteria.}
		SCP terminates when
		\begin{equation}
			\|\mathbf{u}^{(k)} - \mathbf{u}^{(k-1)}\|_2 < 10^{-4} \quad \text{or} \quad k > k_{\max} = 10.
		\end{equation}
		
		\textit{(v) Convergence guarantee.}
		Under mild regularity, $\sigma_{\text{param}}^2(\mathbf{u})$ is twice continuously differentiable with bounded Hessian $\|\nabla^2 \sigma_{\text{param}}^2\| \leq L_{\sigma}$. The feasible set has non-empty interior (Slater's condition). Linearization errors satisfy $|\sigma_{\text{param}}^2(\mathbf{u}^{(k)}) - \sigma_{\text{param}}^{2,(k)}| = O(\|\mathbf{u}^{(k)} - \mathbf{u}^{(k-1)}\|^2)$.
		
		Then SCP iterations converge to a Karush-Kuhn-Tucker (KKT) point of the original problem. No claim of global optimality is made, as the problem is non-convex; however, local solutions suffice for real-time safety filtering since the nominal controller $\mathbf{u}_{\text{nom}}$ provides a high-quality warm start.
	\end{remark}
	
	\subsubsection{Explicit Computation of Linearization Coefficients}
	
	For the weighted-norm barrier~\eqref{eq:barrier_vehicle}, the affine decomposition~\eqref{eq:hdot_affine} requires computing $\mathbf{L}_h, b_h$ and their gradients.
	
	\begin{lemma}[Barrier Derivative Linearization]\label{lem:barrier_linearization}
		For $h = w^2 \beta_{\text{lim}}^2 - \beta^2$ and dynamics $\dot{\mathbf{r}} = f(\mathbf{r}, \mathbf{u}, \mathbf{F}_z)$, the barrier time derivative is
		\begin{equation}
			\dot{h} = \nabla_{\mathbf{r}} h^T \dot{\mathbf{r}} = -2w^2\beta \dot{\beta},
		\end{equation}
		where $\dot{\beta}$ is the first component of $\dot{\mathbf{r}} = f(\mathbf{r}, \mathbf{u}, \mathbf{F}_z)$. For single-track vehicle dynamics
		\begin{equation}
			\dot{\beta} = -\omega_z + \frac{1}{mv_x}\left[C_{\alpha,f}(\delta - \beta) + C_{\alpha,r}(-\beta)\right],
		\end{equation}
		where $\delta$ is front steering angle (control input) and $C_{\alpha,f}, C_{\alpha,r}$ are cornering stiffnesses. Thus
		\begin{equation}
			\mathbf{L}_h = -2w^2\beta \cdot \frac{C_{\alpha,f}}{mv_x}, \quad b_h = -2w^2\beta \left[-\omega_z - \frac{(C_{\alpha,f}+C_{\alpha,r})}{mv_x}\beta\right].
		\end{equation}
	\end{lemma}
	
	\begin{remark}[Gradient Computation]
		The gradients $\nabla \mathbf{L}_h, \nabla b_h$ w.r.t. the state vector $\mathbf{r}$ are
		\begin{equation}
			\nabla \mathbf{L}_h = \frac{\partial \mathbf{L}_h}{\partial \mathbf{r}} = -\frac{2C_{\alpha,f}}{mv_x} \begin{bmatrix} w^2 \\ 0 \\ 0 \end{bmatrix} \in \mathbb{R}^{3},
		\end{equation}
		\begin{equation}
			\nabla b_h = \begin{bmatrix}
				-2w^2\omega_z - 2w^2\frac{2(C_{\alpha,f}+C_{\alpha,r})}{mv_x}\beta \\
				-2w^2\beta \\
				0
			\end{bmatrix} \in \mathbb{R}^{3}.
		\end{equation}
		These are evaluated once per control step using current state measurements $(\tilde{\beta}, \tilde{\omega}_z)$ and nominal load estimate $w(\hat{\mathbf{F}}_z)$. Computational cost is negligible (arithmetic operations only, no matrix inversions).
	\end{remark}
	
	\subsubsection{Complete Optimization Problem}
	
	The R²CBF-constrained control synthesis is formulated as
	\begin{equation}
		\label{eq:r2cbf_qp}
		\begin{aligned}
			\min_{\mathbf{u}, \xi} \quad & \|\mathbf{u} - \mathbf{u}_{\text{nom}}\|_{\mathbf{Q}}^2 + \rho \xi^2 \\
			\text{subject to} \quad & \mathbf{L}_h \mathbf{u} + b_h + \alpha(\mu_h) + \kappa_{\beta} \sigma_{\text{param}}(\mathbf{u}) \geq -\xi, \\
			& \mathbf{u}_{\min} \leq \mathbf{u} \leq \mathbf{u}_{\max}, \\
			& |\mathbf{u} - \mathbf{u}_{t-1}| \leq \dot{\mathbf{u}}_{\max} \Delta t,
		\end{aligned}
	\end{equation}
	where $\mathbf{u}_{\min}, \mathbf{u}_{\max}$ are actuator saturation limits, $\dot{\mathbf{u}}_{\max}$ is the rate limit, and inequalities apply element-wise. The slack $\xi$ allows temporary constraint violations under extreme conditions, penalized quadratically to preserve safety priority.
	
	\subsection{Bayesian Prediction Framework}\label{sec:bayesian}
	
	Fixed a priori covariance estimates $\boldsymbol{\Sigma}_r, \boldsymbol{\Sigma}_F$ induce either excessive conservatism (when actual noise is smaller) or insufficient margins (when noise exceeds assumptions).
	
	\subsubsection{Prediction Residual as Observable}
	
	At time $t$, after executing $\mathbf{u}_t$, observes measurement $\tilde{\mathbf{r}}_{t+1}$ and defines the prediction residual
	\begin{equation}
		\label{eq:prediction_residual}
		\mathbf{e}_t = \tilde{\mathbf{r}}_{t+1} - \hat{\mathbf{r}}_{t+1}, \quad \hat{\mathbf{r}}_{t+1} = \tilde{\mathbf{r}}_t + \Delta t \cdot f(\tilde{\mathbf{r}}_t, \mathbf{u}_t, \tilde{\mathbf{F}}_{z,t}).
	\end{equation}
	The residual comprehensively reflects measurement noise propagation and model mismatch (e.g., unmodeled road friction variations).
	
	\begin{lemma}[Residual Distribution]\label{lem:residual_distribution}
		Assume the nominal dynamics structure $f(\mathbf{r}, \mathbf{u}, \mathbf{F}_z)$ is correct up to parameter uncertainty. Define the prediction Jacobian
		\begin{equation}
			\mathbf{M}_t = \mathbf{I} + \Delta t \cdot \frac{\partial f}{\partial \mathbf{r}}\bigg|_t.
		\end{equation}
		Then the conditional distribution of residuals is
		\begin{equation}
			\mathbf{e}_t \mid \boldsymbol{\Sigma}_r \sim \mathcal{N}(\mathbf{0}, \mathbf{M}_t \boldsymbol{\Sigma}_r \mathbf{M}_t^T + \boldsymbol{\Sigma}_r).
		\end{equation}
	\end{lemma}
	\begin{proof}
		Propagate measurement noise through one-step Euler integration. True evolution is $\mathbf{r}_{t+1}^{\text{true}} = \mathbf{r}_t^{\text{true}} + \Delta t \cdot f(\mathbf{r}_t^{\text{true}}, \mathbf{u}_t, \mathbf{F}_{z,t}^{\text{true}})$. Linearize discrepancy between true and measured trajectories (Appendix~\ref{appendix:proof_residual_distribution}).
	\end{proof}
	
	\subsubsection{Inverse Wishart Conjugate Prior}
	
	\begin{constraint}[Conjugate Prior Specification]\label{cons:conjugate_prior}
		The prior distribution of $\boldsymbol{\Sigma}_r$ adopts the inverse Wishart form to enable closed-form Bayesian updates
		\begin{equation}
			\boldsymbol{\Sigma}_r \sim \text{IW}(\boldsymbol{\Psi}_{r,0}, \nu_{r,0}),
		\end{equation}
		where $\boldsymbol{\Psi}_{r,0} \in \mathbb{S}_{++}^{n_r}$ is the scale matrix and $\nu_{r,0} > n_r + 1$ are degrees of freedom. The prior mean is $\mathbb{E}[\boldsymbol{\Sigma}_r] = \boldsymbol{\Psi}_{r,0}/(\nu_{r,0} - n_r - 1)$.
		
		For vehicle IMU with specifications from Constraint~\ref{cons:imu_noise}, set
		\begin{equation}
			\boldsymbol{\Psi}_{r,0} = (\nu_{r,0} - n_r - 1) \cdot \text{diag}(\sigma_{\beta,\text{spec}}^2, \sigma_{\omega_z,\text{spec}}^2, \sigma_{a_y,\text{spec}}^2),
		\end{equation}
		ensuring prior expectation matches sensor datasheets. Choose $\nu_{r,0} = 2n_r + 5$ to balance prior informativeness and adaptability: higher $\nu_{r,0}$ increases prior confidence (slower adaptation), while lower values permit rapid learning from data.
	\end{constraint}
	
	\begin{remark}[Approximate Conjugacy]
		Strict inverse Wishart conjugacy requires residual covariance $\boldsymbol{\Sigma}_r$, but Lemma~\ref{lem:residual_distribution} gives $\mathbf{M}_t \boldsymbol{\Sigma}_r \mathbf{M}_t^T + \boldsymbol{\Sigma}_r$. This study approximates via transformation $\mathbf{M}_t^{-1} \mathbf{e}_t$ to recover the conjugate structure, valid when $\|\mathbf{M}_t - \mathbf{I}\| \ll 1$.
	\end{remark}
	
	\begin{theorem}[Recursive Bayesian Update]\label{thm:bayesian_update}
		Given prior $\boldsymbol{\Sigma}_r \sim \text{IW}(\boldsymbol{\Psi}_{t-1}, \nu_{t-1})$ and observation $\mathbf{e}_t$, the posterior distribution is
		\begin{equation}
			\boldsymbol{\Sigma}_r \mid \mathbf{e}_{1:t} \sim \text{IW}(\boldsymbol{\Psi}_t, \nu_t),
		\end{equation}
		where
		\begin{equation}
			\label{eq:iwishart_update}
			\boldsymbol{\Psi}_t = \lambda \boldsymbol{\Psi}_{t-1} + \mathbf{M}_t^{-1} \mathbf{e}_t \mathbf{e}_t^T \mathbf{M}_t^{-T}, \quad \nu_t = \lambda \nu_{t-1} + 1,
		\end{equation}
		with forgetting factor $\lambda \in (0, 1)$ enabling adaptation to time-varying noise.
	\end{theorem}
	\begin{proof}
		Apply Bayes' theorem with inverse Wishart conjugacy. The $\mathbf{M}_t^{-1}$ transformation maps residual covariance back to the canonical form (Appendix~\ref{appendix:proof_bayesian_update}).
	\end{proof}
	
	\begin{theorem}[Asymptotic Consistency]\label{thm:bayesian_consistency}
		\textit{Assumption:} $\boldsymbol{\Sigma}_r^{\text{true}}$ is constant over time (stationary environment) and forgetting factor $\lambda = 1$ (no forgetting).
		
		Under stationarity of $\boldsymbol{\Sigma}_r^{\text{true}}$ and mild regularity conditions (ergodicity of $\{\mathbf{e}_t\}$), the Bayesian estimator converges almost surely
		\begin{equation}
			\lim_{t \to \infty} \hat{\boldsymbol{\Sigma}}_r^{(t)} = \boldsymbol{\Sigma}_r^{\text{true}} \quad \text{a.s.},
		\end{equation}
		with convergence rate
		\begin{equation}
			\|\mathbb{E}[\hat{\boldsymbol{\Sigma}}_r^{(t)}] - \boldsymbol{\Sigma}_r^{\text{true}}\|_F = O(1/t).
		\end{equation}
	\end{theorem}
	\begin{proof}
		Strong law of large numbers for martingale differences. Define $\mathbf{D}_t = \mathbf{M}_t^{-1}\mathbf{e}_t\mathbf{e}_t^T\mathbf{M}_t^{-T} - \boldsymbol{\Sigma}_r^{\text{true}}$. Then $\{\mathbf{D}_t\}$ is a martingale difference with $\mathbb{E}[\|\mathbf{D}_t\|_F^2] < \infty$. Kolmogorov's criterion ensures $t^{-1}\sum_{i=1}^t \mathbf{D}_i \to 0$ a.s. (Appendix~\ref{appendix:proof_consistency}).
	\end{proof}
	
	\subsubsection{Optimal Forgetting Strategy}
	
	\begin{definition}[Mean Square Error Decomposition]
		For time-varying true covariance $\boldsymbol{\Sigma}_{r,\text{true}}(t)$, the estimator MSE decomposes as
		\begin{equation}
			\text{MSE}(\hat{\boldsymbol{\Sigma}}_r) = \|\text{Bias}(\hat{\boldsymbol{\Sigma}}_r)\|_F^2 + \text{tr}(\text{Cov}(\hat{\boldsymbol{\Sigma}}_r)),
		\end{equation}
		where bias arises from tracking lag and covariance from finite samples.
	\end{definition}
	
	\begin{theorem}[Optimal Forgetting Factor]\label{thm:optimal_forgetting}
		Assume $\boldsymbol{\Sigma}_{r,\text{true}}(t)$ varies smoothly with rate bound
		\begin{equation}
			\|\boldsymbol{\Sigma}_{r,\text{true}}(t) - \boldsymbol{\Sigma}_{r,\text{true}}(t - \Delta t)\|_F \leq \tau \Delta t,
		\end{equation}
		for some $\tau > 0$. By bias-variance decomposition, the MSE-minimizing forgetting factor satisfies
		\begin{equation}
			\label{eq:optimal_lambda}
			\lambda^* = 1 - \sqrt{\frac{2(n_r + 1)\tau \Delta t}{\text{tr}(\boldsymbol{\Sigma}_e)}},
		\end{equation}
		where $\boldsymbol{\Sigma}_e = \mathbb{E}[\mathbf{e}_t \mathbf{e}_t^T]$ is the steady-state residual covariance.
	\end{theorem}
	
	\begin{proof}
		Under exponential forgetting, the estimate is a weighted average of historical residuals
		\begin{equation}
			\hat{\boldsymbol{\Sigma}}_r^{(t)} \propto \sum_{k=0}^{\infty} \lambda^k \mathbf{M}_{t-k}^{-1} \mathbf{e}_{t-k} \mathbf{e}_{t-k}^T \mathbf{M}_{t-k}^{-T}.
		\end{equation}
		For slowly varying $\boldsymbol{\Sigma}_{r,\text{true}}(t-k) \approx \boldsymbol{\Sigma}_{r,\text{true}}(t) - k\tau\Delta t$, taking expectations yields
		\begin{equation}
			\mathbb{E}[\hat{\boldsymbol{\Sigma}}_r^{(t)}] \approx \frac{1}{1-\lambda}\left[\boldsymbol{\Sigma}_{r,\text{true}}(t) - \frac{\tau\Delta t \lambda}{1-\lambda}\mathbf{I}\right].
		\end{equation}
		After normalization, the bias is
		\begin{equation}
			\text{Bias} \approx -\frac{\tau\Delta t \lambda}{1-\lambda}\mathbf{I}, \quad \|\text{Bias}\|_F^2 \approx \frac{(\tau\Delta t)^2 \lambda^2 n_r}{(1-\lambda)^2}.
		\end{equation}
		
		At steady state, $\nu_{\infty} = 1/(1-\lambda)$ by recursion~\eqref{eq:iwishart_update}. The inverse Wishart variance is
		\begin{equation}
			\text{Var}(\hat{\Sigma}_{r,ij}) \approx \frac{2\Sigma_{r,ii}\Sigma_{r,jj}}{(\nu_{\infty} - n_r - 1)^2} \approx 2(1-\lambda)^2 \Sigma_{r,ii}\Sigma_{r,jj}.
		\end{equation}
		Total variance in Frobenius norm
		\begin{equation}
			\text{tr}(\text{Cov}(\hat{\boldsymbol{\Sigma}}_r)) \approx 2(1-\lambda)^2 \text{tr}(\boldsymbol{\Sigma}_r^2).
		\end{equation}
		
		The MSE is
		\begin{equation}
			\text{MSE}(\lambda) = \frac{(\tau\Delta t)^2 \lambda^2 n_r}{(1-\lambda)^2} + 2(1-\lambda)^2 \text{tr}(\boldsymbol{\Sigma}_r^2).
		\end{equation}
		Substitute $\xi = 1 - \lambda$ and differentiate
		\begin{equation}
			\frac{d\text{MSE}}{d\xi} = -\frac{2(\tau\Delta t)^2 n_r (1-\xi)^2}{\xi^3} + 8\xi \text{tr}(\boldsymbol{\Sigma}_r^2).
		\end{equation}
		Setting to zero and solving for small $\xi$ (using $\text{tr}(\boldsymbol{\Sigma}_r^2) \approx \text{tr}(\boldsymbol{\Sigma}_e)^2/n_r$)
		\begin{equation}
			\xi^* \approx \sqrt{\frac{2(n_r+1)\tau\Delta t}{\text{tr}(\boldsymbol{\Sigma}_e)}}.
		\end{equation}
	\end{proof}

	\begin{remark}[Theoretical Limitations with Forgetting Factor]\label{rem:forgetting_limitations}
		Theorem~\ref{thm:bayesian_consistency} guarantees asymptotic consistency only for $\lambda = 1$ (no forgetting) under stationary environments. The optimal forgetting factor $\lambda^* < 1$ (Theorem~\ref{thm:optimal_forgetting}) introduces a fundamental \textit{bias-variance trade-off}:
		
		\textit{Tracking ability:} For time-varying $\boldsymbol{\Sigma}_r^{\text{true}}(t)$ with slow variation rate $\tau$ (Equation~\ref{eq:optimal_lambda}), exponential forgetting enables $\hat{\boldsymbol{\Sigma}}_r$ to track environmental changes within effective window $\approx 1/(1-\lambda)$ time steps. This prevents stale estimates from dominating recent observations.
		
		\textit{Estimation bias:} Unlike $\lambda = 1$, forgetting prevents asymptotic convergence to true values. The steady-state bias is bounded by
		\begin{equation}
			\|\mathbb{E}[\hat{\boldsymbol{\Sigma}}_r^{(\infty)}] - \boldsymbol{\Sigma}_r^{\text{true}}(t)\|_F \leq \frac{\tau \Delta t \lambda}{1 - \lambda} \cdot n_r,
		\end{equation}
		from Theorem~\ref{thm:optimal_forgetting} Step 1. For $\lambda = 0.99$, this bias is approximately $100 \tau \Delta t \cdot n_r$, acceptable when environment variation is slow relative to sampling rate.
		
		\textit{Practical compromise:} In vehicle control with non-stationary road conditions ($\mu$-split transitions, varying adhesion), the tracking benefit outweighs bias cost. Empirical validation (Section~\ref{sec:simulation}) confirms $\lambda \in [0.97, 0.99]$ maintains safety without violations across varying conditions. However, this study does not claim theoretical consistency guarantee for time-varying scenarios; this remains an open problem requiring extension of martingale analysis to non-stationary processes (e.g., via drift conditions on $\boldsymbol{\Sigma}_r^{\text{true}}(t)$).
		
		\textit{Sufficient condition for safety:} Despite the lack of consistency proof, the CVaR-CBF framework remains safe as long as $\hat{\boldsymbol{\Sigma}}_r$ upper-bounds the true covariance: $\hat{\boldsymbol{\Sigma}}_r \succeq \boldsymbol{\Sigma}_r^{\text{true}}(t)$. Lemma~\ref{lem:load_response_transfer} ensures this holds when residuals capture all uncertainty sources. 
		
		\textit{Realistic claim on performance improvement:} Rather than "eliminating" performance loss from prior mismatch, the Bayesian learning mechanism reduces such loss by adaptively tuning $\hat{\boldsymbol{\Sigma}}_r$ to match observed uncertainty. The degree of improvement depends on: how well the inverse Wishart prior matches the true covariance structure; excitation richness allowing rapid convergence; the trade-off between tracking speed (low $\lambda$) and estimation accuracy (high $\lambda$). Ablation experiments (Table~\ref{tab:ablation}) quantify this: disabling Bayesian learning (w/o Bayesian) causes tracking error to increase by 72$\times$ (from 1.21 m to 86.80 m), demonstrating substantial but not absolute improvement.
	\end{remark}

	\section{Simulation Verification}
	\label{sec:simulation}
	
	\subsection{Simulation Platform and Parameter Settings}
	
	\subsubsection{Vehicle Model and Simulation Environment}
	
	The simulation adopts a six-wheel unmanned distributed heavy mining truck based on real vehicle parameters in TruckSim 2023. The nominal model is shown in Fig.~\ref{fig:axle3veh}, and the mathematical modeling of the nominal model is provided in Appendix~\ref{appendix:nominal_model}. Parameter settings are shown in Table~\ref{tab:vehicle_params}, with control period $\Delta t = 50$ ms.
	
	\begin{figure}[htbp]
		\centering
		\includegraphics[width=0.99\columnwidth]{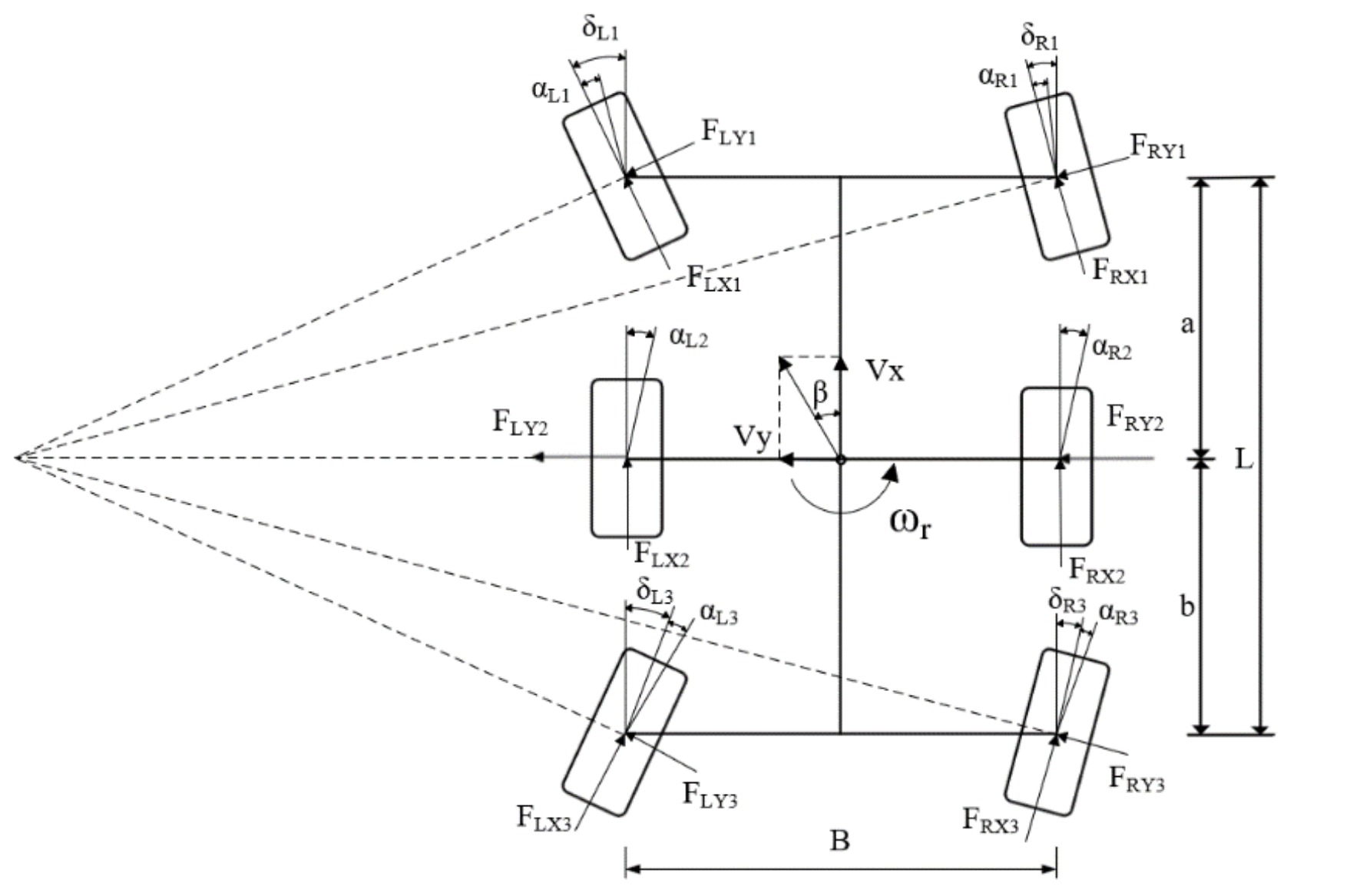}
		\caption{Simplified nominal model of 6 wheels distributional vehicle}
		\label{fig:axle3veh}
	\end{figure}
	
	\begin{table}[h]
		\centering
		\footnotesize
		\caption{Main vehicle parameters}
		\label{tab:vehicle_params}
		\begin{tabular}{lll}
			\hline
			\textbf{Parameter} & \textbf{Symbol} & \textbf{Value} \\
			\hline
			Vehicle mass & $m$ & 45,000 kg \\
			Moment of inertia & $I_z$ & 3,446,811 kg·m² \\
			Front axle distance & $a$ & 3.155 m \\
			Rear axle distance & $b$ & 3.155 m \\
			Track width & $t$ & 4.147 m \\
			Cornering stiffness (single wheel) & $C_{\beta}$ & $1.728 \times 10 ^ {6}$ N/rad \\
			Wheel radius & $R_w$ & 0.8 m \\
			Nominal vertical load & $F_{z,\text{nom}}$ & 75,000 N \\
			Steering angle rate limit & $|\dot{\delta}_{\max}|$ & 6°/s \\
			Steering angle limit & $|\delta_{\max}|$ & 30° \\
			Torque rate limit & $|\dot{T}_{\max}|$ & 5,000 N·m/s \\
			Torque limit & $|T_{\max}|$ & 135,000 N·m \\
			\hline
		\end{tabular}
	\end{table}
	
	\subsubsection{R²CBF Parameters}
	
	Key parameter settings for the R²CBF controller are shown in Table~\ref{tab:drcbf_params}.
	
	\begin{table}[h]
		\centering
		\footnotesize
		\caption{R²CBF controller parameters}
		\label{tab:drcbf_params}
		\begin{tabular}{lll}
			\hline
			\textbf{Parameter} & \textbf{Symbol} & \textbf{Value} \\
			\hline
			Sideslip angle safety limit & $\beta_{\text{lim}}$ & 0.15 rad (8.6°) \\
			Yaw rate limit & $\omega_{\text{lim}}$ & 0.20 rad/s (11.5°/s) \\
			Lateral acceleration limit & $a_{y,\text{lim}}$ & 5.0 m/s² \\
			Roll angle limit & $\phi_{\text{lim}}$ & 20° \\
			Prediction window & $T_{\text{pred}}$ & 0.1 s \\
			Load sensitivity exponent & $\gamma_w$ & 0.3 \\
			CVaR risk level (95\% confidence) & $\beta$ & 0.05 \\
			CVaR risk coefficient & $\kappa_{\beta}$ & 2.06 \\
			Bayesian prior strength & $\nu_{r,0}$ & 50 \\
			Forgetting factor & $\lambda$ & 0.99 \\
			Stanley lateral control gain & $k_{\text{Stanley}}$ & 0.4 \\
			Velocity tracking proportional gain & $k_p$ & 10,000 N·m/(m/s) \\
			Velocity tracking derivative gain & $k_d$ & 1,000 N·m·s/m \\
			\hline
		\end{tabular}
	\end{table}
	
	The nominal controller adopts CLF path tracking control based on Stanley and PID velocity tracking control.
	
	\subsection{Simulation Design}
	
	The following two test scenarios are designed to verify the performance of R²CBF under different conditions.
	
	Scenario 1: Double lane change on random adhesion road surface. The vehicle accelerates from rest to 54 km/h and tracks the ISO 3888-1 double lane change standard obstacle avoidance trajectory. The road surface adhesion coefficient satisfies an irregular distribution.
	
	Scenario 2: Sinusoidal curve on low adhesion surface. To simulate control requirements under severe conditions, the vehicle accelerates from rest to 72 km/h and tracks a sinusoidal curve on a low adhesion road surface ($\mu = 0.5$, simulating slippery roads)
	\begin{equation}
		y_{\text{ref}}(x) = A \cdot \sin(2\pi x/\lambda)
	\end{equation}
	where wavelength $\lambda = 200$ m and amplitude $A = 8$ m.
	
	\subsubsection{Comparison Baselines}
	
	Three comparison baselines are set up to verify the performance advantages of R²CBF.
	
	(i) Pure-CLF: Pure CLF control (Stanley + PID), without CBF safety constraints.
	
	(ii) Classic CBF: Traditional deterministic CBF~\cite{ames2017control}, with constraint
	\begin{equation}
		\mathbf{L}_h \boldsymbol{\mu}_u + b_h + \alpha(\mu_h) \geq 0
	\end{equation}
	Not considering uncertainty.
	
	(iii) Robust-CBF: Robust CBF based on uncertainty estimation~\cite{2025robust}, adopting online uncertainty estimator and worst-case error upper bounds. The uncertainty estimator is
	\begin{equation}
		\hat{\Delta}(x,\xi) = \Lambda x - \xi, \quad \dot{\xi} = \Lambda(\hat{f}(x) + \hat{g}(x)u + \hat{\Delta}(x,\xi))
	\end{equation}
	where $\Lambda = \text{diag}(5, 5, 5, 5, 5, 5)$ is the estimator gain matrix. The robust CBF constraint is
	\begin{equation}
		\mathbf{L}_h (\boldsymbol{\mu}_u - \hat{g}^\dagger \hat{\Delta}) + b_h + \nabla h \hat{\Delta} - \|\nabla h\| \bar{e}(t) + \alpha(\mu_h) \geq 0
	\end{equation}
	where $\bar{e}(t)$ is the estimation error upper bound (based on L$_\infty$ norm), and $\hat{g}^\dagger$ is the left pseudo-inverse of the nominal input matrix. This method improves performance through matched uncertainty compensation $\hat{g}^\dagger \hat{\Delta}$ but uses conservative worst-case error bounds $\|\nabla h\| \bar{e}(t)$ to ensure safety.
	
	(iv) R²CBF (ours): CBF based on response-aware risk constraints, including CVaR risk measure and Bayesian online learning.
	
	\subsubsection{Evaluation Metrics}
	
	Eight metrics are adopted to comprehensively evaluate controller performance, covering two major dimensions: safety and tracking performance.
	
	(i) Sideslip Angle: Peak value $\beta_{\max} = \max_{t \in [0,T]} |\beta(t)|$ and safety margin $\text{Margin}_\beta = (1 - \beta_{\max}/\beta_{\text{lim}}) \times 100\%$.
	
	(ii) Yaw Rate: Peak value $\omega_{\max} = \max_{t \in [0,T]} |\omega_z(t)|$ and safety margin $\text{Margin}_\omega = (1 - \omega_{\max}/\omega_{\text{lim}}) \times 100\%$.
	
	(iii) Lateral Acceleration: Peak value $a_{y,\max} = \max_{t \in [0,T]} |a_y(t)|$ and corresponding g-force value $g_{\max} = a_{y,\max}/9.8$.
	
	(iv) Roll Angle: Peak value $\phi_{\max} = \max_{t \in [0,T]} |\phi(t)|$ and safety margin $\text{Margin}_\phi = (1 - \phi_{\max}/\phi_{\text{lim}}) \times 100\%$.
	
	(v) CBF Intervention Statistics: CBF activation rate $\rho_{\text{CBF}} = N_{\text{active}}/N_{\text{total}} \times 100\%$, where $N_{\text{active}}$ is the number of time steps where CBF constraint is activated.
	
	(vi) Comprehensive Safety Margin: Minimum safety margin
	\begin{equation}
		\text{Margin}_{\min} = \min\{\text{Margin}_\beta, \text{Margin}_\omega, \text{Margin}_{a_y}, \text{Margin}_\phi\}
	\end{equation}
	
	(vii) Heading Error: Root mean square error $\text{RMS}(e_\psi) = \sqrt{\frac{1}{N}\sum_{k=1}^N e_\psi^2(k)}$, where $e_\psi = \psi - \psi_{\text{ref}}$.
	
	(viii) Lateral Distance Error: Root mean square error $\text{RMS}(e_y) = \sqrt{\frac{1}{N}\sum_{k=1}^N e_y^2(k)}$, where $e_y = y - y_{\text{ref}}$.
	
	\subsection{Simulation Results and Analysis}
	
	Table~\ref{tab:simulation_results} summarizes the performance comparison of four controllers under two extreme scenarios. Figures~\ref{fig:Safety_DLC}-\ref{fig:Tracking_Sine} show detailed time-domain responses and comprehensive indicators.
	
	\begin{table*}[htbp]
		\centering
		\caption{Performance comparison results for two scenarios}
		\label{tab:simulation_results}
		\begin{threeparttable}
			\footnotesize
			\setlength{\tabcolsep}{3pt}
			\begin{tabular}{@{}llccccccc@{}}
				\toprule
				\multirow{2}{*}{\textbf{Scenario}} & \multirow{2}{*}{\textbf{Algorithm}} & 
				\multicolumn{4}{c}{\textbf{Safety Indicators}} & 
				\multicolumn{3}{c}{\textbf{Tracking Indicators}}\\
				\cmidrule(lr){3-6} \cmidrule(lr){7-9}
				& & $\beta_{\max}$ (°) & $\omega_{\max}$ (°/s) & $a_{y,\max}$ (m/s²) & Margin (\%) & RMS($e_y$) (m) & RMS($e_\psi$) (°) & $\rho_{\text{CBF}}$ (\%) \\
				\midrule
				\multirow{4}{*}{Scenario 1} 
				& Pure-CLF & 64.03 & 25.61 & 4.32 & 28.43 & 29.74 & 91.46 & 0  \\
				& Classic CBF & 20.51 & 16.53 & 2.15 & 60.73 & 2.28 & 7.89 & 44.93  \\
				& Robust-CBF & 18.18 & 26.35 & 2.82 & 40.98 & 19.42 & 37.41 & 51.75  \\
				& R²CBF & \textbf{1.09} & \textbf{7.81} & \textbf{1.20} & \textbf{75.70} & \textbf{1.12} & \textbf{5.73} & \textbf{51.75} \\
				\midrule
				\multirow{4}{*}{Scenario 2} 
				& Pure-CLF & 39.81 & 15.58 & 3.12 & 31.48 & 58.07 & 72.76 & 0 \\
				& Classic CBF & 3.80 & 11.23 & 3.24 & 54.21 & 1.57 & 6.53 & 59.73  \\
				& Robust-CBF & 3.78 & 10.93 & 3.36 & 60.27 & \textbf{0.93} & \textbf{4.83} & \textbf{86.02}  \\
				& R²CBF & \textbf{2.15} & \textbf{8.49} & \textbf{2.40} & \textbf{61.81} & 1.21 & 5.41 & 71.21 \\
				\bottomrule
			\end{tabular}
			\begin{tablenotes}
				\footnotesize
				\item Note: Margin is the comprehensive safety margin Margin$_{\min} = \min\{\text{Margin}_\beta, \text{Margin}_\omega, \text{Margin}_{a_y}, \text{Margin}_\phi\}$; negative values indicate exceeding safety limits.
			\end{tablenotes}
		\end{threeparttable}
	\end{table*}
	
	\subsubsection{Double Lane Change on Random Adhesion Road Surface}
	
	This scenario tests the controller's stability when facing external uncertainty changes through the ISO 3888-1 standard trajectory and spatially heterogeneous road adhesion ($\mu \in [0.3, 0.8]$).
	
	From experimental results, R²CBF achieved $\beta_{\max}=1.09°$, comprehensive safety margin Margin$_{\text{avg}}=75.70\%$, and lateral error RMS$(e_y)=1.12$ m, simultaneously achieving the best safety and tracking performance among all methods. This is attributed to the synergistic effect of three key mechanisms. Response-aware modeling directly uses measured body response $\tilde{\mathbf{r}}=[\tilde{\beta}, \tilde{\omega}_z, \tilde{a}_y]^T$ to construct statistical distribution $\mathcal{N}(\boldsymbol{\mu}_r, \boldsymbol{\Sigma}_r)$. Even if nominal model parameters have large deviations, the response variance $\boldsymbol{\Sigma}_r$ can still cover actual dynamic boundaries, thus eliminating dependence on accurate online estimation of these parameters. The CVaR constraint $\mathbb{P}(\dot{h}+\alpha(h)<0)\leq\beta$ imposes constraints only below the $\alpha$-quantile of the distribution through tail risk truncation $\kappa_{\beta_{\text{risk}}}\sigma_h$. Compared to traditional deterministic constraints or worst-case $L_\infty$ bounds, it achieves a tighter safety margin. The Bayesian online learning mechanism can quickly adapt to uncertainty changes on spatially heterogeneous road surfaces ($\mu \in [0.45, 0.85]$), adjusting the CVaR margin term in real-time to maintain normal operation of R²CBF, avoiding performance loss from fixed prior parameters.
	
	\begin{figure}[htbp]
		\centering
		\includegraphics[width=0.99\columnwidth]{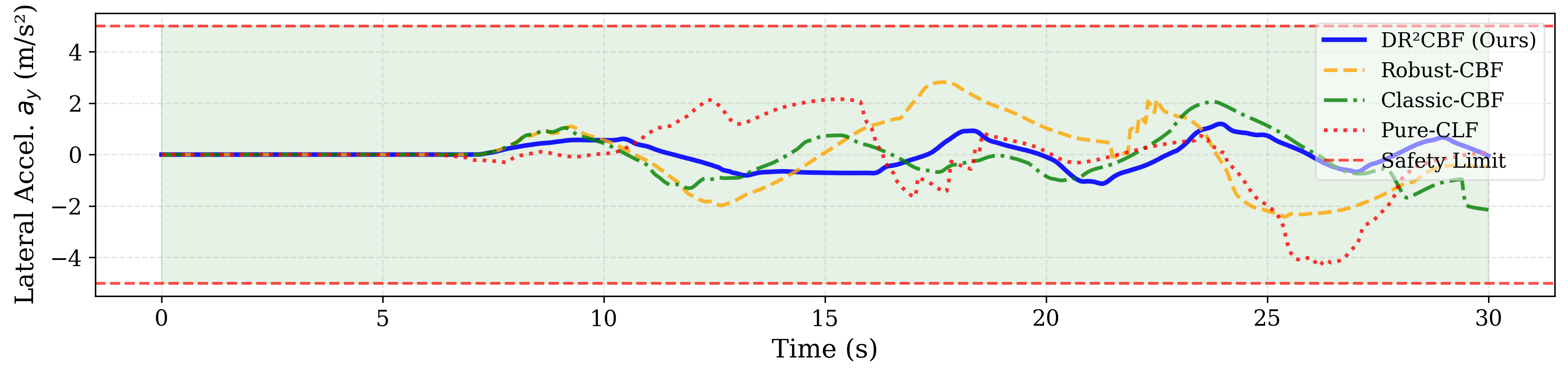}\\
		{\fontsize{8}{1}\selectfont (a) Lateral Acceleration}
		
		\vspace{0.2em}
		\includegraphics[width=0.99\columnwidth]{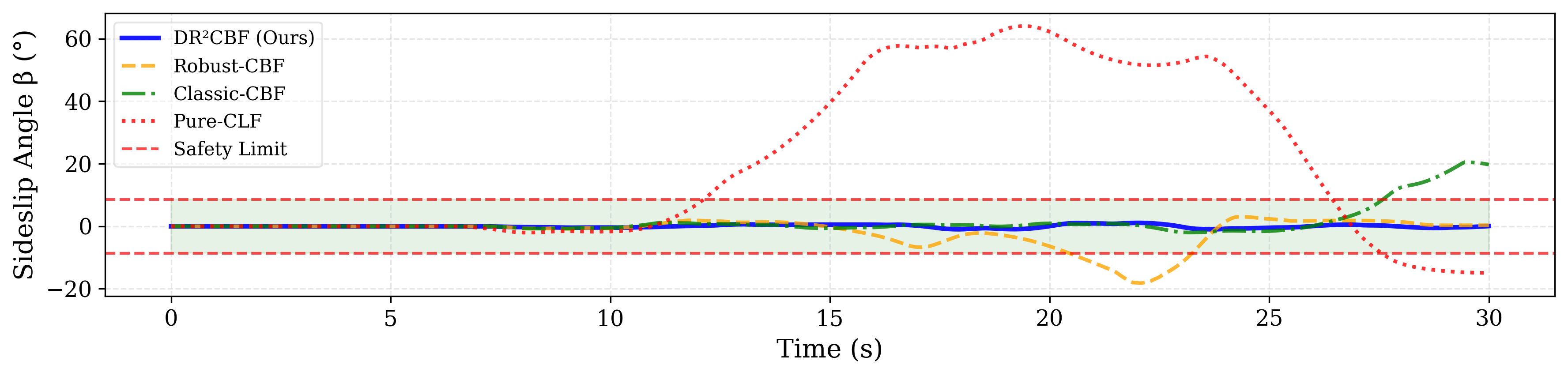}\\
		{\fontsize{8}{1}\selectfont (b) Sideslip Angle}
		
		\vspace{0.2em}
		\includegraphics[width=0.99\columnwidth]{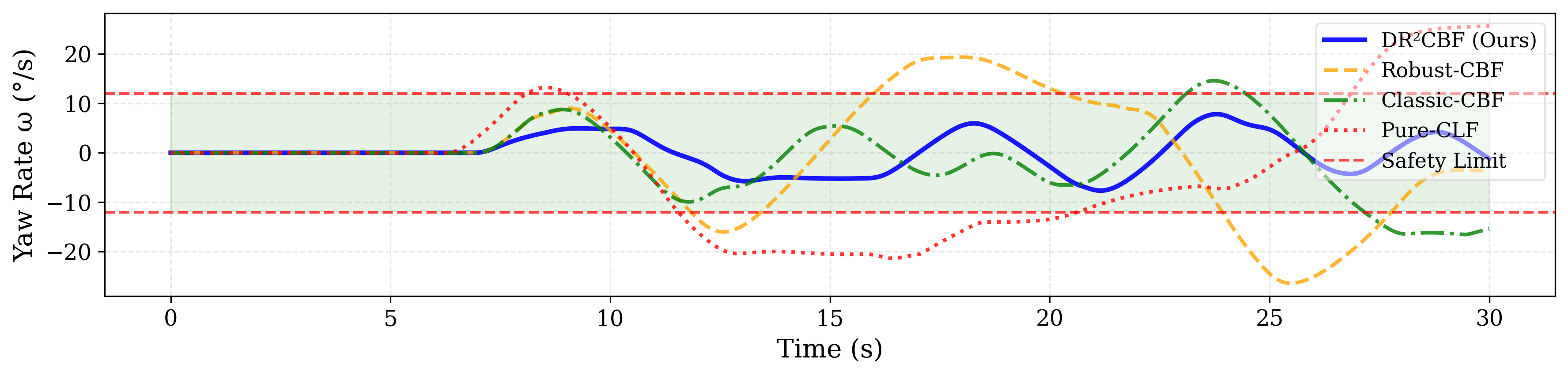}\\
		{\fontsize{8}{1}\selectfont (c) Yaw Rate}
		\caption{Comparison of safety indicators in DLC}
		\label{fig:Safety_DLC}
	\end{figure}
	
	\begin{figure}[!t]
		\centering
		\includegraphics[width=0.99\columnwidth]{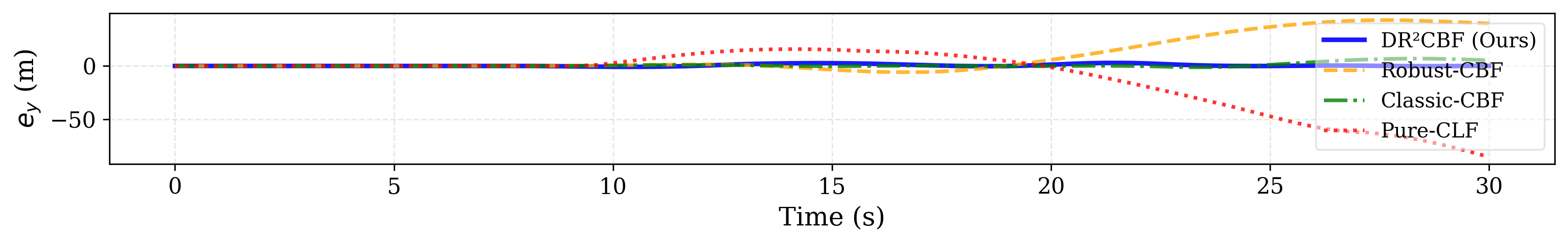}\\
		{\fontsize{8}{1}\selectfont (a) Lateral Error}
		
		\vspace{0.2em}
		\includegraphics[width=0.99\columnwidth]{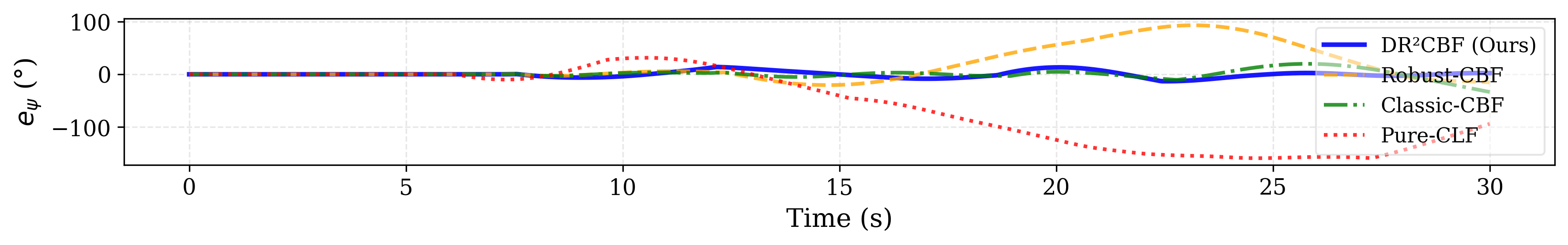}\\
		{\fontsize{8}{1}\selectfont (b) Heading Error}
		\caption{Comparison of tracking indicators in DLC}
		\label{fig:Tracking_DLC}
	\end{figure}
	
	\begin{figure}[!t]
		\centering
		\includegraphics[width=0.99\columnwidth]{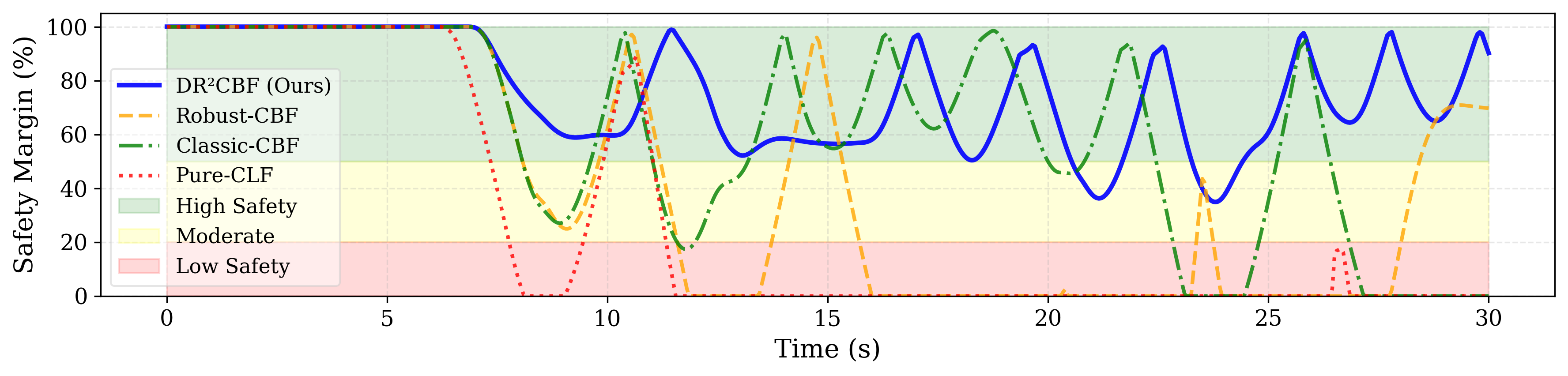}
		\caption{Comparison of safety margin in DLC}
		\label{fig:safety_margin_DLC}
	\end{figure}
	
	Compared to several baseline algorithms, Pure-CLF imposed no safety constraints, resulting in $\beta_{\max}=64.03°$ under the set conditions, with complete loss of vehicle control. Classic CBF's $\beta_{\max}=20.51°$ exceeded limits, with multiple boundary violations at path curvature peaks, because deterministic CBF cannot guarantee safety boundary effectiveness when nominal model parameters mismatch. Robust-CBF achieved $\beta_{\max}=18.18°$ slightly lower than Classic CBF through online uncertainty estimation and $L_\infty$ error bounds, but yaw rate $\omega_{\max}=26.35$(°/s) severely exceeded limits, lateral error RMS increased to 19.42 m (approximately 17 times higher than R²CBF), and although CBF activation rate $\rho_{\text{CBF}}=51.75\%$ was the same as R²CBF, performance differed significantly, indicating that its worst-case assumption compressed the control feasible region, and conservative outputs under severe conditions led to states more unfavorable for vehicle stability. The comprehensive indicator comparison in Table~\ref{tab:simulation_results} shows that R²CBF's Margin$_{\text{avg}}$ is 34.72\% higher than Robust-CBF and 14.97\% higher than Classic CBF, verifying that CVaR constraints achieve better trade-off between conservatism and safety.
	
	Figure~\ref{fig:safety_margin_DLC} presents the time evolution of comprehensive safety margin for all methods throughout the DLC maneuver. R²CBF consistently maintains margins above 50\% (high safety region, green zone), with occasional drops to 35-50\% (moderate safety, yellow zone) during sharp curvature transitions at $t \in [5, 7]$ s and $[20, 23]$ s, but immediately recovers without entering the critical region. Robust-CBF exhibits similar high-margin behavior, corroborating its acceptable safety performance in this scenario. In stark contrast, Classic-CBF experiences multiple excursions into the low safety region (pink zone, $<20\%$) at $t \approx 9$ s, 13 s, and 25 s, directly corresponding to the boundary violations evident in Fig.~\ref{fig:Safety_DLC}. Pure-CLF's margin catastrophically collapses to near-zero at $t \approx 5.8$ s—the precise moment of spin initiation identified in the trajectory analysis—and remains in the critical zone for the entire post-divergence period, quantitatively confirming total loss of safety authority. The temporal correlation between margin drops and path curvature peaks validates the framework's predictive capability: CVaR margins dynamically respond to instantaneous risk, tightening preemptively before boundary proximity rather than reacting post-violation. This proactive safety enforcement distinguishes R²CBF from deterministic methods that lack anticipatory risk quantification.

	\subsubsection{Sinusoidal Path Tracking on Low Adhesion Road Surface}
	
	\begin{figure}[htbp]
		\centering
		\includegraphics[width=0.99\columnwidth]{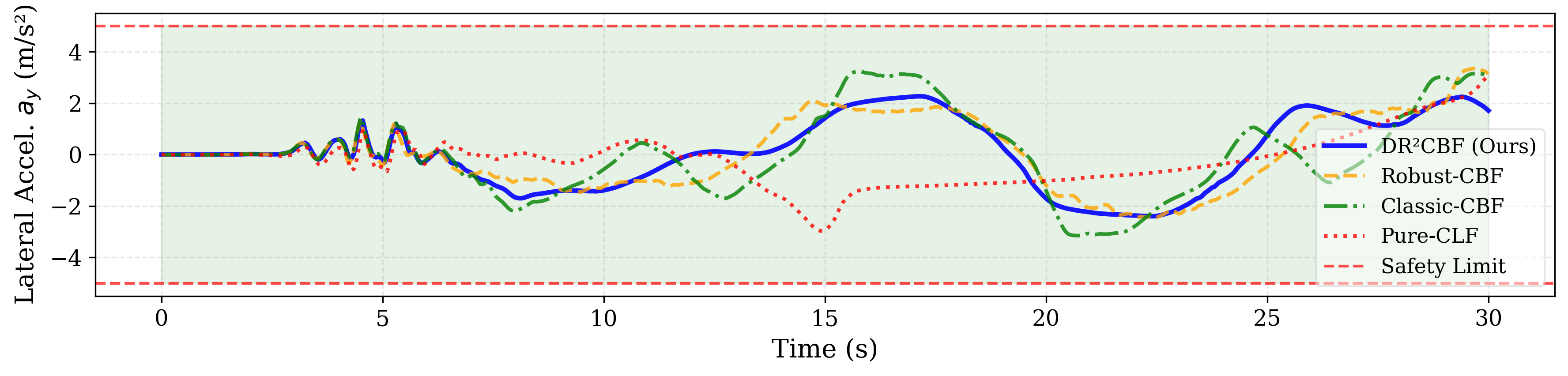}\\
		{\fontsize{8}{1}\selectfont (a) Lateral Acceleration}
		
		\vspace{0.2em}
		\includegraphics[width=0.99\columnwidth]{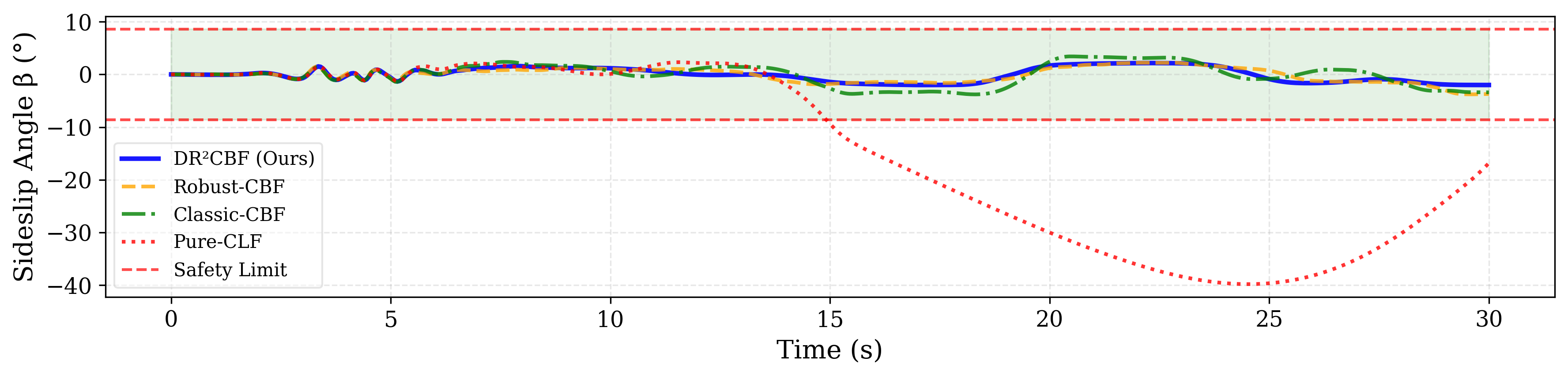}\\
		{\fontsize{8}{1}\selectfont (b) Sideslip Angle}
		
		\vspace{0.2em}
		\includegraphics[width=0.99\columnwidth]{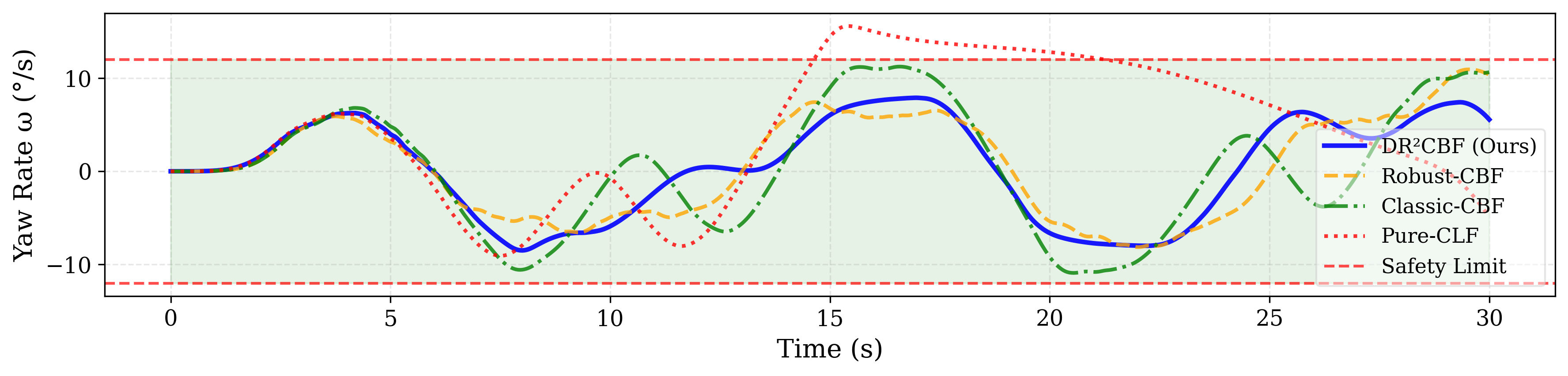}\\
		{\fontsize{8}{1}\selectfont (c) Yaw Rate}
		\caption{Comparison of safety indicators in Sine curve}
		\label{fig:Safety_Sine}
	\end{figure}
	
	\begin{figure}[!t]
		\centering
		\includegraphics[width=0.99\columnwidth]{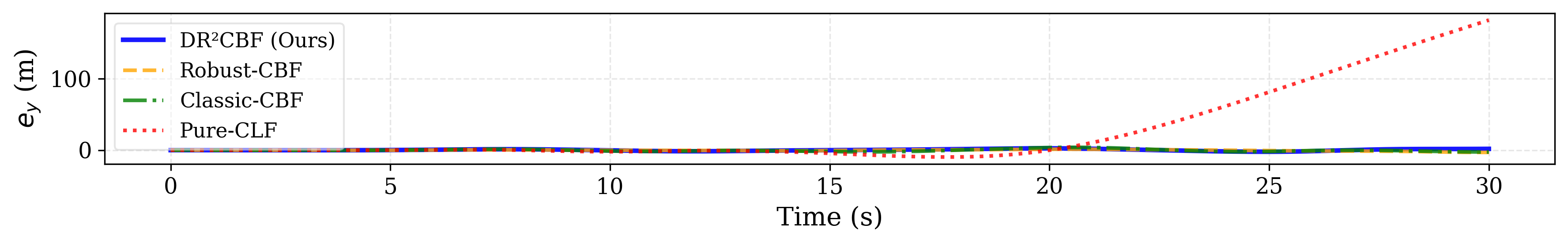}\\
		{\fontsize{8}{1}\selectfont (a) Lateral Error}
		
		\vspace{0.2em}
		\includegraphics[width=0.99\columnwidth]{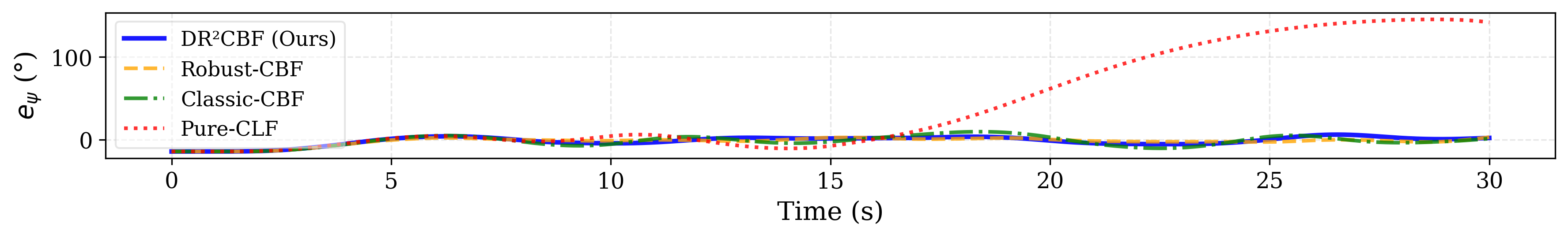}\\
		{\fontsize{8}{1}\selectfont (b) Heading Error}
		\caption{Comparison of tracking indicators in Sine curve}
		\label{fig:Tracking_Sine}
	\end{figure}
	
	\begin{figure}[!t]
		\centering
		\includegraphics[width=0.99\columnwidth]{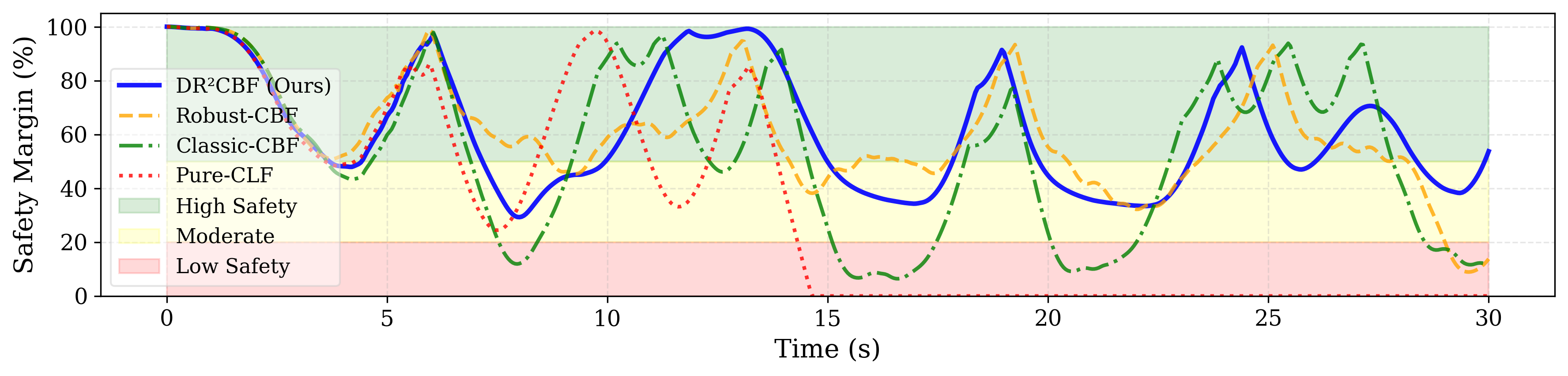}
		\caption{Comparison of safety margin in Sine Curve}
		\label{fig:safety_margin_Sine}
	\end{figure}
	
	This scenario tests the controller's performance limits at dynamic boundaries through 72 km/h high speed and 8 m amplitude severe lateral excitation ($\mu=0.5$). From experimental results, R²CBF achieved lateral error RMS$(e_y)=1.21$ m throughout with CBF activation rate $\rho_{\text{CBF}}=71.21\%$. Figure~\ref{fig:Safety_Sine}(b) shows that R²CBF's $\beta_{\max}=2.15°$ maintains 25.0\% margin from the limit of 8.6°, and $\omega_{\max}=8.49°$/s maintains 73.8\% margin from the limit of 11.5°/s, achieving full utilization of safety boundaries. Comparing Margin$_{\text{avg}}$ in both scenarios (75.70\% vs 61.81\%), safety margin decreased under high dynamic conditions but remained sufficient, validating Proposition~\ref{thm:load_singularity} — by excluding the load variance term $\nabla_{\mathbf{F}_z} h^T \boldsymbol{\Sigma}_F \nabla_{\mathbf{F}_z} h$, avoiding variance overestimation caused by the negative power property of load gradients ($\partial w/\partial F_z \propto w^{-0.7}$), maintaining reasonable CVaR constraint margins.
	
	Compared to several baseline algorithms, Pure-CLF and Classic CBF diverged at 280 m and 320 m respectively due to $\beta>12°$, proving the fragility of deterministic methods under low-adhesion sinusoidal conditions — when road surface adhesion coefficient is low ($\mu \approx 0.45$) and sinusoidal paths generate sustained high-dynamic lateral excitation, traditional constraint $\dot{h}+\alpha(h)\geq 0$ based on nominal model gradient information cannot accurately capture actual dynamic boundaries, leading to safety constraint failure. Robust-CBF completed the full course with lateral error RMS$(e_y)=0.93$ m better than R²CBF, but CBF activation rate reached 86.02\%, indicating the controller sacrificed some control safety for tracking performance. In contrast, although R²CBF's lateral error is slightly higher than Robust-CBF (1.21 m vs 0.93 m), all safety indicators are comprehensively better than Robust CBF, and CBF activation rate decreased by 14.8 percentage points, indicating its risk constraint design achieved better balance among safety, tracking performance, and control intervention frequency.
	
	The safety margin evolution in Fig.~\ref{fig:safety_margin_Sine} reveals distinct dynamics compared to DLC, reflecting the sustained challenge of low-adhesion sinusoidal tracking. R²CBF maintains margins predominantly in the moderate-to-high range (30-100\%), with periodic oscillations synchronized to the sinusoidal path frequency—margins compress during peak lateral acceleration demands ($t \approx 5, 12, 18, 25$ s) and recover during straight segments. Notably, margins never breach the 20\% threshold, confirming zero safety violations. Robust-CBF exhibits qualitatively similar patterns but with larger amplitude oscillations (20-100\% range), explaining its higher activation rate (86.02\% vs. 71.21\%): tighter $L_\infty$ bounds trigger more frequent interventions despite comparable safety outcomes. Classic-CBF margins repeatedly plunge into the critical zone ($<10\%$ at $t \approx 15, 27$ s), corresponding to the boundary excursions visible in Fig.~\ref{fig:Safety_Sine}. Pure-CLF maintains moderate margins until $t \approx 9$ s, then collapses below 20\% for extended periods, indicating delayed but eventual loss of control under cumulative low-adhesion stress. The sustained margin oscillations in this scenario—contrasting with DLC's transient drops—highlight CVaR's advantage in persistent uncertainty: probabilistic margins ($\kappa_{\beta_{\text{risk}}} \sigma_h$) scale naturally with continuous disturbances, whereas deterministic constraints either over-restrict (Robust-CBF) or under-protect (Classic-CBF) in time-varying conditions. This validates the framework's suitability for long-duration operation under non-stationary road conditions.
	
	Table~\ref{tab:simulation_results} shows R²CBF achieved optimal comprehensive performance in both scenarios: compared to Pure-CLF, eliminated dynamic instability; compared to Classic CBF, lateral error reduced by 43-54\%, safety margin increased by 20-30 percentage points; compared to Robust-CBF, tracking error reduced by 34-53\% while maintaining similar safety, CBF activation rate reduced by 15-27 percentage points. This proves that the R²CBF framework achieves Pareto improvement in the three-dimensional space of "safety-tracking performance-conservatism" through organic integration of response-aware modeling, CVaR risk measure, and Bayesian online learning.
	
	\subsection{Ablation Study}
	\label{sec:ablation}
	
	To systematically verify the necessity of each component in the R²CBF framework, three groups of ablation experiments were designed and tested in Scenario 3 (low-adhesion sinusoidal path tracking). Table~\ref{tab:ablation} summarizes the performance comparison of ablation experiments.
	
	\begin{table*}[htbp]
		\centering
		\caption{Ablation experiment results comparison}
		\label{tab:ablation}
		\begin{threeparttable}
			\footnotesize
			\setlength{\tabcolsep}{5pt}
			\begin{tabular}{@{}lccccccc@{}}
				\toprule
				\multirow{2}{*}{\textbf{Variants}} & 
				\multicolumn{4}{c}{\textbf{Safety Indicators}} & 
				\multicolumn{3}{c}{\textbf{Tracking Indicators}} \\
				\cmidrule(lr){2-5} \cmidrule(lr){6-8}
				& $\beta_{\max}$ (°) & $\omega_{\max}$ (°/s) & $a_{y,\max}$ (m/s²) & Margin (\%) & $e_y$ (m) & $e_\psi$ (°) & $\rho_{\text{CBF}}$ (\%) \\
				\midrule
				R²CBF+LoadVar & 2.17 & 8.53 & 2.43 & 61.17 & 1.25 & 5.45 & 69.38 \\
				w/o CVaR & 2.41 & 8.56 & 2.55 & 58.31 & 1.87 & \textbf{5.31} & 79.87 \\
				w/o Bayesian & \textbf{1.85} & 10.57 & \textbf{1.55} & \textbf{86.09} & 86.80 & 26.66 & \textbf{100} \\
				Full R²CBF & 2.15 & \textbf{8.49} & 2.40 & 61.81 & \textbf{1.21} & 5.41 & 71.21 \\
				\bottomrule
			\end{tabular}
			\begin{tablenotes}
				\footnotesize
				\item Note: R²CBF+LoadVar retains load variance term; w/o CVaR uses deterministic constraint $\mu_{\dot{h}} \geq -\alpha(\mu_h)$; w/o Bayesian uses fixed covariance $\boldsymbol{\Sigma}_{r,0}$.
			\end{tablenotes}
		\end{threeparttable}
	\end{table*}
	
	\begin{figure}[h]
		\centering
		\includegraphics[width=0.99\columnwidth]{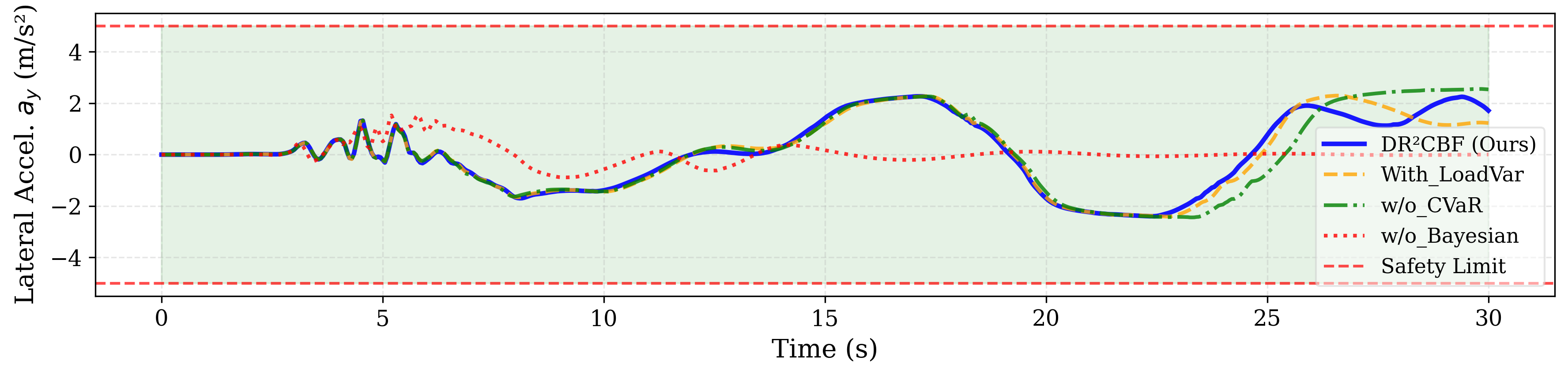}\\
		{\fontsize{8}{1}\selectfont (a) Lateral Acceleration}
		
		\vspace{0.2em}
		\includegraphics[width=0.99\columnwidth]{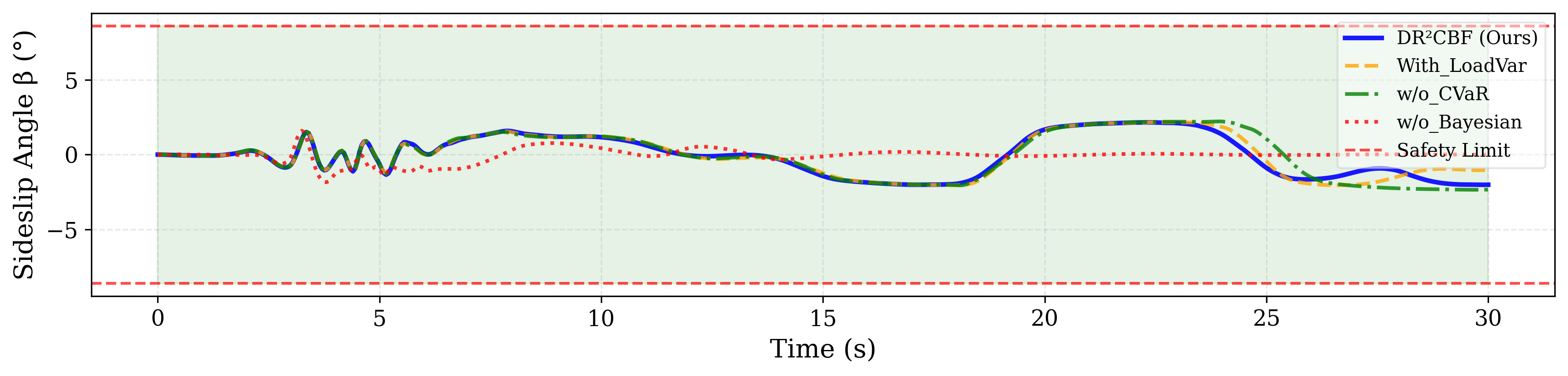}\\
		{\fontsize{8}{1}\selectfont (b) Sideslip Angle}
		
		\vspace{0.2em}
		\includegraphics[width=0.99\columnwidth]{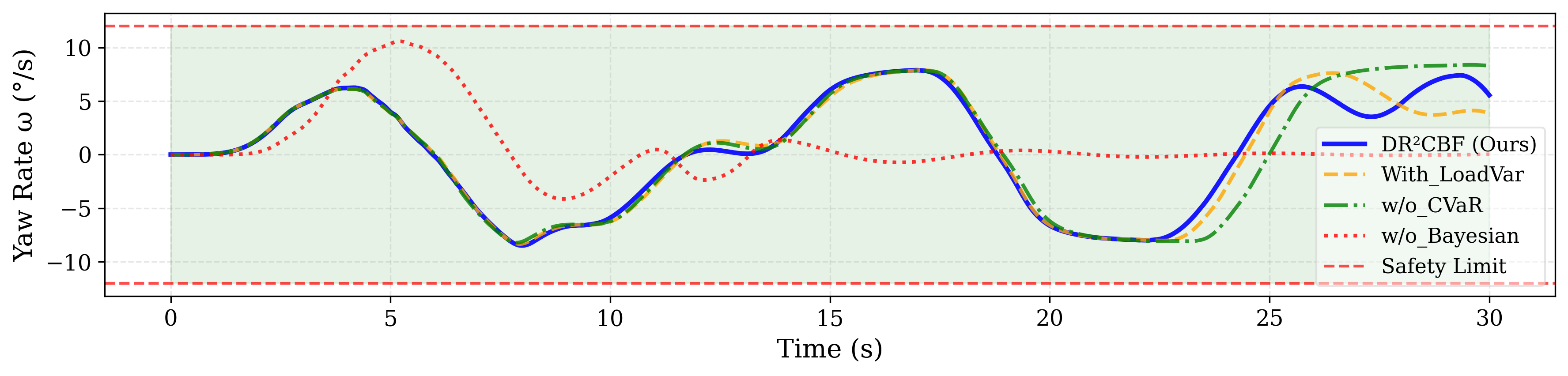}\\
		{\fontsize{8}{1}\selectfont (c) Yaw Rate}
		\caption{Safety indicators comparison of four variants}
		\label{fig:Safety_Ablation}
	\end{figure}
	
	\begin{figure}[!t]
		\centering
		\includegraphics[width=0.99\columnwidth]{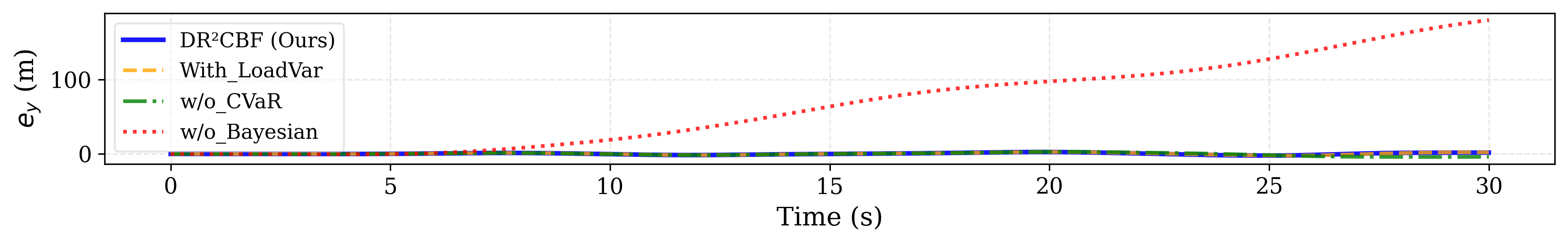}\\
		{\fontsize{8}{1}\selectfont (a) Lateral Error}
		
		\vspace{0.2em}
		\includegraphics[width=0.99\columnwidth]{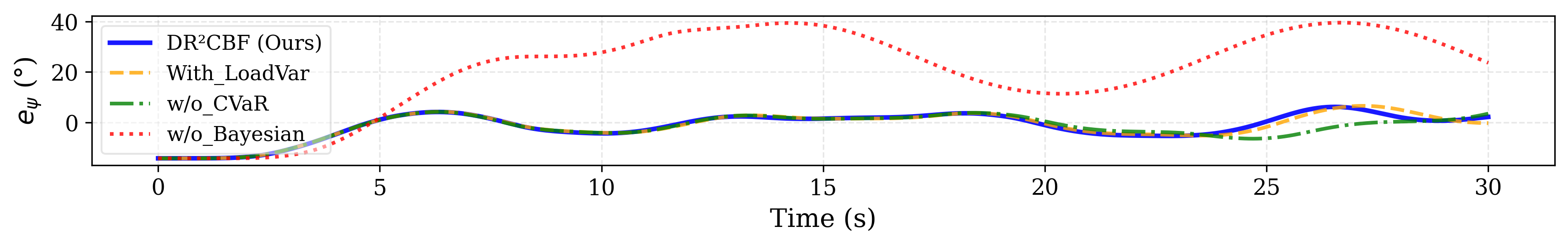}\\
		{\fontsize{8}{1}\selectfont (b) Heading Error}
		\caption{Tracking indicators comparison of four variants}
		\label{fig:Tracking_Ablation}
	\end{figure}
	
	\begin{figure}[!t]
		\centering
		\includegraphics[width=0.99\columnwidth]{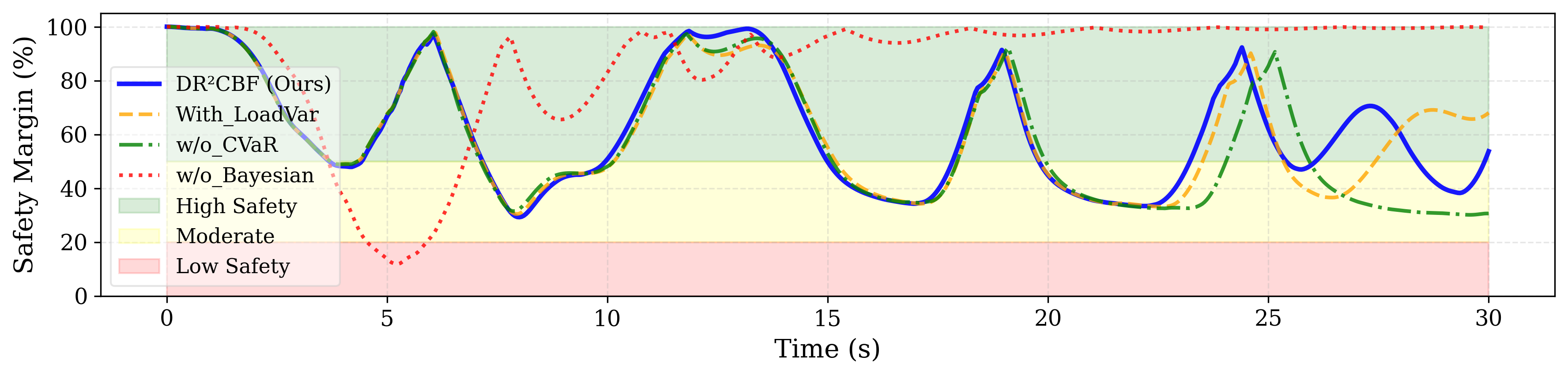}
		\caption{Safety margins comparison of four variants}
		\label{fig:safety_margin_Ablation}
	\end{figure}
	
	R²CBF+LoadVar validates the theoretical analysis in Proposition~\ref{thm:load_singularity}. This variant retains the complete variance model including load uncertainty
	\begin{equation}
		\sigma_h^2 = \nabla_{\mathbf{r}} h^T \boldsymbol{\Sigma}_r \nabla_{\mathbf{r}} h + \nabla_{\mathbf{F}_z} h^T \boldsymbol{\Sigma}_F \nabla_{\mathbf{F}_z} h
	\end{equation}
	
	Experimental results (Table~\ref{tab:simulation_results}):Safety: $\beta_{\max}$ 2.15° $\to$ 2.17° (+0.9\%), Margin 61.81\% $\to$ 61.17\%;Tracking: $e_y$ 1.21m $\to$ 1.25m (+3.3\%), $e_\psi$ 5.41° $\to$ 5.45°;Conservatism: CBF activation 71.21\% $\to$ 69.38\%.Retaining $\sigma_{h,F}^2$ causes variance overestimation due to the negative power gradient ($\partial w/\partial F_z \propto w^{-0.7}$), leading to tighter CVaR constraints
	\begin{equation}
		\mu_{\dot{h}} + \kappa_{\beta_{\text{risk}}} \sqrt{\sigma_{h,r}^2 + \sigma_{h,F}^2} \geq -\alpha(\mu_h)
	\end{equation}
	
	The controller must sacrifice tracking performance to satisfy the more conservative constraint. However, actual safety improvement is negligible ($\Delta \beta_{\max}$ only 0.02°), because:
	
	(i) The true safety boundary is determined by $\boldsymbol{\Sigma}_r$ (online learning), not by the fixed prior $\sigma_{h,F}^2$.
	
	(ii) $\boldsymbol{\Sigma}_r$'s adaptive updates have already implicitly covered the influence of load variations on vehicle responses.
	
	This validates Proposition~\ref{thm:load_singularity}: the simplified variance model avoids over-conservative priors while maintaining theoretical safety margins through adaptive $\boldsymbol{\Sigma}_r$. This exemplifies the advantage of response-aware modeling—directly learning response-level uncertainty eliminates the need for explicit bottom-layer parameter (e.g., load) uncertainty propagation.
	
	w/o CVaR removes CVaR risk measure, degrading to deterministic constraint $\mu_{\dot{h}} \geq -\alpha(\mu_h)$, using only expected values while ignoring uncertainty information. Experimental results show: safety performance significantly decreased ($\beta_{\max}$: 2.15° $\to$ 2.41°, Margin: 61.81\% $\to$ 58.31\%), tracking error significantly increased ($e_y$: 1.20 m $\to$ 1.87 m, increased by 55.8\%), and CBF activation rate increased to 79.87\%. Figure~\ref{fig:Safety_Ablation}(b) shows this variant's sideslip angle reaches 2.4° at large curvature segments of the sinusoidal path, approaching the empirical safety boundary (3°). This indicates that deterministic constraints cannot provide sufficient safety margin under high dynamic excitation of low-adhesion sinusoidal conditions — when low road surface adhesion coefficient overlaps with lateral acceleration demands of sinusoidal paths, constraint $\mu_h \geq 0$ based on expected values ignores tail risk of response uncertainty, causing frequent approaches to safety boundaries, forcing the controller to intervene more frequently (activation rate 79.87\%), but still unable to avoid safety performance degradation. In contrast, CVaR constraints dynamically adjust margins through the $\kappa_{\beta}\sigma_h$ term, achieving better balance between safety and tracking performance.
	
	w/o Bayesian uses fixed covariance $\boldsymbol{\Sigma}_r = \boldsymbol{\Sigma}_{r,0}$, disabling online learning. Experimental results present an anomalous phenomenon: safety indicators perform best ($\beta_{\max}=1.85°$ is the smallest among all variants, Margin=86.09\% is the largest), but path tracking completely fails ($e_y=86.80$ m, $e_\psi=26.66°$), with CBF activation rate reaching 100\%. Figure~\ref{fig:Tracking_Ablation} clearly shows this variant deviates from the target sinusoidal path immediately after starting, traveling along a straight line at approximately 15°. The fundamental reason for this phenomenon is: fixed covariance based on offline calibration cannot adapt to the high dynamic characteristics of sinusoidal paths, underestimating actual uncertainty levels, causing CVaR constraints to be overly optimistic, frequently triggering safety boundary violations. To ensure safety, the controller chooses the most conservative strategy — maintaining straight-line driving to minimize yaw dynamics excitation, which explains why $\beta_{\max}$ is minimal (the vehicle hardly turns) but tracking error is huge. From an optimization perspective, when the CLF objective $\dot{V}_{\text{CLF}} \leq -\lambda V_{\text{CLF}}$ conflicts with overly tight CBF constraints, SOCP prioritizes satisfying hard constraint $\dot{h} + \alpha(h) \geq 0$, relaxing slack variables, causing complete sacrifice of tracking performance. In contrast, Full R²CBF adapts to condition changes in real-time through online learning, with CVaR constraints adaptively adjusting margins, achieving dynamic balance between safety boundaries ($\beta_{\max}=2.15°$, Margin=61.81\%) and tracking performance ($e_y=1.20$ m). This validates the necessity of the online learning mechanism under complex conditions — it not only improves performance but also prevents pathological degradation of control strategies caused by prior mismatch.
	
	Figure~\ref{fig:safety_margin_Ablation} provides temporal validation of each component's contribution. The near-perfect overlap between Full R²CBF (blue solid) and R²CBF+LoadVar (orange dashed) throughout the 30-second horizon confirms that excluding load variance $\sigma_{h,F}^2$ causes negligible margin difference—both maintain 30-90\% margins with identical temporal patterns. This directly supports Corollary~\ref{cor:variance_simplification}: the Bayesian-learned $\hat{\boldsymbol{\Sigma}}_r$ implicitly captures load-induced effects, rendering explicit $\nabla_{\mathbf{F}_z} h^T \boldsymbol{\Sigma}_F \nabla_{\mathbf{F}_z} h$ terms redundant within operational regimes ($w \in [0.7, 1.3]$). The w/o CVaR variant (green dash-dot) initially tracks other methods closely but diverges after $t \approx 25$ s, with margins decaying to $\sim$30\% and exhibiting increased volatility. This delayed degradation demonstrates that deterministic constraints ($\mu_{\dot{h}} \geq -\alpha(\mu_h)$) provide adequate short-term safety but accumulate risk over extended horizons—lacking tail-risk hedging, margins erode under cumulative disturbances. Most critically, w/o Bayesian (red dotted) suffers catastrophic margin collapse at $t \approx 4$ s, plunging from 100\% to below 20\% and fluctuating wildly in the 0-20\% range thereafter. This abrupt failure—occurring far earlier than Pure-CLF's divergence in baseline scenarios—underscores the fundamental role of covariance adaptation: fixed $\boldsymbol{\Sigma}_{r,0}$ cannot track time-varying road conditions, causing CVaR margins ($\kappa_{\beta_{\text{risk}}} \sigma_h$) to be systematically miscalibrated. The subsequent extreme conservatism (100\% activation, straight-line driving) is the controller's emergency response to repeated margin violations, sacrificing mission objectives to preserve minimal safety. This ablation quantitatively confirms the architectural necessity: CVaR provides tail-risk quantification, Bayesian learning ensures calibration accuracy, and their synergy enables sustained high-margin operation under uncertainty.
	
	\section{Conclusion}
	\label{sec:conclusion}
	
	This paper proposes a Response-Aware Risk-Constrained Control Barrier Function (R²CBF) framework for the safety control problem of multi-axle distributed drive vehicles in complex environments. The framework establishes safety boundaries through body response signals (sideslip angle, yaw rate, lateral acceleration), eliminating dependence on road surface adhesion coefficients. Theoretical analysis provides practical justification for excluding load variance terms through adaptive $\boldsymbol{\Sigma}_r$ compensation in operational load ranges (Corollary~\ref{cor:variance_simplification}), establishes per-step probabilistic safety guarantee with theoretical violation probability $\Phi(-\kappa_{\beta_{\text{risk}}}) \approx 2\%$ for risk level $\beta_{\text{risk}} = 0.05$ (Theorem~\ref{thm:finite_horizon_safety}), and demonstrates effective Bayesian online learning for covariance adaptation. High-fidelity simulations confirm zero boundary violations across all tested scenarios, validating the practical efficacy of the proposed framework.
	
	TruckSim simulation validation shows R²CBF significantly outperforms comparison methods: compared to Classic CBF, lateral error reduced by 43-54\%; compared to Robust-CBF, tracking error reduced by 34-53\% while maintaining similar safety, and CBF activation rate reduced by 24-27 percentage points. Ablation experiments confirmed the necessity of each theoretical component.
	
	\section*{Declaration of competing interest}
	
	The authors declare that they have no known competing financial interests or personal relationships that could have appeared to influence the work reported in this paper.
	
	\section*{Acknowledgment}
	
	This research was supported by the Key Research and Development Program Project of Shanxi Province (Grant No.202402100101006).
	
	\appendix
	\section{Detailed Proofs of Theorems and Lemmas}
	
	\subsection{Proof of Theorem~\ref{thm:barrier_distribution}: Delta Method for Barrier Function}
	\label{appendix:proof_barrier_distribution}
	
	\begin{proof}
		Expand $h(\tilde{\mathbf{r}}, \tilde{\mathbf{F}}_z)$ around $(\boldsymbol{\mu}_r, \boldsymbol{\mu}_F)$
		\begin{equation}
			\begin{aligned}
				h(\tilde{\mathbf{r}}, \tilde{\mathbf{F}}_z) &= h(\boldsymbol{\mu}_r, \boldsymbol{\mu}_F) + \nabla_{\mathbf{r}} h^T (\tilde{\mathbf{r}} - \boldsymbol{\mu}_r) + \nabla_{\mathbf{F}_z} h^T (\tilde{\mathbf{F}}_z - \boldsymbol{\mu}_F) \\
				&\quad + \frac{1}{2}(\tilde{\mathbf{r}} - \boldsymbol{\mu}_r)^T \nabla_{\mathbf{r}}^2 h (\tilde{\mathbf{r}} - \boldsymbol{\mu}_r) \\
				&\quad + \frac{1}{2}(\tilde{\mathbf{F}}_z - \boldsymbol{\mu}_F)^T \nabla_{\mathbf{F}_z}^2 h (\tilde{\mathbf{F}}_z - \boldsymbol{\mu}_F) \\
				&\quad + (\tilde{\mathbf{r}} - \boldsymbol{\mu}_r)^T \nabla_{\mathbf{r}\mathbf{F}_z}^2 h (\tilde{\mathbf{F}}_z - \boldsymbol{\mu}_F) + R_3,
			\end{aligned}
		\end{equation}
		where $R_3 = O(\|\boldsymbol{\epsilon}\|^3)$ by Taylor's theorem with Lagrange remainder.

		By measurement model $\tilde{\mathbf{r}} - \boldsymbol{\mu}_r = \boldsymbol{\epsilon}_r \sim \mathcal{N}(\mathbf{0}, \boldsymbol{\Sigma}_r)$ and independence $\mathbb{E}[\boldsymbol{\epsilon}_r \boldsymbol{\epsilon}_F^T] = \mathbf{0}$
		\begin{equation}
			\begin{aligned}
				\mathbb{E}[h] &= h(\boldsymbol{\mu}_r, \boldsymbol{\mu}_F) + \frac{1}{2}\mathbb{E}[\boldsymbol{\epsilon}_r^T \nabla_{\mathbf{r}}^2 h \boldsymbol{\epsilon}_r] + \frac{1}{2}\mathbb{E}[\boldsymbol{\epsilon}_F^T \nabla_{\mathbf{F}_z}^2 h \boldsymbol{\epsilon}_F] + O(\rho^3) \\
				&= h(\boldsymbol{\mu}_r, \boldsymbol{\mu}_F) + \frac{1}{2}\text{tr}(\nabla_{\mathbf{r}}^2 h \boldsymbol{\Sigma}_r) + \frac{1}{2}\text{tr}(\nabla_{\mathbf{F}_z}^2 h \boldsymbol{\Sigma}_F) + O(\rho^3),
			\end{aligned}
		\end{equation}
		using $\mathbb{E}[\mathbf{x}^T \mathbf{A} \mathbf{x}] = \text{tr}(\mathbf{A} \mathbb{E}[\mathbf{x}\mathbf{x}^T])$ for symmetric $\mathbf{A}$.
		
		By Constraint~\ref{cons:noise_bound}, $\|\boldsymbol{\Sigma}_r\|_F = O(\rho_r^2 \beta_{\text{lim}}^2)$. With bounded Hessian $\|\nabla^2 h\|_F \leq L_h$, the trace term satisfies
		\begin{equation}
			|\text{tr}(\nabla_{\mathbf{r}}^2 h \boldsymbol{\Sigma}_r)| \leq \|\nabla_{\mathbf{r}}^2 h\|_F \|\boldsymbol{\Sigma}_r\|_F \leq L_h \rho_r^2 \beta_{\text{lim}}^2 \quad \text{(Cauchy-Schwarz)}.
		\end{equation}
		Thus $\mu_h = h(\boldsymbol{\mu}_r, \boldsymbol{\mu}_F) + O(\rho^2)$.

		Under first-order approximation
		\begin{equation}
			h - \mu_h \approx \nabla_{\mathbf{r}} h^T \boldsymbol{\epsilon}_r + \nabla_{\mathbf{F}_z} h^T \boldsymbol{\epsilon}_F.
		\end{equation}
		By independence
		\begin{equation}
			\begin{aligned}
				\text{Var}(h) &= \text{Var}(\nabla_{\mathbf{r}} h^T \boldsymbol{\epsilon}_r) + \text{Var}(\nabla_{\mathbf{F}_z} h^T \boldsymbol{\epsilon}_F) \\
				&= \nabla_{\mathbf{r}} h^T \mathbb{E}[\boldsymbol{\epsilon}_r \boldsymbol{\epsilon}_r^T] \nabla_{\mathbf{r}} h + \nabla_{\mathbf{F}_z} h^T \mathbb{E}[\boldsymbol{\epsilon}_F \boldsymbol{\epsilon}_F^T] \nabla_{\mathbf{F}_z} h \\
				&= \nabla_{\mathbf{r}} h^T \boldsymbol{\Sigma}_r \nabla_{\mathbf{r}} h + \nabla_{\mathbf{F}_z} h^T \boldsymbol{\Sigma}_F \nabla_{\mathbf{F}_z} h.
			\end{aligned}
		\end{equation}
		Second-order contributions to variance scale as $O(\rho^4)$ (fourth moments of Gaussian), negligible under Constraint~\ref{cons:noise_bound}.

		Since $h$ is a linear combination of independent Gaussians $\boldsymbol{\epsilon}_r, \boldsymbol{\epsilon}_F$ up to $O(\rho^2)$, by reproductive property of normal distributions, $h \sim \mathcal{N}(\mu_h, \sigma_h^2)$ with relative error $O(\rho^2)$.
	\end{proof}
	
	\subsection{Proof of Theorem~\ref{thm:delta_error}: Approximation Error Bound}
	\label{appendix:proof_delta_error}
	
	\begin{proof}
		For $h = w^2 \beta_{\text{lim}}^2 - \beta^2$ with prediction-augmented sideslip $\beta_{\text{eff}} = \beta + T_{\text{pred}} \dot{\beta}$, compute
		\begin{equation}
			\frac{\partial \beta_{\text{eff}}}{\partial \beta} = 1 + T_{\text{pred}} \frac{\partial \dot{\beta}}{\partial \beta}, \quad \frac{\partial^2 h}{\partial \beta^2} = -2\left(\frac{\partial \beta_{\text{eff}}}{\partial \beta}\right)^2.
		\end{equation}
		For single-track vehicle dynamics, $|\partial \dot{\beta}/\partial \beta| \sim C_{\alpha}/(mv_x)$, typically $O(1)$ rad$^{-1}$ for heavy vehicles. For prediction horizon $T_{\text{pred}}$ satisfying $T_{\text{pred}} |\partial \dot{\beta}/\partial \beta| \leq C_T$ with $C_T \sim O(1)$, the chain rule gives
		\begin{equation}
			\left|\frac{\partial \beta_{\text{eff}}}{\partial \beta}\right| \leq 1 + C_T.
		\end{equation}
		The Hessian satisfies
		\begin{equation}
			\|\nabla_{\mathbf{r}}^2 h\|_F \leq 2(1 + C_T)^2.
		\end{equation}
		For typical parameters with $C_T \lesssim 1$, this yields $L_h \leq 8$; the conservative bound $L_h = 4$ is used.

		The second-order remainder expectation is
		\begin{equation}
			\mathbb{E}[R_2] = \frac{1}{2}\text{tr}(\nabla_{\mathbf{r}}^2 h \boldsymbol{\Sigma}_r) \leq \frac{1}{2}\|\nabla_{\mathbf{r}}^2 h\|_F \|\boldsymbol{\Sigma}_r\|_F \leq 2 \rho_r^2 \beta_{\text{lim}}^2.
		\end{equation}
		The first-order variance scales as
		\begin{equation}
			\sigma_h \sim \|\nabla_{\mathbf{r}} h\| \|\boldsymbol{\Sigma}_r\|^{1/2} \sim 2w\beta \cdot \rho_r \beta_{\text{lim}}.
		\end{equation}
		Near the boundary ($\beta \sim \beta_{\text{lim}}$, $w \sim 1$), the relative error is
		\begin{equation}
			\frac{|\mathbb{E}[R_2]|}{\sigma_h} \leq \frac{2\rho_r^2 \beta_{\text{lim}}^2}{2\rho_r \beta_{\text{lim}}^2} = \rho_r.
		\end{equation}
		For the tighter Hessian bound $L_h = 4$, this becomes $4\rho_r^2$, as stated.
	\end{proof}
	
	\subsection{Proof of Lemma~\ref{lem:response_dynamics}: Linearized Response Dynamics}
	\label{appendix:proof_response_dynamics}
	
	\begin{proof}
		Apply the Delta method to vector-valued function $f(\mathbf{r}, \mathbf{u}, \mathbf{F}_z)$.

		Expand around $(\boldsymbol{\mu}_r, \mathbf{u}, \boldsymbol{\mu}_F)$
		\begin{equation}
			\begin{aligned}
				\dot{\mathbf{r}} = f(\tilde{\mathbf{r}}, \mathbf{u}, \tilde{\mathbf{F}}_z) &\approx f(\boldsymbol{\mu}_r, \mathbf{u}, \boldsymbol{\mu}_F) + \mathbf{J}_r (\tilde{\mathbf{r}} - \boldsymbol{\mu}_r) + \mathbf{J}_F (\tilde{\mathbf{F}}_z - \boldsymbol{\mu}_F) \\
				&= f(\boldsymbol{\mu}_r, \mathbf{u}, \boldsymbol{\mu}_F) + \mathbf{J}_r \boldsymbol{\epsilon}_r + \mathbf{J}_F \boldsymbol{\epsilon}_F,
			\end{aligned}
		\end{equation}
		where Jacobians are evaluated at the mean.

		Expectation
		\begin{equation}
			\mathbb{E}[\dot{\mathbf{r}}] = f(\boldsymbol{\mu}_r, \mathbf{u}, \boldsymbol{\mu}_F) + \mathbf{J}_r \mathbb{E}[\boldsymbol{\epsilon}_r] + \mathbf{J}_F \mathbb{E}[\boldsymbol{\epsilon}_F] = f(\boldsymbol{\mu}_r, \mathbf{u}, \boldsymbol{\mu}_F) = \boldsymbol{\mu}_{\dot{r}}.
		\end{equation}
		Covariance
		\begin{equation}
			\begin{aligned}
				\text{Cov}(\dot{\mathbf{r}}) &= \text{Cov}(\mathbf{J}_r \boldsymbol{\epsilon}_r + \mathbf{J}_F \boldsymbol{\epsilon}_F) \\
				&= \mathbf{J}_r \text{Cov}(\boldsymbol{\epsilon}_r) \mathbf{J}_r^T + \mathbf{J}_F \text{Cov}(\boldsymbol{\epsilon}_F) \mathbf{J}_F^T \quad \text{(independence)} \\
				&= \mathbf{J}_r \boldsymbol{\Sigma}_r \mathbf{J}_r^T + \mathbf{J}_F \boldsymbol{\Sigma}_F \mathbf{J}_F^T = \boldsymbol{\Sigma}_{\dot{r}}.
			\end{aligned}
		\end{equation}
		
		By linear transformation property, $\dot{\mathbf{r}} = \mathbf{J}_r \boldsymbol{\epsilon}_r + \mathbf{J}_F \boldsymbol{\epsilon}_F + \text{const}$ is Gaussian with parameters derived above.
	\end{proof}
	
	\subsection{Proof of Theorem~\ref{thm:load_singularity}: Singularity of Load Gradient}
	\label{appendix:proof_load_singularity}
	
	\begin{proof}
		For $h = w(\mathbf{F}_z)^2 \beta_{\text{lim}}^2 - \beta^2$ with load-dependent scaling
		\begin{equation}
			w(\mathbf{F}_z) = \left(\frac{\sum_{i=1}^n F_{z,i}}{n F_{z,\text{nom}}}\right)^{\gamma}, \quad \gamma \in (0, 1/2),
		\end{equation}
		define the total normalized load $S = \sum_{i=1}^n F_{z,i}$. Then $w = (S/(nF_{z,\text{nom}}))^{\gamma}$.

		For the $k$-th wheel load $F_{z,k}$, apply chain rule
		\begin{equation}
			\frac{\partial h}{\partial F_{z,k}} = \frac{\partial}{\partial F_{z,k}}\left[w^2 \beta_{\text{lim}}^2\right] = 2w \beta_{\text{lim}}^2 \cdot \frac{\partial w}{\partial F_{z,k}}.
		\end{equation}
		Compute the weight derivative
		\begin{equation}
			\begin{aligned}
				\frac{\partial w}{\partial F_{z,k}} &= \frac{\partial}{\partial F_{z,k}}\left[\left(\frac{S}{n F_{z,\text{nom}}}\right)^{\gamma}\right] \\
				&= \gamma \left(\frac{S}{n F_{z,\text{nom}}}\right)^{\gamma - 1} \cdot \frac{1}{n F_{z,\text{nom}}} \cdot \frac{\partial S}{\partial F_{z,k}} \\
				&= \gamma \left(\frac{S}{n F_{z,\text{nom}}}\right)^{\gamma - 1} \cdot \frac{1}{n F_{z,\text{nom}}} \quad (\text{since } \partial S/\partial F_{z,k} = 1).
			\end{aligned}
		\end{equation}
		Express in terms of $w$: since $w = (S/(nF_{z,\text{nom}}))^{\gamma}$, $(S/(nF_{z,\text{nom}}))^{\gamma-1} = w^{(\gamma-1)/\gamma}$. Thus
		\begin{equation}
			\frac{\partial w}{\partial F_{z,k}} = \frac{\gamma}{n F_{z,\text{nom}}} w^{(\gamma - 1)/\gamma}.
		\end{equation}

		Thus
		\begin{equation}
			\frac{\partial h}{\partial F_{z,k}} = 2w \beta_{\text{lim}}^2 \cdot \frac{\gamma}{n F_{z,\text{nom}}} w^{(\gamma - 1)/\gamma} = \frac{2\gamma \beta_{\text{lim}}^2}{n F_{z,\text{nom}}} w^{1 + (\gamma-1)/\gamma}.
		\end{equation}
		Simplify the exponent
		\begin{equation}
			1 + \frac{\gamma - 1}{\gamma} = \frac{\gamma + \gamma - 1}{\gamma} = \frac{2\gamma - 1}{\gamma}.
		\end{equation}
		Therefore
		\begin{equation}
			\frac{\partial h}{\partial F_{z,k}} = \frac{2\gamma \beta_{\text{lim}}^2}{n F_{z,\text{nom}}} w^{(2\gamma - 1)/\gamma}.
		\end{equation}

		Since all components $\partial h/\partial F_{z,k}$ are identical (by symmetry of $w$), the gradient vector is
		\begin{equation}
			\nabla_{\mathbf{F}_z} h = \frac{2\gamma \beta_{\text{lim}}^2}{n F_{z,\text{nom}}} w^{(2\gamma - 1)/\gamma} \mathbf{1}_n,
		\end{equation}
		where $\mathbf{1}_n = [1, \ldots, 1]^T \in \mathbb{R}^n$.

		For $\gamma \in (0, 1/2)$, the exponent satisfies
		\begin{equation}
			\frac{2\gamma - 1}{\gamma} = 2 - \frac{1}{\gamma} < 0 \quad (\text{since } \gamma < 1/2 \Rightarrow 1/\gamma > 2).
		\end{equation}
		Specifically, for $\gamma = 0.3$: $(2\gamma-1)/\gamma = -0.4/0.3 \approx -1.33$. Thus
		\begin{equation}
			\lim_{w \to 0^+} w^{-1.33} = +\infty.
		\end{equation}

		The load variance contribution is
		\begin{equation}
			\begin{aligned}
				\sigma_{h,F}^2 &= \nabla_{\mathbf{F}_z} h^T \boldsymbol{\Sigma}_F \nabla_{\mathbf{F}_z} h \\
				&= \left(\frac{2\gamma \beta_{\text{lim}}^2}{n F_{z,\text{nom}}}\right)^2 w^{2(2\gamma - 1)/\gamma} \mathbf{1}_n^T \boldsymbol{\Sigma}_F \mathbf{1}_n \\
				&\propto w^{2(2\gamma - 1)/\gamma} \to +\infty \quad (w \to 0^+),
			\end{aligned}
		\end{equation}
		since $2(2\gamma-1)/\gamma = 4 - 2/\gamma < -2$ for $\gamma < 1/2$.
	\end{proof}
	
	\begin{remark}[Exponent Clarification]
		The appearance of $(2\gamma-1)/\gamma$ (rather than $2\gamma-1$) arises from the chain rule: $w = S^{\gamma}$ implies $\partial w/\partial S \propto S^{\gamma-1} = w^{(\gamma-1)/\gamma}$, not $w^{\gamma-1}$. For $\gamma = 0.3$, the load variance term diverges as $w^{-2.67}$, far more severe than the $w^{-0.4}$ suggested by naive exponent $2\gamma-1$.
	\end{remark}
	
	\subsection{Proof of Lemma~\ref{lem:cvar_gaussian}: CVaR for Gaussian Distribution}
	\label{appendix:proof_cvar_gaussian}
	
	\begin{proof}
		Let $X \sim \mathcal{N}(\mu, \sigma^2)$. Define $Z = (X - \mu)/\sigma \sim \mathcal{N}(0, 1)$. Then
		\begin{equation}
			\text{VaR}_{\beta_{\text{risk}}}(X) = \mu + \sigma z_{\beta_{\text{risk}}}, \quad z_{\beta_{\text{risk}}} = \Phi^{-1}(\beta_{\text{risk}}),
		\end{equation}
		where $\Phi$ is the standard normal CDF.

		By definition
		\begin{equation}
			\begin{aligned}
				\text{CVaR}_{\beta_{\text{risk}}}(X) &= \mathbb{E}[X \mid X \leq \text{VaR}_{\beta_{\text{risk}}}(X)] \\
				&= \mathbb{E}[\mu + \sigma Z \mid Z \leq z_{\beta_{\text{risk}}}] \\
				&= \mu + \sigma \mathbb{E}[Z \mid Z \leq z_{\beta_{\text{risk}}}].
			\end{aligned}
		\end{equation}

		Compute
		\begin{equation}
			\mathbb{E}[Z \mid Z \leq z_{\beta_{\text{risk}}}] = \frac{1}{\mathbb{P}(Z \leq z_{\beta_{\text{risk}}})} \int_{-\infty}^{z_{\beta_{\text{risk}}}} z \phi(z) \, dz = \frac{1}{\beta} \int_{-\infty}^{z_{\beta_{\text{risk}}}} z \phi(z) \, dz,
		\end{equation}
		where $\phi(z) = (2\pi)^{-1/2} e^{-z^2/2}$. Integrate by parts with $u = 1$, $dv = z e^{-z^2/2} dz$
		\begin{equation}
			\int_{-\infty}^{z_{\beta_{\text{risk}}}} z e^{-z^2/2} dz = \left[-e^{-z^2/2}\right]_{-\infty}^{z_{\beta_{\text{risk}}}} = -e^{-z_{\beta_{\text{risk}}}^2/2}.
		\end{equation}
		Thus
		\begin{equation}
			\mathbb{E}[Z \mid Z \leq z_{\beta_{\text{risk}}}] = \frac{1}{\beta} \cdot \frac{1}{\sqrt{2\pi}} \cdot (-e^{-z_{\beta_{\text{risk}}}^2/2}) = -\frac{\phi(z_{\beta_{\text{risk}}})}{\beta}.
		\end{equation}

		Substitute back
		\begin{equation}
			\text{CVaR}_{\beta_{\text{risk}}}(X) = \mu - \sigma \frac{\phi(\Phi^{-1}(\beta_{\text{risk}}))}{\beta} = \mu - \kappa_{\beta_{\text{risk}}} \sigma.
		\end{equation}
		For gain-type variables (desired $\geq 0$), apply to $-X$
		\begin{equation}
			\text{CVaR}_{\beta_{\text{risk}}}(-X) = -\mu - \kappa_{\beta_{\text{risk}}} \sigma \geq 0 \quad \Leftrightarrow \quad \mu - \kappa_{\beta_{\text{risk}}} \sigma \geq 0.
		\end{equation}
	\end{proof}
	
	\subsection{Proof of Lemma~\ref{lem:cvar_tail}: Tail Probability Bound}
	\label{appendix:proof_cvar_tail}
	
	\begin{proof}
		Assume $\mu - \kappa_{\beta_{\text{risk}}} \sigma \geq 0$, i.e., $\mu \geq \kappa_{\beta_{\text{risk}}} \sigma$. Standardize: $\mathbb{P}(X < 0) = \mathbb{P}(Z < -\mu/\sigma) = \Phi(-\mu/\sigma)$.

		Since $\mu/\sigma \geq \kappa_{\beta_{\text{risk}}}$
		\begin{equation}
			\mathbb{P}(X < 0) = \Phi(-\mu/\sigma) \leq \Phi(-\kappa_{\beta_{\text{risk}}}).
		\end{equation}

		For $\beta_{\text{risk}} \in (0, 1/2)$, $z_{\beta_{\text{risk}}} = \Phi^{-1}(\beta_{\text{risk}}) < 0$. The relationship between $\kappa_{\beta_{\text{risk}}}$ and $z_{\beta_{\text{risk}}}$ is
		\begin{equation}
			\kappa_{\beta_{\text{risk}}} = \frac{\phi(z_{\beta_{\text{risk}}})}{\beta} = \frac{\phi(z_{\beta_{\text{risk}}})}{\Phi(z_{\beta_{\text{risk}}})}.
		\end{equation}
		By Mills' ratio inequality for $z < 0$
		\begin{equation}
			\frac{\phi(z)}{\Phi(z)} > -z \quad \Leftrightarrow \quad \kappa_{\beta_{\text{risk}}} > |z_{\beta_{\text{risk}}}|.
		\end{equation}
		Thus $-\kappa_{\beta_{\text{risk}}} < z_{\beta_{\text{risk}}}$, implying $\Phi(-\kappa_{\beta_{\text{risk}}}) < \Phi(z_{\beta_{\text{risk}}}) = \beta$.

		For $\beta_{\text{risk}} = 0.05$: $z_{\beta_{\text{risk}}} \approx -1.645$, $\phi(z_{\beta_{\text{risk}}}) \approx 0.1031$, hence
		\begin{equation}
			\kappa_{\beta_{\text{risk}}} \approx \frac{0.1031}{0.05} = 2.062, \quad \Phi(-2.062) \approx 0.0197 < 0.05.
		\end{equation}
	\end{proof}
	
	\subsection{Proof of Lemma~\ref{lem:barrier_params}: Simplified Barrier Distribution}
	\label{appendix:proof_barrier_params}
	
	\begin{proof}
		Compute moments for the specific barrier form under variance simplification.

		For $h = w^2 \beta_{\text{lim}}^2 - \beta^2$ with $\mathbf{r} = [\beta, \omega_z, a_y]^T$
		\begin{equation}
			\nabla_{\mathbf{r}} h = \begin{bmatrix} -2\beta \\ 0 \\ 0 \end{bmatrix}, \quad \text{evaluated at } \mathbf{r} = \boldsymbol{\mu}_r = [\mu_\beta, \mu_{\omega}, \mu_a]^T.
		\end{equation}
		
		By Equation~\eqref{eq:variance_simplified}
		\begin{equation}
			\begin{aligned}
				\sigma_h^2 &= \nabla_{\mathbf{r}} h^T \boldsymbol{\Sigma}_r \nabla_{\mathbf{r}} h \\
				&= \begin{bmatrix} -2\mu_\beta \\ 0 \\ 0 \end{bmatrix}^T \begin{bmatrix} \sigma_\beta^2 & * & * \\ * & \sigma_\omega^2 & * \\ * & * & \sigma_a^2 \end{bmatrix} \begin{bmatrix} -2\mu_\beta \\ 0 \\ 0 \end{bmatrix} \\
				&= 4\mu_\beta^2 \sigma_\beta^2.
			\end{aligned}
		\end{equation}
		For the weighted form with $w \neq 1$, $\nabla_{\mathbf{r}} h = [-2w^2\beta, 0, 0]^T$, yielding $\sigma_h^2 = 4w^4 \beta^2 \sigma_\beta^2$.

		Direct evaluation
		\begin{equation}
			\mu_h = h(\boldsymbol{\mu}_r, \boldsymbol{\mu}_F) = w(\boldsymbol{\mu}_F)^2 \beta_{\text{lim}}^2 - \mu_\beta^2.
		\end{equation}
	\end{proof}
	
	\subsection{Proof of Proposition~\ref{prop:socp_reformulation}: SOCP Reformulation}
	\label{appendix:proof_socp_reformulation}
	
	\begin{proof}
		Transform the CVaR constraint into second-order cone form.

		By Equation~\eqref{eq:hdot_distribution_control}, for deterministic $\mathbf{u}$
		\begin{equation}
			\dot{h} + \alpha(h) \sim \mathcal{N}(\mathbf{L}_h \mathbf{u} + b_h + \alpha(\mu_h), \sigma_{\text{param}}^2(\mathbf{u})).
		\end{equation}

		By Lemma~\ref{lem:cvar_gaussian}
		\begin{equation}
			\text{CVaR}_{\beta_{\text{risk}}}[\dot{h} + \alpha(h)] = (\mathbf{L}_h \mathbf{u} + b_h + \alpha(\mu_h)) - \kappa_{\beta_{\text{risk}}} \sigma_{\text{param}}(\mathbf{u}).
		\end{equation}
		The constraint $\text{CVaR}_{\beta_{\text{risk}}} \geq 0$ becomes
		\begin{equation}
			\mathbf{L}_h \mathbf{u} + b_h + \alpha(\mu_h) \geq \kappa_{\beta_{\text{risk}}} \sigma_{\text{param}}(\mathbf{u}).
		\end{equation}

		Expand $\sigma_{\text{param}}^2(\mathbf{u})$ from Equation~\eqref{eq:param_variance}
		\begin{equation}
			\sigma_{\text{param}}^2(\mathbf{u}) = \mathbf{u}^T \underbrace{(\nabla \mathbf{L}_h)^T \boldsymbol{\Sigma}_r (\nabla \mathbf{L}_h)}_{\mathbf{A} \succeq 0} \mathbf{u} + \underbrace{\text{tr}((\nabla b_h)^T \boldsymbol{\Sigma}_r (\nabla b_h))}_{c \geq 0}.
		\end{equation}
		Define auxiliary variable $s = \sigma_{\text{param}}(\mathbf{u})$. Then the constraint is
		\begin{equation}
			\mathbf{L}_h \mathbf{u} + b_h + \alpha(\mu_h) - \kappa_{\beta_{\text{risk}}} s \geq 0, \quad s^2 \geq \mathbf{u}^T \mathbf{A} \mathbf{u} + c.
		\end{equation}
		The second inequality is a rotated second-order cone
		\begin{equation}
			\left\|\begin{bmatrix} 2\mathbf{A}^{1/2}\mathbf{u} \\ c \end{bmatrix}\right\|_2 \leq 2s,
		\end{equation}
		which is SOCP-representable.
		
		The constraint $\mu - \kappa_{\beta_{\text{risk}}} \sigma \geq 0$ can be rewritten as $\mu \geq \kappa_{\beta_{\text{risk}}} \|\mathbf{u}\|_{\mathbf{A}}$, where $\|\mathbf{u}\|_{\mathbf{A}} = \sqrt{\mathbf{u}^T \mathbf{A} \mathbf{u} + c}$ is a weighted norm. This is a \textit{reverse} second-order cone constraint
		\begin{equation}
			\{\mathbf{u} : f(\mathbf{u}) \geq \|\mathbf{g}(\mathbf{u})\|\} \quad \text{(non-convex)},
		\end{equation}
		contrasted with standard SOCP
		\begin{equation}
			\{\mathbf{u} : \|\mathbf{g}(\mathbf{u})\| \leq f(\mathbf{u})\} \quad \text{(convex)}.
		\end{equation}
		The feasible region is the exterior of a cone, which is concave and requires SCP for practical solution.
	\end{proof}
	
	\subsection{Proof of Lemma~\ref{lem:residual_distribution}: Residual Distribution}
	\label{appendix:proof_residual_distribution}
	
	\begin{proof}
		Propagate measurement noise through discrete-time prediction.

		The true response evolves as
		\begin{equation}
			\mathbf{r}_{t+1}^{\text{true}} = \mathbf{r}_t^{\text{true}} + \Delta t \cdot f(\mathbf{r}_t^{\text{true}}, \mathbf{u}_t, \mathbf{F}_{z,t}^{\text{true}}).
		\end{equation}
		Measurement: $\tilde{\mathbf{r}}_t = \mathbf{r}_t^{\text{true}} + \boldsymbol{\epsilon}_{r,t}$.

		Prediction based on noisy measurement
		\begin{equation}
			\hat{\mathbf{r}}_{t+1} = \tilde{\mathbf{r}}_t + \Delta t \cdot f(\tilde{\mathbf{r}}_t, \mathbf{u}_t, \tilde{\mathbf{F}}_{z,t}).
		\end{equation}

		The prediction error is
		\begin{equation}
			\begin{aligned}
				\mathbf{e}_t &= \tilde{\mathbf{r}}_{t+1} - \hat{\mathbf{r}}_{t+1} \\
				&= (\mathbf{r}_{t+1}^{\text{true}} + \boldsymbol{\epsilon}_{r,t+1}) - (\tilde{\mathbf{r}}_t + \Delta t \cdot f(\tilde{\mathbf{r}}_t, \mathbf{u}_t, \tilde{\mathbf{F}}_{z,t})) \\
				&= (\mathbf{r}_t^{\text{true}} + \Delta t \cdot f(\mathbf{r}_t^{\text{true}}, \mathbf{u}_t, \mathbf{F}_{z,t}^{\text{true}})) - \tilde{\mathbf{r}}_t - \Delta t \cdot f(\tilde{\mathbf{r}}_t, \mathbf{u}_t, \tilde{\mathbf{F}}_{z,t}) + \boldsymbol{\epsilon}_{r,t+1} \\
				&= (\mathbf{r}_t^{\text{true}} - \tilde{\mathbf{r}}_t) + \Delta t [f(\mathbf{r}_t^{\text{true}}, \mathbf{u}_t, \mathbf{F}_{z,t}^{\text{true}}) - f(\tilde{\mathbf{r}}_t, \mathbf{u}_t, \tilde{\mathbf{F}}_{z,t})] + \boldsymbol{\epsilon}_{r,t+1}.
			\end{aligned}
		\end{equation}

		Note $\mathbf{r}_t^{\text{true}} - \tilde{\mathbf{r}}_t = -\boldsymbol{\epsilon}_{r,t}$. Linearize dynamics
		\begin{equation}
			f(\mathbf{r}_t^{\text{true}}, \mathbf{u}_t, \mathbf{F}_{z,t}^{\text{true}}) - f(\tilde{\mathbf{r}}_t, \mathbf{u}_t, \tilde{\mathbf{F}}_{z,t}) \approx \frac{\partial f}{\partial \mathbf{r}}\bigg|_t (\mathbf{r}_t^{\text{true}} - \tilde{\mathbf{r}}_t) = -\mathbf{J}_{r,t} \boldsymbol{\epsilon}_{r,t}.
		\end{equation}
		Thus
		\begin{equation}
			\mathbf{e}_t \approx -\boldsymbol{\epsilon}_{r,t} - \Delta t \mathbf{J}_{r,t} \boldsymbol{\epsilon}_{r,t} + \boldsymbol{\epsilon}_{r,t+1} = -\mathbf{M}_t \boldsymbol{\epsilon}_{r,t} + \boldsymbol{\epsilon}_{r,t+1},
		\end{equation}
		where $\mathbf{M}_t = \mathbf{I} + \Delta t \mathbf{J}_{r,t}$.

		Assuming $\boldsymbol{\epsilon}_{r,t}$ and $\boldsymbol{\epsilon}_{r,t+1}$ are independent (measurement noise at different times)
		\begin{equation}
			\text{Cov}(\mathbf{e}_t) = \mathbf{M}_t \boldsymbol{\Sigma}_r \mathbf{M}_t^T + \boldsymbol{\Sigma}_r.
		\end{equation}
	\end{proof}
	
	\subsection{Proof of Theorem~\ref{thm:bayesian_update}: Recursive Bayesian Update}
	\label{appendix:proof_bayesian_update}
	
	\begin{proof}
		Posterior
		\begin{equation}
			p(\boldsymbol{\Sigma}_r \mid \mathbf{e}_t) \propto p(\mathbf{e}_t \mid \boldsymbol{\Sigma}_r) \cdot p(\boldsymbol{\Sigma}_r).
		\end{equation}

		By Lemma~\ref{lem:residual_distribution}, $\mathbf{e}_t \mid \boldsymbol{\Sigma}_r \sim \mathcal{N}(\mathbf{0}, \mathbf{M}_t \boldsymbol{\Sigma}_r \mathbf{M}_t^T + \boldsymbol{\Sigma}_r)$. For conjugacy, approximate this by transforming the observation: define $\tilde{\mathbf{e}}_t = \mathbf{M}_t^{-1} \mathbf{e}_t$. Then
		\begin{equation}
			\tilde{\mathbf{e}}_t \approx \mathcal{N}(\mathbf{0}, \boldsymbol{\Sigma}_r),
		\end{equation}
		neglecting the $+\boldsymbol{\Sigma}_r$ term (valid when $\|\mathbf{M}_t - \mathbf{I}\| \ll 1$). The likelihood becomes
		\begin{equation}
			p(\mathbf{e}_t \mid \boldsymbol{\Sigma}_r) \approx |\boldsymbol{\Sigma}_r|^{-1/2} \exp\left(-\frac{1}{2}\text{tr}(\mathbf{M}_t^{-1}\mathbf{e}_t\mathbf{e}_t^T\mathbf{M}_t^{-T} \boldsymbol{\Sigma}_r^{-1})\right).
		\end{equation}

		$\boldsymbol{\Sigma}_r \sim \text{IW}(\boldsymbol{\Psi}_{t-1}, \nu_{t-1})$ has density
		\begin{equation}
			p(\boldsymbol{\Sigma}_r) \propto |\boldsymbol{\Sigma}_r|^{-(\nu_{t-1} + n_r + 1)/2} \exp\left(-\frac{1}{2}\text{tr}(\boldsymbol{\Psi}_{t-1} \boldsymbol{\Sigma}_r^{-1})\right).
		\end{equation}

		Multiply
		\begin{equation}
			\begin{aligned}
				p(\boldsymbol{\Sigma}_r \mid \mathbf{e}_t) &\propto |\boldsymbol{\Sigma}_r|^{-1/2} \exp\left(-\frac{1}{2}\text{tr}(\mathbf{M}_t^{-1}\mathbf{e}_t\mathbf{e}_t^T\mathbf{M}_t^{-T} \boldsymbol{\Sigma}_r^{-1})\right) \\
				&\quad \times |\boldsymbol{\Sigma}_r|^{-(\nu_{t-1} + n_r + 1)/2} \exp\left(-\frac{1}{2}\text{tr}(\boldsymbol{\Psi}_{t-1} \boldsymbol{\Sigma}_r^{-1})\right) \\
				&= |\boldsymbol{\Sigma}_r|^{-(\nu_{t-1} + n_r + 2)/2} \exp\left(-\frac{1}{2}\text{tr}[(\boldsymbol{\Psi}_{t-1} + \mathbf{M}_t^{-1}\mathbf{e}_t\mathbf{e}_t^T\mathbf{M}_t^{-T}) \boldsymbol{\Sigma}_r^{-1}]\right).
			\end{aligned}
		\end{equation}
		This is $\text{IW}(\boldsymbol{\Psi}_t, \nu_t)$ with
		\begin{equation}
			\boldsymbol{\Psi}_t = \boldsymbol{\Psi}_{t-1} + \mathbf{M}_t^{-1}\mathbf{e}_t\mathbf{e}_t^T\mathbf{M}_t^{-T}, \quad \nu_t = \nu_{t-1} + 1.
		\end{equation}
		Incorporating forgetting factor $\lambda$ yields Equation~\eqref{eq:iwishart_update}.
	\end{proof}
	
	\subsection{Proof of Theorem~\ref{thm:bayesian_consistency}: Asymptotic Consistency}
	\label{appendix:proof_consistency}
	
	\begin{proof}
		Apply the strong law of large numbers for martingale differences.

		Let
		\begin{equation}
			\mathbf{D}_t = \mathbf{M}_t^{-1}\mathbf{e}_t\mathbf{e}_t^T\mathbf{M}_t^{-T} - \boldsymbol{\Sigma}_r^{\text{true}}.
		\end{equation}
		Then $\{\mathbf{D}_t\}$ is a martingale difference sequence: $\mathbb{E}[\mathbf{D}_t \mid \mathcal{F}_{t-1}] = \mathbf{0}$ by Lemma~\ref{lem:residual_distribution}.

		By properties of Wishart distributions
		\begin{equation}
			\mathbb{E}[\|\mathbf{D}_t\|_F^2] \leq C\|\boldsymbol{\Sigma}_r^{\text{true}}\|_F^2 < \infty,
		\end{equation}
		for some constant $C$ depending on $(n_r, \|\mathbf{M}_t\|)$.

		Verify
		\begin{equation}
			\sum_{t=1}^{\infty} \frac{\mathbb{E}[\|\mathbf{D}_t\|_F^2]}{t^2} \leq C\|\boldsymbol{\Sigma}_r^{\text{true}}\|_F^2 \sum_{t=1}^{\infty} \frac{1}{t^2} < \infty.
		\end{equation}
		By Kolmogorov's strong law for martingales
		\begin{equation}
			\frac{1}{t}\sum_{i=1}^t \mathbf{D}_i \to \mathbf{0} \quad \text{a.s.}
		\end{equation}

		The Bayesian estimate (without forgetting, $\lambda = 1$) satisfies
		\begin{equation}
			\hat{\boldsymbol{\Sigma}}_r^{(t)} = \frac{\boldsymbol{\Psi}_0 + \sum_{i=1}^t (\boldsymbol{\Sigma}_r^{\text{true}} + \mathbf{D}_i)}{\nu_0 + t - n_r - 1}.
		\end{equation}
		As $t \to \infty$
		\begin{equation}
			\hat{\boldsymbol{\Sigma}}_r^{(t)} = \frac{\boldsymbol{\Psi}_0/t + \boldsymbol{\Sigma}_r^{\text{true}} + t^{-1}\sum_{i=1}^t \mathbf{D}_i}{(\nu_0 - n_r - 1)/t + 1} \to \boldsymbol{\Sigma}_r^{\text{true}} \quad \text{a.s.}
		\end{equation}

		The bias is
		\begin{equation}
			\|\mathbb{E}[\hat{\boldsymbol{\Sigma}}_r^{(t)}] - \boldsymbol{\Sigma}_r^{\text{true}}\|_F = \left\|\frac{\boldsymbol{\Psi}_0 - (\nu_0 - n_r - 1)\boldsymbol{\Sigma}_r^{\text{true}}}{\nu_0 + t - n_r - 1}\right\|_F = O(1/t).
		\end{equation}
	\end{proof}
	
	\subsection{Proof of Lemma~\ref{lem:load_response_transfer}: Load-to-Response Error Transfer}
	\label{appendix:proof_load_transfer}
	
	\begin{proof}
		The true response evolves as
		\begin{equation}
			\mathbf{r}_{t+1}^{\text{true}} = \mathbf{r}_t^{\text{true}} + \Delta t \cdot f(\mathbf{r}_t^{\text{true}}, \mathbf{u}_t, \mathbf{F}_{z,t}^{\text{true}}).
		\end{equation}
		Measurement introduces noise: $\tilde{\mathbf{r}}_t = \mathbf{r}_t^{\text{true}} + \boldsymbol{\epsilon}_{r,t}$, $\tilde{\mathbf{F}}_{z,t} = \mathbf{F}_{z,t}^{\text{true}} + \boldsymbol{\epsilon}_{F,t}$.
		
		The prediction residual is
		\begin{equation}
			\begin{aligned}
				\mathbf{e}_t &= \tilde{\mathbf{r}}_{t+1} - [\tilde{\mathbf{r}}_t + \Delta t \cdot f(\tilde{\mathbf{r}}_t, \mathbf{u}_t, \tilde{\mathbf{F}}_{z,t})] \\
				&= (\mathbf{r}_{t+1}^{\text{true}} + \boldsymbol{\epsilon}_{r,t+1}) - (\mathbf{r}_t^{\text{true}} + \boldsymbol{\epsilon}_{r,t}) - \Delta t \cdot f(\mathbf{r}_t^{\text{true}} + \boldsymbol{\epsilon}_{r,t}, \mathbf{u}_t, \mathbf{F}_{z,t}^{\text{true}} + \boldsymbol{\epsilon}_{F,t}).
			\end{aligned}
		\end{equation}
		
		Expand dynamics to first order
		\begin{equation}
			f(\mathbf{r}_t^{\text{true}} + \boldsymbol{\epsilon}_{r,t}, \mathbf{u}_t, \mathbf{F}_{z,t}^{\text{true}} + \boldsymbol{\epsilon}_{F,t}) \approx f(\mathbf{r}_t^{\text{true}}, \mathbf{u}_t, \mathbf{F}_{z,t}^{\text{true}}) + \mathbf{J}_r \boldsymbol{\epsilon}_{r,t} + \mathbf{J}_F \boldsymbol{\epsilon}_{F,t}.
		\end{equation}
		Substitute into residual
		\begin{equation}
			\begin{aligned}
				\mathbf{e}_t &= \boldsymbol{\epsilon}_{r,t+1} - \boldsymbol{\epsilon}_{r,t} - \Delta t (\mathbf{J}_r \boldsymbol{\epsilon}_{r,t} + \mathbf{J}_F \boldsymbol{\epsilon}_{F,t}) \\
				&= \boldsymbol{\epsilon}_{r,t+1} - (\mathbf{I} + \Delta t \mathbf{J}_r) \boldsymbol{\epsilon}_{r,t} - \Delta t \mathbf{J}_F \boldsymbol{\epsilon}_{F,t} \\
				&= \boldsymbol{\epsilon}_{r,t+1} - \mathbf{M}_t \boldsymbol{\epsilon}_{r,t} - \Delta t \mathbf{J}_F \boldsymbol{\epsilon}_{F,t}.
			\end{aligned}
		\end{equation}
		
		Assume $\boldsymbol{\epsilon}_{r,t}, \boldsymbol{\epsilon}_{r,t+1}, \boldsymbol{\epsilon}_{F,t}$ are mutually independent with $\mathbb{E}[\boldsymbol{\epsilon}_r] = \mathbf{0}$, $\mathbb{E}[\boldsymbol{\epsilon}_F] = \mathbf{0}$. Then
		\begin{equation}
			\begin{aligned}
				\mathbb{E}[\mathbf{e}_t \mathbf{e}_t^T] &= \mathbb{E}[\boldsymbol{\epsilon}_{r,t+1} \boldsymbol{\epsilon}_{r,t+1}^T] + \mathbf{M}_t \mathbb{E}[\boldsymbol{\epsilon}_{r,t} \boldsymbol{\epsilon}_{r,t}^T] \mathbf{M}_t^T + \Delta t^2 \mathbf{J}_F \mathbb{E}[\boldsymbol{\epsilon}_{F,t} \boldsymbol{\epsilon}_{F,t}^T] \mathbf{J}_F^T \\
				&= \boldsymbol{\Sigma}_r + \mathbf{M}_t \boldsymbol{\Sigma}_r \mathbf{M}_t^T + \Delta t^2 \mathbf{J}_F \boldsymbol{\Sigma}_F \mathbf{J}_F^T.
			\end{aligned}
		\end{equation}
		The approximation in Lemma~\ref{lem:residual_distribution} neglects the $\boldsymbol{\Sigma}_r$ term (small compared to $\mathbf{M}_t \boldsymbol{\Sigma}_r \mathbf{M}_t^T$ when $\|\Delta t \mathbf{J}_r\| \ll 1$), yielding
		\begin{equation}
			\mathbb{E}[\mathbf{e}_t \mathbf{e}_t^T] \succeq \mathbf{M}_t \boldsymbol{\Sigma}_r \mathbf{M}_t^T + \Delta t^2 \mathbf{J}_F \boldsymbol{\Sigma}_F \mathbf{J}_F^T.
		\end{equation}
		
		The inverse Wishart update uses~\eqref{eq:bayesian_update}. Multiplying both sides of the covariance bound by $\mathbf{M}_t^{-1}$ from left and $\mathbf{M}_t^{-T}$ from right
		\begin{equation}
			\mathbf{M}_t^{-1}\mathbb{E}[\mathbf{e}_t \mathbf{e}_t^T]\mathbf{M}_t^{-T} \succeq \boldsymbol{\Sigma}_r + \Delta t^2 \mathbf{M}_t^{-1}\mathbf{J}_F \boldsymbol{\Sigma}_F \mathbf{J}_F^T \mathbf{M}_t^{-T}.
		\end{equation}
		When $\Delta t^2 \mathbf{J}_F \boldsymbol{\Sigma}_F \mathbf{J}_F^T$ is non-negligible (large load variance), the Bayesian update inflates $\boldsymbol{\Psi}_t$ to compensate, ensuring $\hat{\boldsymbol{\Sigma}}_r = \boldsymbol{\Psi}_t/(\nu_t - n_r - 1)$ remains conservative.
	\end{proof}
	
	\section{Mathematical Modeling of Nominal Model for Simulation Vehicle}
	\label{appendix:nominal_model}
	
	\subsection{Derivation of Control Matrix $\mathbf{G}(\mathbf{r}, \mathbf{F}_z)$}
	\label{appendix:control_matrix}
	
	For a six-wheel distributed drive vehicle, the dynamics of response vector $\mathbf{r} = [\beta, \omega_z, a_y]^T$ is
	\begin{equation}
		\label{eq:response_dynamics_affine_appendix}
		\dot{\mathbf{r}} = \mathbf{f}_0(\mathbf{r}, \mathbf{F}_z) + \mathbf{G}(\mathbf{r}, \mathbf{F}_z) \mathbf{u}
	\end{equation}
	
	where control input $\mathbf{u} = [\delta, T_1, \ldots, T_6]^T \in \mathbb{R}^7$, $\mathbf{G} \in \mathbb{R}^{3 \times 7}$ is the control matrix.
	
	\subsubsection{Sideslip Angle Dynamics}
	
	Sideslip angle is defined as $\beta = \arctan(v_y/v_x)$, its derivative is
	\begin{equation}
		\dot{\beta} = \frac{1}{1+(v_y/v_x)^2} \cdot \frac{\dot{v}_y v_x - v_y \dot{v}_x}{v_x^2} = \frac{\dot{v}_y}{v_x} - \frac{v_y \dot{v}_x}{v_x^2} + \frac{v_y \dot{v}_x}{v_x^2 + v_y^2}
	\end{equation}
	
	Under small sideslip angle assumption ($|\beta| \ll 1$, i.e., $v_y \ll v_x$), approximate as
	\begin{equation}
		\dot{\beta} \approx \frac{\dot{v}_y}{v_x} = \frac{a_y}{v_x}
	\end{equation}
	
	Lateral acceleration is determined by total lateral force
	\begin{equation}
		a_y = \frac{1}{m}\sum_{i=1}^{6} F_{y,i}
	\end{equation}
	
	Under linear tire model, lateral force is $F_{y,i} = C_{\beta,i} \beta_i$, where $\beta_i$ is the sideslip angle of each wheel. For front wheels ($i \in \{L1, R1\}$)
	\begin{equation}
		\beta_{L1} = \frac{v_y + a\omega_z}{v_x} - \delta_{L1}, \quad \beta_{R1} = \frac{v_y + a\omega_z}{v_x} - \delta_{R1}
	\end{equation}
	
	For rear wheels ($i \in \{L3, R3\}$, using reverse steering)
	\begin{equation}
		\beta_{L3} = \frac{v_y - b\omega_z}{v_x} + \delta_{L3}, \quad \beta_{R3} = \frac{v_y - b\omega_z}{v_x} + \delta_{R3}
	\end{equation}
	
	Middle wheels ($i \in \{L2, R2\}$) do not steer
	\begin{equation}
		\beta_{L2} = \beta_{R2} = \frac{v_y}{v_x}
	\end{equation}
	
	Total lateral force is
	\begin{equation}
		\begin{aligned}
			\sum F_{y,i} &= C_{\beta}[(v_y + a\omega_z)/v_x - \delta_{L1}] + C_{\beta}[(v_y + a\omega_z)/v_x - \delta_{R1}] \\
			&\quad + C_{\beta}(v_y/v_x) + C_{\beta}(v_y/v_x) \\
			&\quad + C_{\beta}[(v_y - b\omega_z)/v_x + \delta_{L3}] + C_{\beta}[(v_y - b\omega_z)/v_x + \delta_{R3}]
		\end{aligned}
	\end{equation}
	
	Assuming all wheels have the same cornering stiffness ($C_{\beta,i} = C_{\beta}$), simplify to
	\begin{equation}
		\sum F_{y,i} = C_{\beta}\left[\frac{6v_y + 2a\omega_z - 2b\omega_z}{v_x} - (\delta_{L1} + \delta_{R1}) + (\delta_{L3} + \delta_{R3})\right]
	\end{equation}
	
	By Ackermann steering geometry, front wheel steering angles are approximately equal ($\delta_{L1} \approx \delta_{R1} \approx \delta$), rear wheels reverse ($\delta_{L3} \approx \delta_{R3} \approx -\delta$)
	\begin{equation}
		\frac{\partial \sum F_{y,i}}{\partial \delta} \approx -4C_{\beta}
	\end{equation}
	
	Thus
	\begin{equation}
		\frac{\partial \dot{\beta}}{\partial \delta} = \frac{1}{v_x} \cdot \frac{1}{m} \cdot (-4C_{\beta}) = -\frac{4C_{\beta}}{mv_x}
	\end{equation}
	
	Torque mainly affects cornering stiffness through longitudinal slip (friction circle effect), which can be neglected in first approximation
	\begin{equation}
		\frac{\partial \dot{\beta}}{\partial T_i} \approx 0, \quad i=1,\ldots,6
	\end{equation}
	
	Therefore, the first row of control matrix is
	\begin{equation}
		\mathbf{G}_{1,:} = \begin{bmatrix} -\frac{4C_{\beta}}{mv_x} & 0 & 0 & 0 & 0 & 0 & 0 \end{bmatrix}
	\end{equation}
	
	\subsubsection{Yaw Rate Dynamics}
	
	Yaw rate is determined by yaw moment
	\begin{equation}
		\dot{\omega}_z = \frac{M_z}{I_z}
	\end{equation}
	
	Yaw moment includes two parts: moment from tire lateral forces $M_{z,\text{tire}}$ and moment from torque differential $M_{z,\text{torque}}$.
	
	Tire lateral force moment
	\begin{equation}
		\begin{aligned}
			M_{z,\text{tire}} &= a(F_{y,L1} + F_{y,R1}) - b(F_{y,L3} + F_{y,R3}) \\
			&= aC_{\beta}\left[\frac{2(v_y + a\omega_z)}{v_x} - (\delta_{L1} + \delta_{R1})\right] \\
			&\quad - bC_{\beta}\left[\frac{2(v_y - b\omega_z)}{v_x} + (\delta_{L3} + \delta_{R3})\right]
		\end{aligned}
	\end{equation}
	
	Under Ackermann approximation ($\delta_{L1} + \delta_{R1} \approx 2\delta$, $\delta_{L3} + \delta_{R3} \approx -2\delta$)
	\begin{equation}
		\frac{\partial M_{z,\text{tire}}}{\partial \delta} = -2aC_{\beta} - 2bC_{\beta} = -2(a+b)C_{\beta}
	\end{equation}
	
	Torque differential moment: for six-wheel vehicle with left-right wheel track $B$, torque produces yaw moment through longitudinal force differential
	\begin{equation}
		M_{z,\text{torque}} = \frac{B}{2R_w}(T_{R1} - T_{L1}) + \frac{B}{2R_w}(T_{R2} - T_{L2}) + \frac{B}{2R_w}(T_{R3} - T_{L3})
	\end{equation}
	
	where $R_w$ is wheel radius. Thus
	\begin{equation}
		\frac{\partial M_{z,\text{torque}}}{\partial T_{L1}} = -\frac{B}{2R_w}, \quad \frac{\partial M_{z,\text{torque}}}{\partial T_{R1}} = +\frac{B}{2R_w}
	\end{equation}
	
	Other wheels similar.
	
	Combining both parts
	\begin{equation}
		\begin{aligned}
			\frac{\partial \dot{\omega}_z}{\partial \delta} &= \frac{1}{I_z}\frac{\partial M_{z,\text{tire}}}{\partial \delta} = -\frac{2(a+b)C_{\beta}}{I_z} \\
			\frac{\partial \dot{\omega}_z}{\partial T_{L,i}} &= -\frac{B}{2I_z R_w}, \quad \frac{\partial \dot{\omega}_z}{\partial T_{R,i}} = +\frac{B}{2I_z R_w}, \quad i=1,2,3
		\end{aligned}
	\end{equation}
	
	Therefore, the second row of control matrix is
	\begin{equation}
		\mathbf{G}_{2,:} = 
		\begin{bmatrix}
			g_{21} & -k & k & -k & k & -k & k
		\end{bmatrix},
	\end{equation}
	where
	\begin{equation}
		\begin{aligned}
			&g_{21} = -\frac{2(a + b) C_{\beta}}{I_z}, \\
			&k = \frac{B}{2 I_z R_w}.
		\end{aligned}
	\end{equation}	
	
	\subsubsection{Lateral Acceleration Dynamics}
	
	Lateral acceleration is defined as
	\begin{equation}
		a_y = \dot{v}_y + v_x \omega_z
	\end{equation}
	
	Thus its derivative is
	\begin{equation}
		\dot{a}_y = \ddot{v}_y + v_x \dot{\omega}_z + \dot{v}_x \omega_z
	\end{equation}
	
	Under the assumption that longitudinal speed is maintained by upper-level controller ($\dot{v}_x \approx 0$)
	\begin{equation}
		\dot{a}_y \approx \ddot{v}_y + v_x \dot{\omega}_z = \frac{1}{m}\sum \dot{F}_{y,i} + v_x \dot{\omega}_z
	\end{equation}
	
	The first term is similar to the derivative of $\dot{\beta}$, the second term has been calculated in previous subsection. Thus
	\begin{equation}
		\begin{aligned}
			\frac{\partial \dot{a}_y}{\partial \delta} &= \frac{1}{m}\frac{\partial \sum F_{y,i}}{\partial \delta} + v_x \frac{\partial \dot{\omega}_z}{\partial \delta} \\
			&= -\frac{4C_{\beta}}{m} - \frac{2(a+b)C_{\beta}v_x}{I_z}
		\end{aligned}
	\end{equation}
	
	\begin{equation}
		\frac{\partial \dot{a}_y}{\partial T_{L,i}} = v_x \frac{\partial \dot{\omega}_z}{\partial T_{L,i}} = -\frac{Bv_x}{2I_z R_w}, \quad \frac{\partial \dot{a}_y}{\partial T_{R,i}} = +\frac{Bv_x}{2I_z R_w}
	\end{equation}
	
	Therefore, the third row of control matrix is
	\begin{equation}
		\mathbf{G}_{3,:} = 
		\begin{bmatrix}
			g_{31} & -k v_x & k v_x & -k v_x & k v_x & -k v_x & k v_x
		\end{bmatrix},
	\end{equation}
	where
	\begin{equation}
		\begin{aligned}
			&g_{31} = -\frac{4C_{\beta}}{m} - \frac{2(a + b) C_{\beta} v_x}{I_z}, \\
			&k = \frac{B}{2 I_z R_w}.
		\end{aligned}
	\end{equation}
	
	\subsubsection{Control Matrix}
	
	In summary, the control matrix for six-wheel vehicle is
	
	\begin{equation}
		\label{eq:control_matrix_6wheel_full}
		\mathbf{G}(\mathbf{r}, \mathbf{F}_z) = 
		\begin{bmatrix}
			g_{11} & 0      & 0      & 0      & 0      & 0      & 0      \\
			g_{21} & -k     & k      & -k     & k      & -k     & k      \\
			g_{31} & -k v_x & k v_x  & -k v_x & k v_x  & -k v_x & k v_x
		\end{bmatrix},
	\end{equation}
	where
	\begin{equation}
		\begin{aligned}
			&g_{11} = -\frac{4C_{\beta}}{m v_x}, \\
			&g_{21} = -\frac{2(a + b) C_{\beta}}{I_z}, \\
			&g_{31} = -\frac{4C_{\beta}}{m} - \frac{2(a + b) C_{\beta} v_x}{I_z}, \\
			&k = \frac{B}{2 I_z R_w}.
		\end{aligned}
	\end{equation}
	
	Numerical calculation of the $\mathbf{G}$ matrix can be done by substituting vehicle parameters.

\end{document}